\documentclass[11pt,a4paper,reqno,obeyFinal]{article}

\usepackage{template}
\usepackage{subcaption}

\usepackage{fullpage}

\def\d{\mathrm{d}}

\newcommand{\PLP}{\mathrm{PLP}}

\newcommand{\C}{\mathcal{C}}

\newcommand{\Z}{\mathbb{Z}}
\newcommand{\LL}{\mathbb{L}}
\newcommand{\R}{\mathbb{R}}
\newcommand{\N}{\mathbb{N}}

\newcommand{\eqd}{\stackrel{d}{=}}

\newcommand{\IP}{\mathbb{P}}
\newcommand{\IE}{\mathbb{E}}

\newcommand{\clzero}{\tilde{\mathsf{C}}_0}

  \newcounter{iconst}
  \newcommand{\newiconst}[1]{\refstepcounter{iconst}\label{#1}}
  \newcommand{\useiconst}[1]{c_{\textnormal{\tiny \ref{#1}}}}
\begin{document}

\title{Cylinders' percolation: decoupling and applications}

\date{\today}
\author{Caio Alves \thanks{Alfréd Rényi Institute of Mathematics, Budapest, 1053 Hungary.}
  \and
  Augusto Teixeira \thanks{IMPA, Estrada Dona Castorina 110, 22460-320 Rio de Janeiro, RJ - Brazil}}

\maketitle

\begin{abstract}

  In this paper we establish a strong decoupling inequality for the cylinder's percolation process introduced by Tykesson and Windisch in \cite{TW10b}.
  This model features a very strong dependency structure, making it difficult to study, and this is why such decoupling inequalities are desirable.
  It is important to notice that the type of dependencies featured by cylinder's percolation is particularly intricate, given that the cylinders have infinite range (unlike some models like Boolean percolation) while at the same time being rigid bodies (unlike processes such as Random Interlacements).
  Our work introduces a new notion of fast decoupling, proves that it holds for the model in question and finishes with an application.
  More precisely, we prove that for a small enough density of cylinders, a random walk on a connected component of the vacant set is transient for all dimensions $d \geq 3$.

  \bigskip

\end{abstract}

%\newpage

\section{Introduction}

The Cylinder's Percolation model, introduced by Tykesson and Windisch in \cite{TW10b} by suggestion of Itai Benjamini, consists of a random cloud of cylinders in $\mathbb{R}^d$, for $d \geq 3$.
While the width of these cylinders is fixed to be one, their central axes are randomly distributed according to a Poisson Point process in the space of lines.
This Poisson process has intensity proportional to the Haar measure, which is the unique (up to multiplication constants) measure on the space of lines, which is invariant with respect to both rotations and translations of $\mathbb{R}^d$.
See Subsection~\ref{ss:definition_model} for a precise definition of the model and Figure~\ref{f:process_sim} for an illustration.
The intensity of the model is governed by a multiplicative constant $u$, that modulates how many cylinders are present in the picture.

\begin{figure}
  \centering
  \begin{subfigure}[b]{0.49\textwidth}
      \centering
      \includegraphics[width=\textwidth]{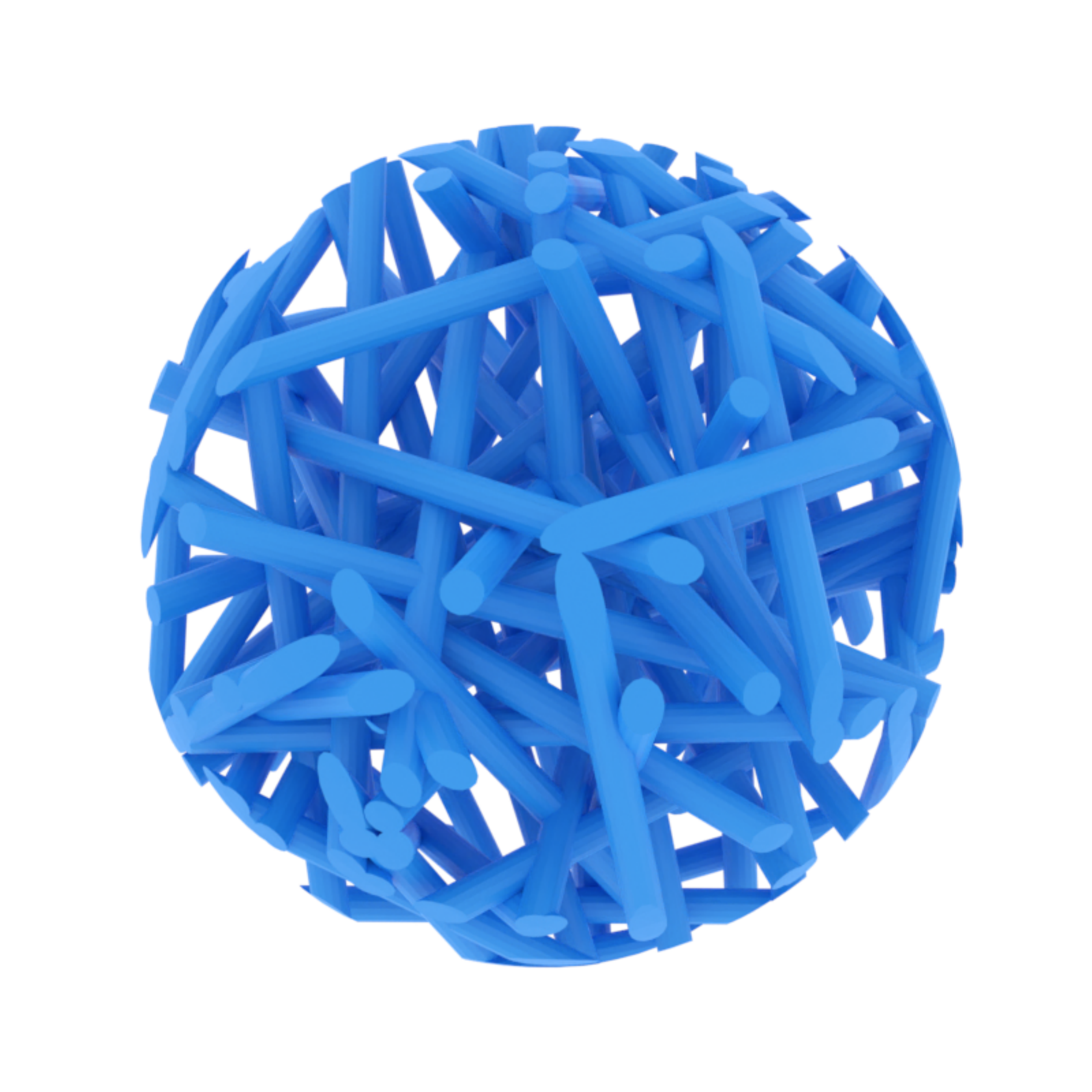}
      \label{f:process_sim_1}
  \end{subfigure}
  \begin{subfigure}[b]{0.49\textwidth}
      \centering
      \includegraphics[width=\textwidth]{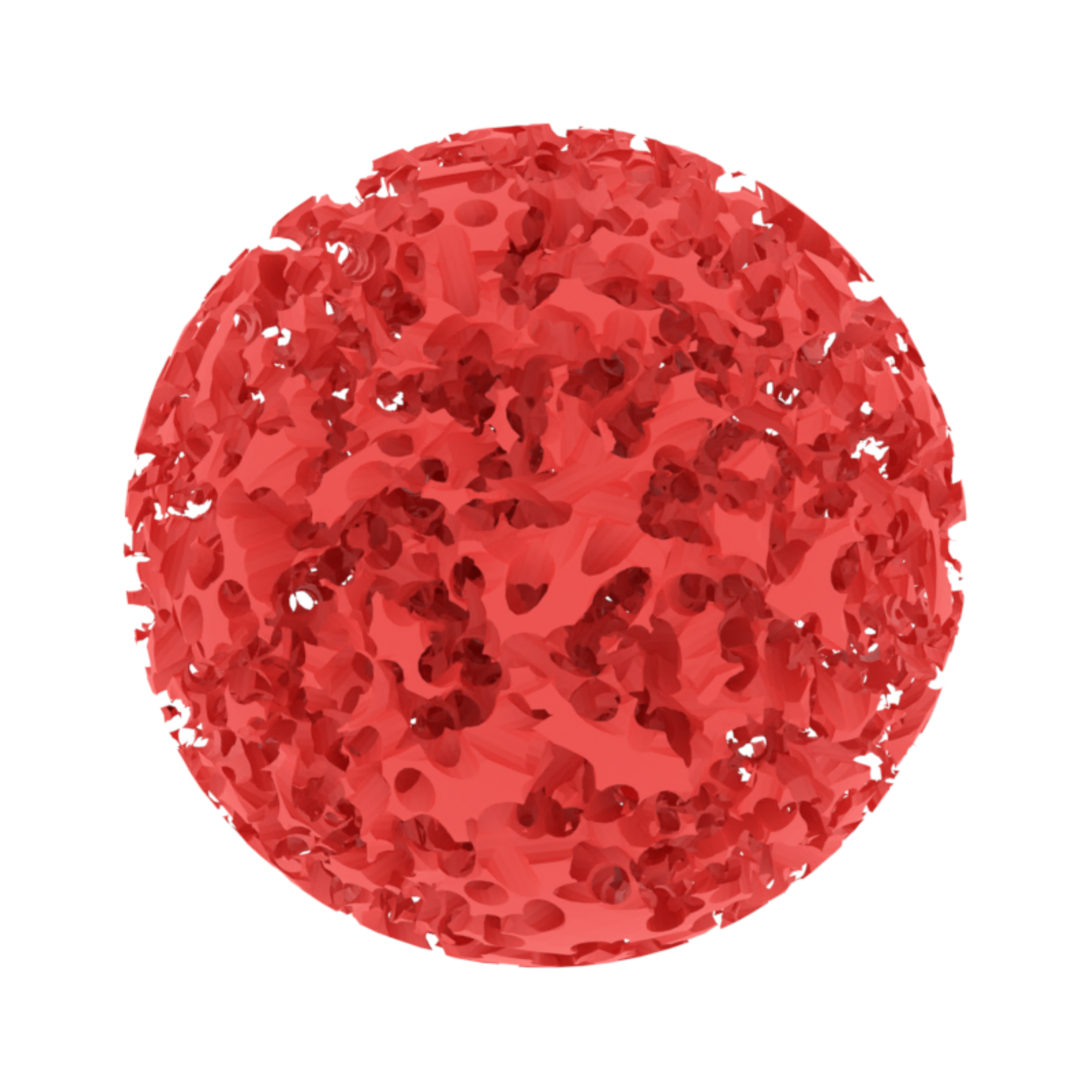}
      \label{f:process_sim_2}
  \end{subfigure}
  \caption{Simulations of the occupied ($u = 0.07$) and vacant ($u = 0.07$) sets of the Poisson cylinder process intersected with a ball of radius $24$.}
  \label{f:process_sim}
\end{figure}

In the original work \cite{TW10b}, the authors proved that the vacant set left after removing the cylinders undergoes a percolative phase transition for $d \geq 4$.
More precisely, for any dimensions $d \geq 3$ and for large enough intensity $u$, they prove that the vacant set does not percolate, while for $d \geq 4$ and $u$ small enough there is an unbounded connected component on the vacant set.
The existence of a percolative phase for the vacant set in $d = 3$ has been established in \cite{HST12}.

The above cited works make careful use of the weak decoupling inequalities that provide a polynomial decay of correlations present in the model.
More precisely, in Lemma~3.3 of \cite{TW10b}, the authors prove that for any functions $f$ and $g$ that only depend on the configuration of the cylinder set inside balls $B(x, r)$ and $B(y, r)$ respectively, we have
\begin{equation}
  \label{e:decouple_slow}
  \big| \mathbb{E}^u(f g) - \mathbb{E}^u(f) \mathbb{E}(g) \big| \leq c u \Big( \frac{(r + 1)^2}{|x - y|} \Big)^{d - 1}.
\end{equation}
The main weakness of the decoupling inequality \eqref{e:decouple_slow} is its slow decay of covariance which is related to the probability that the same cylinder hits the two balls $B_1$ and~$B_2$.

Although \cite{TW10b} and \cite{HST12} have successfully employed the above polynomial decay to establish the existence of a phase transition for the model (through detailed constructions), such weak decoupling does not allow us to prove more refined properties of the model as the ones we present in Sections~\ref{s:paths} and~\ref{s:flow}.

For other dependent percolation models such as Random Interlacements, better decorrelation bounds have been obtained that decay stretched exponentially, see \cite{zbMATH06509926}.
These bounds use a small sprinkling in the intensity of the process in order to blur (and effectively dominate) the dependence induced by objects that touch both balls.
In this article we employ a similar technique, but due to the rigidity of cylinders, we need to employ sprinklings both on the density of cylinders and in their radii, so that we are able to prove a decay that is faster than polynomial, see Theorem~\ref{thm:2boxdec_intro} below.

\bigskip

In order to state precisely our result, we have to introduce a notation for the cylinder set at intensity $u$ and radius $\rho$.
As mentioned earlier, the cylinder's process is governed by a Poisson Point Process on the space of lines in $\mathbb{R}^d$.
This process has an intensity $u \geq 0$ and can be written as $\omega = \sum_{i \geq 0} \delta_{l_i}$ where $l_i$ are lines in $\mathbb{R}^d$, see \eqref{e:mu} for more details.
Given this point measure, we define the cylinder's set with radius $\rho$ as
\begin{equation}
  \mathcal{C}^\rho_u = \mathcal{C}^\rho_u(\omega) = \bigcup_{i \geq 0} B(l_i, \rho),
\end{equation}
where $B(A, r)$ stands for the set of points within distance at most $\rho$ of the line $l_i$.

Given two balls $B_1(x_1, L)$ and $B_2(x_2, L)$, our main Theorem~\ref{thm:2boxdec_intro} below can be understood as controlling the dependence between what happens with the cylinder process at $B_1$ and $B_2$.
For this we will make a sprinkling in the intensity of the cylinders ($u \to u + \delta$) and on their radii ($\rho \to \rho + \eps$).

\begin{theorem}
  \label{thm:2boxdec_intro}
  There exists a constant~$\useiconst{c:2boxdec}>0$ depending only on the dimension~$d$ such that, for any~$\eps, \delta, \alpha \in (0,1)$ and $\rho\in[1,4]$, and any pair of increasing functions
  \begin{equation}
    f_i: \Omega \to [0,1], \text{ measurable with respect to $\sigma(\mathcal{C}^\rho_u \cap B_i)$}, \text{ for~$i = 1, 2$},
  \end{equation}
  if $|x_1 - x_2| \geq L^{2 + \alpha}/\eps$ we have
  \begin{equation}
    \label{e:2boxdec}
      \IE \big[ f_1 \big( \mathcal{C}^u_\rho  \big) f_2 \big( \mathcal{C}^u_\rho  \big) \big]
      \leq \;
       \IE \big[ f_1 \big( \mathcal{C}^u_{\rho + \eps}  \big) \big]
      \IE \big[ f_2 \big( \mathcal{C}^{u + \delta}_{\rho + \eps}  \big) \big]
       + c_2^{-1} \exp \big\{ -\useiconst{c:2boxdec} \delta \eps^{d-1} L^{\alpha(d-1)} \big\}.
  \end{equation}
  An analogous result for non-increasing functions also holds, see Theorem~\ref{thm:2boxdec}.
\end{theorem}

\begin{remark}
  This is a good point to make a few observations.
  \begin{enumerate}[\qquad a)]
  \item The upper bound presented in Theorem~\ref{thm:2boxdec_intro} is the most relevant part of the decoupling, since the corresponding lower bound holds trivially without any error due to the FKG inequality.
  \item Note that unlike \eqref{e:decouple_slow}, we have a control over the dependencies that decays as a stretched exponential, instead of as a polynomial.
  \item Inequalities that are very similar to the one presented in Theorem~\ref{thm:2boxdec_intro} have been previously established for models such as Random Interlacements \cite{zbMATH06509926}, Gaussian Free Field \cite{popov2015decoupling} and Random Walk Loop Soup \cite{AS18}.
    And although such results have proven themselves to be very useful in studying the underlying models \cite{RS12, DRS12, T10, duminil2020equality, drewitz2014chemical}, the techniques developed so far could not be adapted to cylinders' percolation due to the rigidity of cylinder's themselves.
  \item Note that all of the above mentioned decoupling inequalities (in \cite{zbMATH06509926,popov2015decoupling,drewitz2014chemical}) involve a sprinkling $u \to u + \delta$, similar to the one we employ in our main result.
    However, in the current article we also employ a second sprinkling (with respect to the radii of the cylinders from $\rho \to \rho + \eps$) which is crucial to deal with the rigidity of these objects.
  \item As an indication of how heavy the dependencies induced by the Poisson Cylinder's model are, it is instructive to observe the effect of conditioning the process on its trace inside a box.
  In this case, one would be able to extrapolate indefinitely the cylinders that touch the box, effectively obtaining an infinite-range information about the process on the remainder of $\mathbb{R}^d$.
  \end{enumerate}
\end{remark}

\bigskip

Although the above remark mentions that decoupling inequalities have proved themselves useful in the study of other models, we felt that presenting Theorem~\ref{thm:2boxdec_intro} without any applications would feel too abstract for the readers.
For this reason we have decided to include one interesting application of Theorem~\ref{thm:2boxdec_intro} to the study of a random walk on the vacant set left by this soup of cylinders.

We denote the vacant set left by random cylinders by $\mathcal{V}^\rho_u = \mathbb{R}^d \setminus \mathcal{C}^\rho_u$.
As mentioned previously, this set undergoes a percolation phase transition as we vary $u$, in particular for small enough values of $u > 0$, $\mathcal{V}^\rho_u$ contains almost surely an unbounded connected component.
This result has been specially difficult to establish for $d = 3$, requiring a separate article \cite{HST12} and the proof strongly relies on planarity arguments, since the infinite connected component is constructed inside a two-dimensional surface.

Given the above difficulties, we have decided to focus this article in a question that is inherently non-planar.
More precisely whether a random walk on the infinite component of $\mathcal{V}^1_u$ is transient or not.
This is the content of the following theorem.

\begin{theorem}
  \label{t:transient}
  For any $d \geq 3$, one endows the set $\mathbb{Z}^d$ with nearest neighbor edges $\mathcal{E} = \big\{ \{x, y\}; |x - y| = 1 \big\}$ and consider the random subset of edges
  \begin{equation*}
    \mathcal{E}' := \bigg\{
    \begin{aligned}
    \{x, y\} \in \mathcal{E}; &
      \text{the whole line segment connecting $x$ to $y$ }
      \\
      & \text{is contained in $\mathcal{V}^1_u$}
    \end{aligned}
    \bigg\}.
  \end{equation*}
  Then for $u$ small enough depending only on the dimension, the graph $(\mathbb{Z}^d, \mathcal{E}')$ contains a connected component that is transient for the simple random walk.
\end{theorem}

\begin{remark}
  \label{r:rw}
  Observe that the above result gives in particular the existence of an unbounded connected component of $\mathcal{V}^1_u$, as previously proved in \cite{HST12}.

  Theorem~\ref{t:transient} is stated in terms of a random walk (instead of a diffusion) to avoid technicalities involved in the construction and analysis of the Brownian Motion in the presence of potentially complex boundaries, see Remark~\ref{r:discrete}.

  It is also interesting to note that the transience of the simple random walk is an intrinsically non-planar property.
  This is the reason why we have chosen to present this result that does not rely on planarity as \cite{HST12}.
\end{remark}

\paragraph{Previous results on the model}

As mentioned earlier, the Poisson Cylinder's process was introduced in \cite{TW10b}, where a phase transition for the percolation of its vacant set was proved for all $d \geq 4$.
Later in \cite{HST12} the phase transition for $d = 3$ was established in a slab.
Since then, the model has been studied and extended in various directions.

The connectivity of the occupied set was proved in \cite{10.1214/14-AIHP641}, while a shape theorem was obtained in \cite{hilario2019shape}.
Cylinder models have been constructed in the hyperbolic space \cite{10.1214/EJP.v20-3645} and with axes that are parallel to the Euclidean basis \cite{hilario2011coordinate,hilario2019bernoulli}.
A fractal version of the cylinder's percolation model was presented in \cite{broman2020fractal}.
Also the intersection of cylinder's percolation with a plane gives rise to a random collection of stretched ellipses, whose more in depth exploration was done in \cite{teixeira2017ellipses}.

\paragraph{Overview of the proofs}

Let us now give a brief description of the proof of Theorem~\ref{thm:2boxdec_intro}.
Observe first that one can focus on the cylinders that intersect both boxes $B_1$ and $B_2$, since these are the cylinders that can carry dependence between them.

Roughly speaking, we will ``perturb'' each such cylinder $B(l_i, \rho)$, by first making them slightly thicker $B(l_i, \rho + \epsilon)$, as in the statement of Theorem~\ref{thm:2boxdec_intro}.
The most important observation at this point is that this thickening allows us to change slightly the original cylinder's direction (say from $B(l_i, \rho)$ to $B(l'_i, \rho)$), while still guaranteeing that $B(l_i', \rho) \cap B_1 \subseteq B(l_i, \rho + \epsilon)$.

This directional perturbation (together with the fact that the two boxes are well separated) is sufficient to make sure that the landing point of $B(l'_i, \rho)$ in $B_2$ is very delocalized.
Therefore, the process of ``perturbed'' cylinders viewed from $B_2$ looks indistinguishable from an independent cloud of random cylinders.
At this point we sprinkle the intensity $u$ of the process in order to dominate this cloud in $B_2$, finishing the proof of Theorem~\ref{thm:2boxdec_intro}.

The proof of Theorem~\ref{t:transient} follows a classical argument by Thompson, that provides a systematic way to prove transience of a simple random walk on a graph by building a finite energy flow from the origin to infinity.
The construction of this flow follows a multi-scale argument, since this technique is very well suited to the decoupling inequalities that we established before.

It is important to notice that Theorem~\ref{thm:2boxdec_intro} is not strong enough to be used in the renormalization schemes that we employ in our applications.
Therefore we prove a slightly strengthened version of Theorem~\ref{thm:2boxdec_intro} in  Section~\ref{s:3box}, see Theorem~\ref{thm:3boxdec}.

\paragraph{Open problems}

We believe that several questions for percolation of cylinders have been left unanswered because of a lack of a fast decoupling inequality like the one presented in Theorem~\ref{thm:2boxdec_intro}.

For this reason we list here some of the directions for which research in this model may now advance in the form of a list of open questions, all concerning the phase $u > 0$ small enough:
\begin{enumerate}[\qquad a)]
\item Is there a unique unbounded component for the vacant set left by cylinders?
\item Can we control on the radius of $C_0$ (the cluster of $\mathcal{V}^1_u$ containing the origin)? More precisely, can one prove a decay for $\mathbb{P}[C_0 \not \subset B(0, r), C_0 \text{ bounded}]$?
\item Can one establish quantitative bounds for the time constant of the first passage percolation on $\mathcal{V}^1_u$?
\item Does a Functional Central Limit Theorem hold for the Brownian motion on~$\mathcal{V}^1_u$?
\item Is it true that the phase transition for percolation on the Poisson Cylinder's model is sharp? This has been established for strongly dependent percolation models such as level sets of the Gaussian Free Field \cite{duminil2020equality} and Random Interlacements \cite{ri_sharp}.
\end{enumerate}
Although all of the above problems require new ideas and techniques to be solved, we believe that the present work will make these questions more approachable and appealing for future works.

\paragraph{Organization of the paper}

This paper is organized as follows.
In Section~\ref{s:preliminaries}, we introduce the basic notation and the definition of the Poisson Cylinder's model, finishing with proofs for some of its basic properties.
Our main decoupling inequality Theorem~\ref{thm:2boxdec_intro} is re-stated and proved it Section~\ref{s:2box}.
Section~\ref{s:3box} is dedicated to extending our main theorem to three boxes, which is surprisingly necessary in order to prove our main applications.
Finally, Sections~\ref{s:renorm}, \ref{s:paths} and \ref{s:flow} respectively: presents our main renormalization scheme, constructs the paths and builds the flows that culminate in the proof of Theorem~\ref{t:transient}.

\paragraph{A word about constants}
Throughout the text, the unnumbered letter~$c$ will denote a positive constant depending only on the dimension, its value could change from line to line. Numbered letters $c_0,c_1\dots$ are also positive constants, but their values are fixed on their first use in the text.

\paragraph{Acknowledgments}
During this research, AT has been supported by grants ``Projeto Universal'' (406250/2016-2) and ``Produtividade em Pesquisa'' (304437/2018-2) from CNPq and ``Jovem Cientista do Nosso Estado'', (202.716/2018) from FAPERJ. CA was supported by the FAPESP grant 2013/24928, the Noise-Sensitivity Everywhere ERC Consolidator Grant 772466, and the DFG Grant SA 3465/1-1.

\section{Preliminaries}
\label{s:preliminaries}

We begin this section with the basic notation that will be used throughout this paper.
We write $\mathbb{N}$ for the set $\{0, 1, 2, \dots\}$. Let $d\geq 3$ be a fixed integer. We let $|\cdot|$ denote the Euclidean norm on $\R^d$. Given $r>0$ and $x\in\R^d$, we define~$B(x,r)$ as the closed Euclidean ball of radius~$r$ centered at~$x$ and~$B_\infty(x,r)$ as the closed ball in the $l_\infty$-norm with same center and radius. Given $A,B\subset\R^d$ we define
\begin{equation*}
\dist(A,B):=\inf\{|x-y|:x\in A,y\in B\},
\end{equation*}
the Euclidean distance between~$A$ and~$B$, and
\begin{equation*}
B(A,r):=\bigcup_{x\in A} B(x,r),
\end{equation*}
the set of all points with distance at most~$r$ from~$A$.

\subsection{The Poisson cylinder process}
\label{ss:definition_model}

Regarding~$\R^k$, for some $k\geq 0$, we denote its canonical basis by~${\bf e}_1,\dots,{\bf e}_k$, its typical element by $(x_1,\dots,x_k)$, its Borel $\sigma$-algebra by $\mathcal{B}(\R^k)$ and its Lebesgue measure by~$\d v_k$.
We let $\nu$ denote the unique normalized Haar measure of~$SO_d$, the topological group of rigid rotations of~$\R^d$.

Let us now give a overview of the definition of the Poisson cylinder percolation process on $\mathbb{R}^d$ (defined in Section~$2$ of~\cite{TW10b}), a more detailed description will be presented later.
Define the set~$\LL$ of lines (or affine Grassmanian of $1$-dimensional affine spaces) of~$\R^d$. We start with a Poisson point process in that plane $\R^{d-1}\times\{0\}\subsetneq\R^d$ with intensity $u\d v_{d-1}$, where~$u$ is a positive real number. Through each of the points of the process we draw a line orthogonal to the plane. We then sample an element of~$SO_d$ according to~$\nu$ independently for each line. Finally, to each line we apply its associated random rotation around the origin of~$\R^d$. The resulting random subset of~$\LL$ is stationary under translations and rotations of~$\R^d$. By considering this set of lines as a subset of~$\R^d$, and then viewing each line as the axis of a cylinder with radius~$1$, we arrive at the definition of the cylinder set.

In more rigorous terms, given $x\in\R^{d-1}$, we let
\begin{equation*}
  \tau_x: \mathbb{R}^{d-1} \to \R^{d-1} \text{ defined through } \tau_x(y) = x + y
\end{equation*}
denote the translation by~$x$ in~$\mathbb{R}^{d-1}$. We identify~$\R^{d-1}$ with~$\R^{d-1}\times\{0\}$ and consider~$SO_d$ endowed with its natural topology. We then consider the function
\begin{equation}
  \lambda: \mathbb{R}^{d-1} \times SO_d \to \mathbb{L} \text{ that takes }
(x, \Gamma) \text{ and maps to } \Gamma(\tau_x(\{t {\bf e}_d:t\in\R\})),
\end{equation}
and the finest topology on~$\LL$ that makes~$\lambda$ continuous. We construct from this topology the $\sigma$-algebra~$\mathcal{B}(\LL)$ of borelian sets of~$\LL$. We also use the pushforward~$\lambda_{*}$ associated to~$\lambda$ to define the measure
\begin{equation}
\label{e:mu}
\mu = \lambda_{*}(\d v_{d-1} \otimes \nu)
\end{equation}
on $(\LL,\mathcal{B}(\LL))$. We introduce the space of locally finite point measures on~$\LL\times\R_+$:
\begin{equation}
  \Omega  = \Bigg\{ \sum_{i \geq 0} \delta_{(l_i,u_i)};
  \begin{array}{c} \;(l_i,u_i)\in \mathbb{L}\times\R_+ \text{ and }
    \sum_{i \geq 0} \delta_{(l_i,u_i)}(A) < \infty, \\ \text{ for every compact $A \in \mathcal{B}(\mathbb{L}\times\R_+)$}
  \end{array}
  \Bigg\},
\end{equation}
endowed with the $\sigma$-algebra $\mathcal{E}$ generated by the evaluation maps
\[
\varphi_A:\sum_{i \geq 0} \delta_{(l_i,u_i)}\in\Omega\mapsto \sum_{i \geq 0} \delta_{(l_i,u_i)}(A)\in\Z.
\]

We are now able to construct the space $(\Omega,\mathcal{E},\IP)$ of the Poisson point process on $\LL\times \R_+$ with intensity measure~$\mu\otimes \d v_1$, where $\d v_1$ denotes the Lebesgue measure on~$\R_+$. In particular, for~$u\geq 0$, we consider the restriction of said Poisson point process to~$\Omega\times [0,u]$, denoting its distribution as~$\PLP (u\mu)$. In what follows we let~$\omega$ be distributed according to~$\PLP (u\mu)$, and we will frequently identify~$\omega$ with its associated unlabeled set of lines in~$\mathbb{L}$. The cylinder set (with radius~$1$ and intensity~$u$) is then defined as
\begin{equation}
\C_u=\C_u(\omega):=\bigcup_{i;u_i\leq u} B(l_i,1).
\end{equation}
The corresponding vacant set is defined as
\begin{equation}
\mathcal{V}_u=\mathcal{V}_u(\omega):=\R^d\setminus\C_u.
\end{equation}

It will be important for us to define cylinder sets with different radii. Given $\rho>0$, we define
\begin{equation}
\C_u^\rho=\C_u^\rho(\omega):=\bigcup_{l_i\in\omega} B(l_i,\rho),
\end{equation}
the cylinder set with intensity~$u$ and radius~$\rho$. The complementary vacant set is defined analogously:
\begin{equation}
\mathcal{V}_u^\rho:=\R^d\setminus\C_u^\rho.
\end{equation}
The probability measure and expectation associated with these random sets will be denoted by $\IP_u^\rho$ and $\IE_u^\rho$, respectively. When the intensity of the process and radius of the cylinders are clear from the context, or when speaking of the measure which couples all the processes together on $\mathbb{L} \times \mathbb{R}_+$, we will drop the indexes, using simply~$\IP$ and~$\IE$.

Given a bounded measurable set~$A\subset\R^d$, let~$\mathcal{N}_{A}^{u,\rho}$ denote the number of cylinders of~$\C_u^\rho$ intersecting~$A$. We write
\begin{equation}
\label{eq:maurhodef}
\mathcal{M}_A^{u,\rho}(\omega):=\left( \mathcal{C}_u^\rho(\omega) \cap A, \mathcal{N}_{A}^{u,\rho}(\omega)   \right)
\end{equation}
for the variable enconding both the cylinder set intersected with~$A$ and the number of cylinders intersecting~$A$. Given~$\mathcal{B}(A) \times \N$, the set of Borelian subsets of~$A$ times~$\N$, we consider the partial order~$\preceq$ which, for $B,B'\in\mathcal{B}(A)$ and~$m, m' \in \N$, yields
\[
(B,m)\preceq (B',m')\iff B\subseteq B'\text{ and } m\leq m'.
\]
We say that a variable~$f$, measurable with respect to
\[
\sigma(\{\mathcal{M}_{A}^{u,\rho}(\omega);u,\rho\in\R_+\}),
\]
is \emph{increasing} if for~$u,u',\rho,\rho'\in\R_+$ and for different realizations $\omega,\omega'$ of the Poisson line process such that~$\mathcal{M}_{A}^{u,\rho}(\omega)\preceq \mathcal{M}_{A}^{u',\rho'}(\omega')$, we have~$f(\mathcal{M}_{A}^{u,\rho}(\omega))\leq f(\mathcal{M}_{A}^{u',\rho'}(\omega'))$. We say~$g$ is \emph{decreasing} if~$-g$ is increasing.

Though certainly useful, this previous characterization of the Poisson cylinder process will not satisfy our needs completely. It will be crucial to have a characterization that explicitly gives the intersection point of each cylinder axis in~$\LL$ with a given hyperplane, as well as the direction of each axis when viewed from its associated intersection point. With this in mind, we define the set of lines of~$\LL$ which are not contained in any of the planes parallel to $\R^{d-1}\times\{0\}$,
\begin{equation*}
\LL^*:=\LL\setminus\bigcup_{z\in\R}\left\{l\in\LL:l\subset\R^{d-1}\times\{z\}\right\},
\end{equation*}
which has total~$\mu$-measure. We also define the ``northern hemisphere'' of~$S^{d-1}$:
\begin{equation*}
\mathbb{D}:=\left\{  w\in\R^d:|w|=1, \langle w,{\bf e}_d \rangle > 0   \right\}.
\end{equation*}
We can unequivocally associate to each line~$l\in\LL^*$ its intersection point with $\R^{d-1}\times\{0\}$, denoted by~$p(l)$, and its direction~$d(l)\in \mathbb{D}$ when viewed from~$p(l)$. The function
\begin{equation}
  \label{e:2ndplpdef}
  \xi: \LL^* \to (\mathbb{R}^{d-1}\times\{0\})\times \mathbb{D} \text{ defined through } \xi(l) = (p(l),d(l))
\end{equation}
is clearly a bijection. Using the underlying measure structure inherited from~$\R^d$, we introduce in~$\mathbb{D}$ the probability measure~$\chi$ defined by
\newiconst{c:phi}
\newiconst{c:mu}
\begin{equation}
\label{e:defspheremes}
\chi (A):=\useiconst{c:phi}\int_A \langle w,{\bf e}_d \rangle  \sigma(\d w),
\end{equation}
for every measurable set $A\subseteq \mathbb{D}$, where~$\useiconst{c:phi}>0$ is a normalizing constant and~$\sigma$ is the Lebesgue measure on the sphere~$S^{d-1}\supset\mathbb{D}$. We can then use the bijection~$\xi$ to construct a new probability measure on~$\LL^*$:
\begin{equation*}
\tilde{\mu}:=\xi^{-1}_*(\d v_{d-1} \otimes \chi).
\end{equation*}
Since $\tilde{\mu}(\LL\setminus\LL^*)=0$, we can extend the measure~$\tilde{\mu}$ to the whole set~$\LL$ without any trouble. Using Proposition~$2.2.1$ of~\cite{teseDaniel} we can then see that, up to a constant factor, $\mu$ and~$\tilde{\mu}$ are equal: there exists a constant $\useiconst{c:mu}>0$ such that
\begin{equation}
\label{e:tildemu}
\mu=\uc{c:mu}\tilde{\mu}.
\end{equation}
Then by basic properties of the Poisson point process, we have
\begin{lemma}
\label{l:danielplpsampling}
We can regard any~$\omega\eqd\PLP(u\mu)$ as being sampled in the following way:
\begin{itemize}
\item[(i)]Sample a Poisson point process~$\sum_{i} \delta_{x_i}$ in~$\R^{d-1}\times\{0\}$ with intensity measure given by~$u\uc{c:mu}\d v_{d-1}$

\item[(ii)]Independently for each point~$x_i$ sampled by the above process, sample a vector~$d_i\in \mathbb{D}$ according to the measure~$\chi$.

\item[(iii)] For each~$x_i$, consider the line passing through~$x_i$ with direction~$d_i$ relative to the plane~$\R^{d-1}\times\{0\}$.

\item[(iv)]The resulting collection of lines will have the desired distribution.
\end{itemize}
\end{lemma}

Given a compact set~$A\subset \R^d$, we denote by~$\LL_A$ the set of lines in~$\LL$ that intersect~$A$. For~$B\subset \R^d$ also compact, we also write~$\LL_{A,B}:=\LL_A\cap \LL_B$, the set of lines that intersect both~$A$ and~$B$.

\section{Decoupling inequalities}
\label{s:2box}
In this section we will establish a decoupling inequality for the cylinder percolation process, one of the main results of our paper, and also necessary for the subsequent investigations here present. Heuristically we will show that, after a sprinkling of \emph{both} the Poisson process intensity \emph{and} the cylinders' radii, the correlation between the states of the process in two distant boxes becomes stretched exponentially small in the distance, at least when considering monotone functions of said states.

We start with the basic notation needed. Fix the box radius~$L>0$ and three numbers: $\alpha\in(0,1)$, related to the distance between the boxes, the radius-sprinkling value $\eps\in(0,1)$ and the initial cylinder radius $\rho\in[1,4]$.

Given these values we can define the boxes
\begin{equation}
B_1:=B_\infty(0,L) \qquad \text{and} \qquad B_2:=B_1+(2L+L^{2+\alpha}\eps^{-1})\cdot {\bf e}_d,
\end{equation}
so that the distance between~$B_1$ and~$B_2$ equals~$L^{2+\alpha}\eps^{-1}$.

The main result of this section states: \newiconst{c:2boxdec}
\begin{theorem}
  \label{thm:2boxdec}
  There exists a constant~$\useiconst{c:2boxdec}>0$ depending only on the dimension~$d$ such that, for any~$\delta>0$,~$\alpha,\eps\in(0,1)$,~$\rho\in[1,4]$, and any increasing variables
  \begin{equation}
      f_i: \Omega \to [0,1], \text{ measurable with respect to $\sigma(\{\mathcal{M}_{B_i}^{u,\rho}(\omega);u,\rho\in\R_+\})$}
  \end{equation}
  for~$i = 1, 2$, we have
  \begin{equation}
    \label{e:2boxdec}
    \begin{split}
      \IE \big[ f_1 \big( \mathcal{M}_{B_1}^{u, \rho}(\omega) \big) f_2 \big( \mathcal{M}_{B_2}^{u, \rho}(\omega) \big) \big]
    \leq \;
      & \IE \big[ f_1 \big( \mathcal{M}_{B_1}^{u, \rho + \eps}(\omega) \big) \big]
      \IE \big[ f_2 \big( \mathcal{M}_{B_2}^{u + \delta, \rho + \eps}(\omega) \big) \big] \\
      & + c_2^{-1} \exp \big\{ -\useiconst{c:2boxdec} \delta \eps^{d-1} L^{\alpha(d-1)} \big\}.
    \end{split}
  \end{equation}
  Analogously, if
  \begin{equation}
    g_i: \Omega \to [0, 1] \text{ are measurable with respect to } \{\mathcal{M}_{B_i}^{u,\rho}(\omega);u,\rho\in\R_+\}
  \end{equation}
  and decreasing variables for~$i = 1, 2$, then we have
  \begin{equation}
    \label{e:2boxdec2}
    \begin{split}
      \IE \big[ g_1 \big( \mathcal{M}_{B_1}^{u, \rho}(\omega) \big) g_2 \big( \mathcal{M}_{B_2}^{u, \rho}(\omega) \big) \big]
      \leq \;
      & \IE \big[ g_1 \big( \mathcal{M}_{B_1}^{u, \rho - \eps}(\omega) \big) \big]
      \IE \big[ g_2 \big( \mathcal{M}_{B_2}^{u - \delta, \rho - \eps}(\omega) \big) \big] \\
      & + c_2^{-1} \exp \big\{ -\useiconst{c:2boxdec}    \delta \eps^{d-1} L^{\alpha(d-1)} \big\}.
    \end{split}
  \end{equation}
\end{theorem}

\begin{remark}
  We should emphasize that the dependecy present between~$\mathcal{M}_{B_1}^{u, \rho}(\omega)$ and~$\mathcal{M}_{B_2}^{u, \rho}(\omega) \big)$ comes exactly from the cylinders which are able to intersect both boxes $B_1$ and $B_2$. In fact, inequality \eqref{e:decouple_slow} from \cite{TW10b} comes from a bound on the intensity measure of such cylinders.
\end{remark}

What follows is a (very) heuristic roadmap explaining how we will obtain inequality~\eqref{e:2boxdec} (inequality~\eqref{e:2boxdec2} is obtained in an analogous way).
\vspace{.3cm}

\paragraph{Roadmap for the 2-box decoupling inequality}
\begin{itemize}
\item[(i)] We notice that, for large~$L$, the radii of the boxes $B_1$ and $B_2$ are much smaller than their mutual distance.
  Therefore, the lines that touch both boxes (which are the ones that may carry information between them) are those whose directions are closely aligned to~${\bf e}_d$.
  During the proof, we make small perturbations to the directions of these ``problematic'' lines which are ``close'' to intersecting both boxes~$B_1$ and $B_2$.
  These perturbations being done independently for each line and for each box;
\item[(ii)] We show that inside $B_1$, the ``problematic'' cylinders are still covered by their perturbed versions, so long as the perturbed cylinders have a slightly enlarged thickness;
\item[(iii)] Finally, we study the influence that the enlarged cylinder set intersecting~$B_1$ has on the respective set intersecting~$B_2$. We show, using a poissonization argument, that this influence can be dominated by a sprinkling of the parameter~$u$, at least when we exclude an event with vanishingly small probability.
\end{itemize}

In order to rigorously implement the above plan we will need additional definitions. We consider~$\Pi_1:=\R^{d-1}\times\{L\}$, and~$\Pi_2=\R^{d-1}\times\{L+L^{2+\alpha}\eps^{-1}\}$, so that~$\Pi_1$ and~$\Pi_2$ contain opposing faces of the hypercubes~$B_1$ and~$B_2$.

\begin{figure}
  \centering
  \begin{subfigure}[b]{0.3\textwidth}
      \centering
      \includegraphics[width=\textwidth]{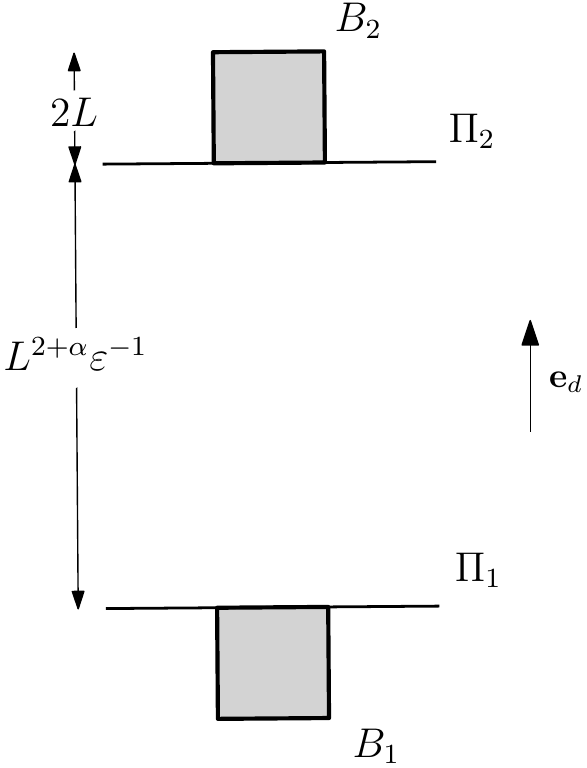}
  \end{subfigure}
  \hfill
  \begin{subfigure}[b]{0.55\textwidth}
      \centering
      \includegraphics[width=\textwidth]{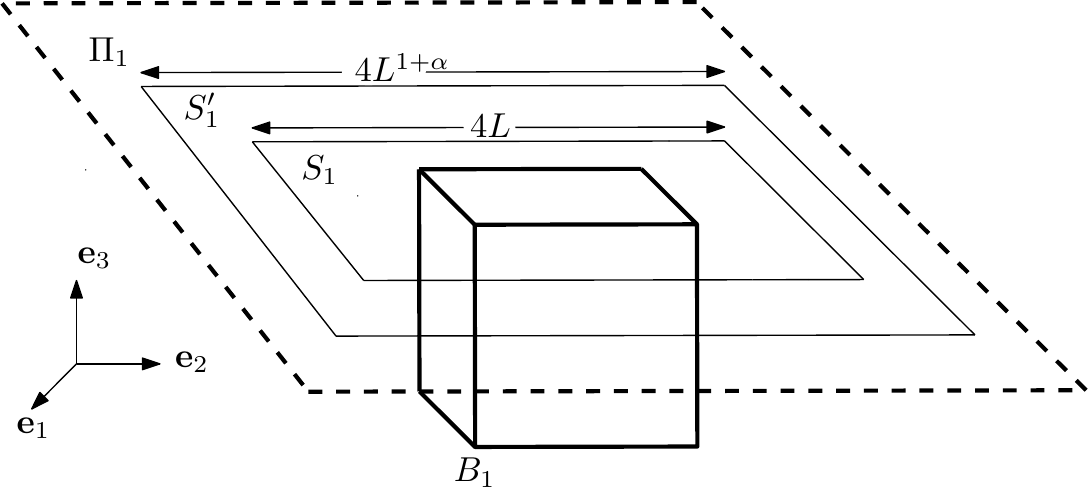}
  \end{subfigure}
  \caption{Two images showing representations of some of the sets involved in the decoupling inequality.}
  \label{f:s1def}
\end{figure}
The sets $\Pi_1$ and $\Pi_2$ allow us to consider two different parametrizations  of the lines in~$\LL^*$. For~$i = 1, 2$, we characterize a line~$l \in \LL^*$ by~$p_i(l)$, its intersection point with~$\Pi_i$, and its direction~$d_i(l)\in\mathbb{D}$, in an analogous manner to that of~\eqref{e:2ndplpdef}. We note that~$d_1(l) = d_2(l) = d(l)$, and that, by translation invariance of~$\mu$, we can sample~$\PLP(u\mu)$ in the manner of Lemma~\ref{l:danielplpsampling}, starting with a Poisson point process in either~$\Pi_1$ or~$\Pi_2$ instead of~$\R^{d-1}\times\{0\}$.

It is important to consider the subsets of $\Pi_i$ where the ``problematic'' lines start:
\begin{equation}
  \begin{array}{cclcl}
    S_1 & := & [-2L, 2L]^{d - 1} \times \{L\} & \subset & \Pi_1, \\
    S_2 & := & [-2L, 2L]^{d - 1} \times \{L + L^{2 + \alpha} \eps^{-1}\} & \subset & \Pi_2.
\end{array}
\end{equation}
Since we are going to perturb these lines, it is also important to consider a larger version of the above sets
\begin{equation}
  \begin{array}{cclcl}
    S_1' & := & [-2L^{1 + \alpha}, 2L^{1 + \alpha}]^{d - 1} \times \{L\} & \subset & \Pi_1, \\
   S_2' & := & [-2L^{1 + \alpha}, 2L^{1 + \alpha}]^{d - 1} \times \{L + L^{2 + \alpha} \eps^{-1} \} & \subset & \Pi_2.
  \end{array}
\end{equation}
The $(d-1)$-dimensional squares~$S_1', S_2'$ are the sets which will tell us if a line is ``close'' to~$B_1, B_2$, respectively, in the context of item (i) of our Roadmap.
Note that $S_1 \subset S_1'$ and $S_2 \subset S_2'$, see Figure~\ref{f:s1def}.

As we mentioned in the proof overview, the ``problematic'' lines are those aligned with the vertical direction.
It is therefore natural to define the spherical cap centered at the ``north pole''~${\bf e}_d$ with (Euclidean metric) diameter~$\eps/(8L)$:
\begin{equation}
\label{e:depsldef}
D_{\eps,L}:=\left\{ x \in \mathbb{D}; \dist(x,{\bf e}_d)< \eps(8L)^{-1}      \right\}.
\end{equation}
Define also~$\tilde{B}_i:=B(B_i,\rho(1+\eps))$, $i=1,2$, and note that if a line does not intersect this open neighborhood of~$B_i$, then the associated cylinder of radius~$\rho+\eps$ does not intersect~$B_i$.
We note that, for sufficiently large~$L$,
\begin{display}
  in order for a line with direction in~$D_{\eps,L}$ to intersect~$\tilde B_i$,\\
  it has also to intersect~$S_i$, for~$i=1,2$.
\end{display}

The final ingredient in our proof is a decomposition of our point measure into independent processes, distinguishing the lines depending on their directions and the sets they intersect.
This decomposition will make it clear why the vertically aligned lines are the source of dependence between $B_1$ and $B_2$.

Let us decompose the lines from
\[
\omega\equiv \sum_{i \geq 0,u_i\leq u} \delta_{(l_i,u_i)}\eqd\PLP(u\mu),
\]
intersecting~$\tilde B_i$ into separate (but not necessarily disjoint) point measures. As we mentioned, the first two are not troublesome, as they are unable to carry information from the cylinder state inside one box to the other. On the other hand, controlling the dependencies associated to the third point measure is the main focus of this section. Consider
\begin{equation}
\begin{split}
\eta_1^0 &:=\sum_{i; \; u_i\leq u} \delta_{(l_i,u_i)}{\bf 1} \{ l_i\cap   \tilde B_1\neq \varnothing; d_1(l_i)\notin D_{\eps,L}             \},
\\
\eta_2^0 &:=\sum_{i; \; u_i\leq u} \delta_{(l_i,u_i)}{\bf 1} \{ l_i\cap  \tilde B_2\neq \varnothing; d_2(l_i)\notin D_{\eps,L}             \},
\\
\eta& :=\sum_{i; \; u_i\leq u} \delta_{(l_i,u_i)}{\bf 1} \{  d(l_i) \in D_{\eps,L} \text{ and either }      p_1(l)\in S_1' \text{ or }    p_2(l)\in  S_2   '  \}.
\end{split}
\end{equation}
We note that, by elementary trigonometry, for large enough~$L$ the lines of~$\eta_1^0$ do not intersect~$\tilde B_2$, and the same holds changing the places of the indices~$1$ and~$2$.

We can now define the way in which we will perturb the directions of the cylinders' axes inside each box, as previewed in item (i) of the Roadmap. We will define two stochastic operations that essentially re-sample the direction~$d_i(l)\in D_{\eps,L}$ of each line~$l\in\eta$ while fixing the intersection point~$p_i(l)$ of~$l$ with~$S_i'$. Denote by~$\bar \chi_{\eps,L}$ the probability measure~$\chi$ defined in~\eqref{e:defspheremes} conditioned on sampling a point in~$D_{\eps,L}$. For~$i=1,2$ we define the stochastic operation
\begin{equation}
\label{eq:gammadef}
\begin{array}{cclc}
\Gamma_i: &\eta & \to& \Gamma_i(\eta)\\
&(p_i(l),d_i(l)) & \mapsto& (p_i(l), d_i'(l)) ,
\end{array}
\end{equation}
where~$ d_i'(l)$ is defined to be either a random vector in~$D_{\eps,L}$ sampled according to~$\bar \chi_{\eps,L}$ independently for each~$l\in\eta$ if~$p_i(l)\in S_i'$, or simply equal to~$d_i(l)$ otherwise. See Figure~\ref{f:gamma1gamma2eta} for an illustration of these stochastic operations.

\begin{figure}[ht]
  \centering
  \includegraphics[scale = .6]{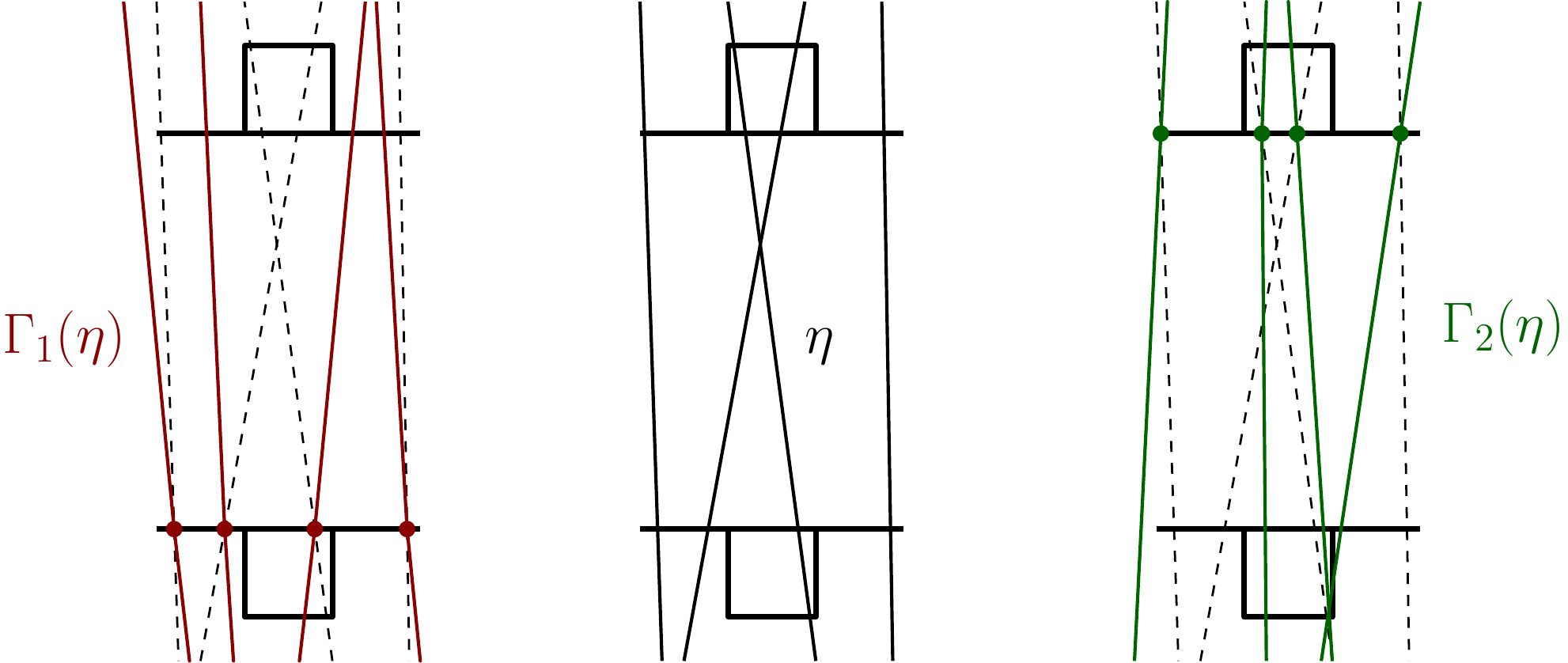}
  \vspace{0.1cm}
  \caption{The \emph{potentially problematic} lines of $\eta$, together with its perturbed versions, $\Gamma_1(\eta)$ and $\Gamma_2(\eta)$. The stochastic operation $\Gamma_i$ consists in fixing the intersection point of a line with the plane $\Pi_i$ and resampling its direction, conditioned on it being ``problematic''.}
  \label{f:gamma1gamma2eta}
\end{figure}

Crucially, by elementary properties of the Poisson process, we get that $\Gamma_i$ are reversible.
More precisely, they satisfy the detailed balance conditions
\begin{equation}
  \label{eq:gammarevers}
  \big( \eta, \Gamma_i (\eta) \big) \eqd \big( \Gamma_i(\eta), \eta \big), \end{equation}
for $i = 1, 2$.

We can now rigorously state and prove step (ii) of the Roadmap.
The lemma below is a deterministic statement which, informally speaking, says that the wiggling introduced by the $\Gamma_i$ operators can be dominated by slight thickening of the cylinders' radii. For an illustration showing this domination, see Figure~\ref{f:cylinder_contain}.
\begin{figure}[ht]
  \centering
  \includegraphics[scale = .5]{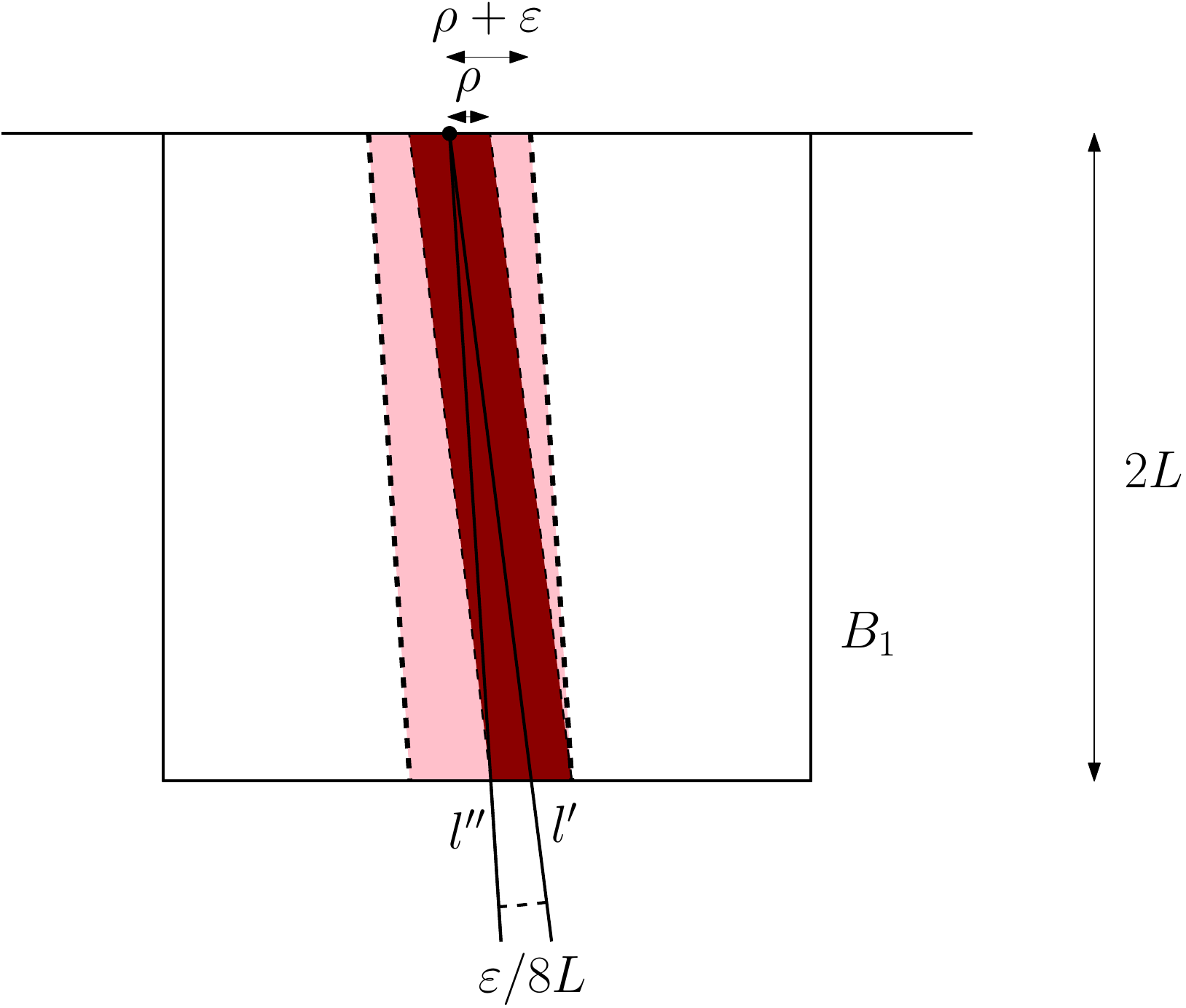}
  \vspace{0.1cm}
  \caption{By enlarging the radii of the perturbed cylinders, we will have the required domination between the cylinder processes: indeed the intersection of the smaller cylinders with the box $B_!$ will be contained in the intersection of the larger cylinders with the same box.}
  \label{f:cylinder_contain}
\end{figure}

\begin{lemma}
\label{l:bend}
With the notation above developed we have, for~$i=1,2$, and sufficiently large~$L$,
\begin{equation}
\label{eq:lemmabend}
\mathcal{M}_{B_i}^{u,\rho-\eps}\big(\Gamma_i(\eta)+\eta_i^0\big)\preceq \mathcal{M}_{B_i}^{u,\rho}\big(\eta+\eta_i^0\big) \preceq \mathcal{M}_{B_i}^{u,\rho+\eps}\big(\Gamma_i(\eta)+\eta_i^0\big).
\end{equation}
\end{lemma}

\begin{proof}
  We will focus on the case~$i=1$ since the other follows analogously.
  Fix $v',v''\in D_{\eps,L}$, and denote respectively by~$v'_d,v''_d\in\R$ their~$d$-th coordinates. Let then~$p=(p_1,\dots,p_{d-1},L)\in \Pi_1$, and consider the lines
  \[
    l'=\{v' t +p;t\in \R \}\quad\text{ and }\quad l''=\{v''s +p; s\in \R \}.
  \]
  For~$z_0\in [-L,L]$, we show that
  \begin{display}
    \label{e:wiggle}
    for large enough $L$, the distance between the points\\
    $(\R^{d-1}\times \{z_0\} )\cap l'$ and $(\R^{d-1}\times \{z_0\} )\cap l''$ is smaller than~$3\eps/4$,
\end{display}
  which will prove the result. Write~$t_0:=(z_0-L)(v_d')^{-1}$ and~$s_0:=(z_0-L)(v_d'')^{-1}$. We have
  \[
    (\R^{d-1}\times \{z_0\} )\cap l' =p+t_0 v'\quad\text{ and }\quad(\R^{d-1}\times \{z_0\} )\cap l''  =p+ s_0 v''.
  \]
  Notice that by the Law of cosines,
  \[
    1-\frac{\eps^2}{128 L^2} \leq v_d',v_d'' \leq 1.
  \]
  We can then show, for sufficiently large~$L$,
  \begin{equation}
    \label{e:item3theta}
    \begin{split}
      \left|  p+t_0 v' -p - s_0 v''         \right|
      &=
      |t_0|\left|v'-\frac{s_0}{t_0}v''\right|
      \leq
      |t_0|\left(\left|v'-v''\right|  + \left|v''-\frac{s_0}{t_0}v''\right|           \right)
      \\
      &\leq
      \left|\frac{z_0-L}{v_d'}\right|\left(\frac{\eps}{4L}  + \left|1-\frac{s_0}{t_0}\right|           \right)
      \\
      &\leq
      2L\left(1+\frac{\eps^{2}}{64 L^2}\right)\left(\frac{\eps}{4L}  + 1-  \left(1+\frac{\eps^{2}}{64 L^2}\right)     \right)
      \leq \frac{3}{4}\varepsilon,
    \end{split}
  \end{equation}
  where in the last two inequalities, we used the fact that~$\eps/L$ was sufficiently small.
  This proves \eqref{e:wiggle} and consequently the result.
\end{proof}

Our next objective is to show how a slight change in the intensity of the process can be used to dominate the negative information that we may have obtained by looking at the other box.

We first split the random point measure~$\Gamma_1(\eta)$ into two measures taking into account whether their constituent lines intersect~$S_1$ or not. First define
\begin{equation}
\label{e:etas1}
\begin{split}
\eta_{S_1}&:=\sum_{(l_i,u_i)\in\eta} \delta_{(l_i,u_i)}{\bf 1} \{  p_1(l_i)\in S_1 ;    d_1(l_i)\in D_{\eps,L}      \};
\\
\eta_{S_1'\setminus S_1}&:=\eta-\eta_{S_1}.
\end{split}
\end{equation}
We then define the images of the above point measures after applying the stochastic operation~$\Gamma_1$,
\begin{equation}
\label{e:etas1g}
\begin{split}
 \Gamma_1(\eta_{S_1})&:=\sum_{(l_i,u_i)\in\Gamma_1(\eta)} \delta_{(l_i,u_i)}{\bf 1} \{  p_1(l_i)\in S_1 ;    d_1'(l_i)\in D_{\eps,L}      \};
\\
\Gamma_1(\eta_{S_1'\setminus S_1})&:=\Gamma_1(\eta)-\Gamma_1(\eta_{S_1}).
\end{split}
\end{equation}
Recall that, for sufficiently large~$L$, in order for a line with direction in~$D_{\eps,L}$ to intersect~$\tilde B_1$, it has also to intersect~$ S_1$. Therefore, between the two measures above, $\Gamma_1(\eta_{S_1})$ is the only one that can actually influence the cylinder set inside~$B_1$.

The following proposition rigorously states the first part of item~(iii) of the Roadmap. It provides us with a quantitative statement concerning the influence of~$\Gamma_1(\eta_{S_1})$ on the cylinder set intersected with~$B_2$, and it will be the kernel of the proof of Theorem~\ref{thm:2boxdec}.
\begin{proposition}
  \label{p:sprinkle}
  There exists a constant~$\useiconst{c:2boxdec}>0$ depending only on the dimension~$d$ such that, for~$\delta>0$,~$\alpha,\eps\in(0,1)$ and any increasing variable
  \begin{equation}
    f_2: \Omega \to [0,1], \text{ measurable with respect to $\sigma(\{\mathcal{M}_{B_2}^{u,\rho}(\omega);u,\rho\in\R_+\})$}
  \end{equation}
  we have,
  \begin{equation}
    \label{e:sprinkle1}
    \begin{split}
      \IE \big[ f_2 \big( \mathcal{M}_{B_2}^{u, \rho + \eps} \big( \Gamma_2(\eta) + \eta_2^0 \big) \big) \big\vert \Gamma_1(\eta_{S_1}) \big]
      \leq & \; \IE \big[ f_2 \big( \mathcal{M}_{B_2}^{u + \delta, \rho + \eps} \big( \Gamma_2(\eta) + \eta_2^0 \big) \big) \big] \\
      & + {\bf 1} \{\eta_{S_1} \in A\}
    \end{split}
  \end{equation}
  where the event $A$ satisfies
  \begin{equation}
    \mathbb{P} [ \eta_{S_1} \in A ] \leq \exp \big\{ -\useiconst{c:2boxdec} \delta \eps^{d - 1} L^{\alpha(d - 1)} \big\}.
  \end{equation}
  Furthermore, if
  \begin{equation}
    g_2: \Omega \to [0,1], \text{ measurable with respect to $\sigma(\{\mathcal{M}_{B_2}^{u,\rho}(\omega);u,\rho\in\R_+\})$}
  \end{equation}
  is a decreasing variable, then for $\delta \in (0, u)$ and~$\eps \in (0, \rho)$, we have
  \begin{equation}
    \label{e:sprinkle2}
    \begin{split}
      \IE \big[ g_2 \big( \mathcal{M}_{B_2}^{u, \rho - \eps} \big( \Gamma_2(\eta) + \eta_2^0 \big) \big) \big\vert \Gamma_1(\eta_{S_1}) \big]
      \leq & \; \IE \big[ g_2 \big( \mathcal{M}_{B_2}^{u - \delta, \rho - \eps} \big( \Gamma_2(\eta) + \eta_2^0 \big) \big) \big] + {\bf 1} \{ \eta_{S_1} \in B \}
    \end{split}
  \end{equation}
  where
  \begin{equation}
    \mathbb{P} [\eta_{S_1} \in B] \leq \exp \big\{ -\useiconst{c:2boxdec} \delta\eps^{d - 1}L^{\alpha(d - 1)} \big\}.
  \end{equation}
\end{proposition}

The above proposition is the heart of the proof of our main theorem.
We thus postpone its proof to the end of the Section and show now that it is enough to establish Theorem~\ref{thm:2boxdec}.

\begin{proof}[Proof of Theorem~\ref{thm:2boxdec}]
Using Lemma~\ref{l:bend}, Proposition~\ref{p:sprinkle}, the fact that the lines in~$\Gamma_1(\eta_{S_1'\setminus S_1})$ do not intersect~$\tilde B_1$, and the fact that~$f_1,f_2$ are increasing functions, we obtain
\begin{equation}
  \label{e:decorr1}
  \begin{split}
    \IE_u^\rho & \big[ f_1 \big( \mathcal{M}_{B_1}^{u, \rho}(\omega) \big) f_2 \big( \mathcal{M}_{B_2}^{u, \rho}(\omega) \big) \big]
    = \IE \big[ f_1 \big( \mathcal{M}_{B_1}^{u, \rho} \big(\eta + \eta_1^0 \big) \big) f_2\big(\mathcal{M}_{B_2}^{u,\rho}\big(\eta+\eta_2^0\big)\big)    \big] \\
    &\!\!\!\!\!\!\!\!\!\stackrel{\mathrm{Lemma}~\ref{l:bend}}{\leq} \IE \big[ f_1 \big( \mathcal{M}_{B_1}^{u, \rho + \eps} \big( \Gamma_1(\eta) + \eta_1^0 \big) \big) f_2 \big( \mathcal{M}_{B_2}^{u, \rho + \eps} \big(\Gamma_2(\eta) + \eta_2^0 \big) \big) \big] \\
    &= \IE \big[ f_1 \big( \mathcal{M}_{B_1}^{u, \rho + \eps} \big( \Gamma_1(\eta_{S_1}) + \eta_1^0 \big) \big) f_2 \big( \mathcal{M}_{B_2}^{u, \rho + \eps} \big( \Gamma_2(\eta) + \eta_2^0 \big) \big) \big] \\
    &= \IE \big[ f_1 \big( \mathcal{M}_{B_1}^{u, \rho + \eps} \big( \Gamma_1(\eta_{S_1}) + \eta_1^0 \big) \big) \IE \big[ f_2 \big( \mathcal{M}_{B_2}^{u, \rho + \eps} \big( \Gamma_2(\eta) + \eta_2^0 \big) \big)  \big\vert \Gamma_1(\eta_{S_1}), \eta_1^0 \big] \big] .
  \end{split}
\end{equation}
Furthermore, using Proposition~\ref{p:sprinkle}, the fact~$\Gamma_2(\eta)$ and $\eta_2^0$ are both independent from~$\eta_1^0$, and that $\|f_1\|_\infty,\|f_2\|_\infty\leq 1$, we get
\begin{equation}
  \begin{split}
    \label{e:decorr2}
    \IE_u^\rho & \big[ f_1 \big( \mathcal{M}_{B_1}^{u, \rho}(\omega) \big) f_2 \big( \mathcal{M}_{B_2}^{u, \rho}(\omega) \big) \big]\\
    \leq & \; \IE \big[ f_1 \big( \mathcal{M}_{B_1}^{u, \rho + \eps} \big( \Gamma_1(\eta_{S_1}) + \eta_1^0 \big) \big) \big]
    \IE \big[ f_2 \big( \mathcal{M}_{B_2}^{u + \delta, \rho + \eps} \big( \Gamma_2(\eta) + \eta_2^0 \big) \big) \big] \\
    & + \exp \big\{ -\useiconst{c:2boxdec} \delta \eps^{d - 1}L^{\alpha(d - 1)} \big\}
  \end{split}
\end{equation}
and since $\mathcal{M}_{B_1}^{u,\rho+\eps}\big(\Gamma_1(\eta_{S_1})+\eta_1^0 \big)$ has the same distribution as~$\mathcal{M}_{B_1}^{u,\rho+\eps}\big(\omega \big)$,
\begin{equation}
  \begin{split}
    = & \; \IE \big[f_1\big(\mathcal{M}_{B_1}^{u,\rho+\eps}\big(\omega \big) \big) \big] \IE \big[ f_2 \big( \mathcal{M}_{B_2}^{u + \delta, \rho + \eps} \big( \omega \big) \big) \big]
    +\exp\big\{ -\useiconst{c:2boxdec} \delta \eps^{d - 1}L^{\alpha(d - 1)}      \big\}
  \end{split}
\end{equation}
Equation~\eqref{e:2boxdec2} follows by an analogous argument.
\end{proof}

Now that we have demonstrated how Proposition~\ref{p:sprinkle} can be used to derive our main result, let us turn to the proof of this proposition.

We start by considering the main expectation appearing in the proposition. Using that $\eta = \eta_{S_1} + \eta_{S_1' \setminus S_1}$ and the fact that $\Gamma_2$ acts independently in each line, we can write
\begin{equation}
  \label{e:main_expectation}
  \IE \Big[ f_2 \Big( \mathcal{M}_{B_2}^{u, \rho + \eps} \big( \Gamma_2(\eta_{S_1}) + \Gamma_2(\eta_{S_1' \setminus S_1}) + \eta_2^0 \big) \Big) \Big\vert \Gamma_1(\eta_{S_1}) \Big].
\end{equation}
Observing now that $\eta_{S_1}, \eta_{S_1' \setminus S_1}$ and $\eta^0_2$ are independent, we see that the only information obtained by the conditioning is contained in the term $\eta_{S_1}$.

Note that, the detailed balance conditions in \eqref{eq:gammarevers} are also valid for the corresponding restrictions to $S_1$ and $S_1' \setminus S_1$, that is
\begin{equation}
  \label{e:etasgetaeqd}
  \begin{array}{c}
    \big( \eta_{S_1}, \Gamma_1(\eta_{S_1}) \big) \eqd \big( \Gamma(\eta_{S_1}), \eta_{S_1} \big) \quad \text{and} \quad
    \big( \eta_{S_1' \setminus S_1}, \Gamma_1(\eta_{S_1' \setminus S_1}) \big) \eqd \big( \Gamma(\eta_{S_1' \setminus S_1}), \eta_{S_1' \setminus S_1} \big).
  \end{array}
\end{equation}
Therefore, we can rewrite \eqref{e:main_expectation} as
\begin{equation}
  \label{e:main_expectation_2}
  \IE \Big[ f_2 \Big( \mathcal{M}_{B_2}^{u, \rho + \eps} \big( \Gamma_2(\Gamma_1(\eta_{S_1})) + \Gamma_2(\eta_{S_1' \setminus S_1}) + \eta_2^0 \big) \Big) \Big\vert \eta_{S_1} \Big].
\end{equation}

Based on the above calculations, the next step in our proof is to reduce Proposition~\ref{p:sprinkle} to a simpler statement.

\begin{proposition}
  \label{p:sprinkle2}
  There exists an event $A$ such that
  \begin{equation}
    \label{e:prob_A}
    \mathbb{P}[ \eta_{S_1} \in A ] \leq \exp \big\{ -\useiconst{c:2boxdec} \delta \eps^{d - 1} L^{\alpha(d - 1)} \big\},
  \end{equation}
  and moreover, in the event $\eta_{S_1} \notin A$,
  \begin{equation}
    \label{e:wiggle_dominates}
    \begin{array}{c}
      \Gamma_2 \big( \Gamma_1(\eta_{S_1}) \big) \text{ is stochastically dominated by } \eta_\delta; \text{ where } \eta_\delta \text{ is distributed as } PLP(\delta \mu),
      \\
      \text{ and is independent from } \eta_{S_1}, \, \eta_2^0, \text{ and } \Gamma_2(\eta_{S_1' \setminus S_1}).
    \end{array}
  \end{equation}
\end{proposition}

Assuming the validity of the above, we can jump to the following.

\begin{proof}[Proof of Proposition~\ref{p:sprinkle}]
  We will prove~\eqref{e:2boxdec}, since~\eqref{e:2boxdec2} has essentially the same proof, with the difference being that the sprinkling term~$\delta$ clearly cannot be larger than the parameter~$u$.

  We use \eqref{e:main_expectation} and \eqref{e:main_expectation_2} to write the main expectation as:
  \begin{equation}
    \label{e:sprinkle1}
    \begin{split}
      \IE \big[ & f_2 \big( \mathcal{M}_{B_2}^{u, \rho + \eps} \big( \Gamma_2(\eta) + \eta_2^0 \big) \big) \big\vert \Gamma_1(\eta_{S_1}) \big] \\
      &
      \begin{array}{e}
        & = &\IE \Big[ f_2 \Big( \mathcal{M}_{B_2}^{u, \rho + \eps} \big( \Gamma_2(\Gamma_1(\eta_{S_1})) + \Gamma_2(\eta_{S_1' \setminus S_1}) + \eta_2^0 \big) \Big) \Big\vert \eta_{S_1} \Big] \\
        & \overset{\eqref{e:etasgetaeqd}}= &\IE \Big[ f_2 \Big( \mathcal{M}_{B_2}^{u, \rho + \eps} \big( \Gamma_2(\Gamma_1(\eta_{S_1})) + \eta_{S_1' \setminus S_1} + \eta_2^0 \big) \Big) \Big\vert \eta_{S_1} \Big] \\
        & \overset{\eqref{e:prob_A}}= & \IE \Big[ f_2 \Big( \mathcal{M}_{B_2}^{u, \rho + \eps} \big( \Gamma_2(\Gamma_1(\eta_{S_1})) + \eta_{S_1' \setminus S_1} + \eta_2^0 \big) \Big) {\bf 1}_A \Big\vert \eta_{S_1} \Big] + {\bf 1} \{\eta_{S_1} \in A\} \\
        & \overset{\eqref{e:wiggle_dominates}}\leq & \IE \Big[ f_2 \Big( \mathcal{M}_{B_2}^{u, \rho + \eps} \big( \eta_\delta + \eta_{S_1' \setminus S_1} + \eta_2^0 \big) \Big) \Big\vert \eta_{S_1} \Big] + {\bf 1} \{\eta_{S_1} \in A\} \\
        & \leq & \IE \big[ f_2 \big( \mathcal{M}_{B_2}^{u + \delta, \rho + \eps} \big( \Gamma_2(\eta) + \eta_2^0 \big) \big) \big] + {\bf 1} \{\eta_{S_1} \in A\},
      \end{array}
    \end{split}
  \end{equation}
  as desired.
\end{proof}

We are now left with the proof of Proposition~\ref{p:sprinkle2}, which in turn will be based on a comparison of the intensities of Poisson Point Processes.
Therefore it is natural to start with the estimate of the measure of $S_i \times D_{\eps, L}$ below.

\newiconst{c:s1s2intens}

\begin{lemma}
  \label{l:s1s2intens}
  There exists a constant~$\useiconst{c:s1s2intens}>0$ such that for~$i=1,2$,
  \begin{equation}
    \label{e:s1s2intens}
    \tilde\mu\left(  S_i \times D_{\eps,L}\right) = \useiconst{c:s1s2intens} \eps^{d-1}\Big(1-\frac{\eps^2}{256L^2}\Big)^{\frac{d-1}{2}}.
  \end{equation}
\end{lemma}

\begin{proof}
  We know by the definition of~$\tilde \mu$ that
  \[
    \tilde\mu\left(  S_i \times D_{\eps,L}\right) = 4^{d-1}L^{d-1}\chi(D_{\eps,L}),
  \]
  so that we need only to properly estimate~$\chi(D_{\eps,L})$ in order to prove the result. Using the Law of cosines and spherical coordinates with~${\bf e}_d$ as the north pole, we can parametrize~$D_{\eps,L}$ as
  \begin{equation}
    D_{\eps,L}:=\left\{
      \begin{array}{c}
	r=1,(\phi_1,\dots,\phi_{d-2})\in[0,\pi]^{d-2},\psi\in[0,2\pi]; \\ \phi_1\leq\arccos\left(1-\frac{\eps^2}{128L^2}\right)
      \end{array}
    \right\}.
  \end{equation}
  Equation~\eqref{e:defspheremes} then implies
  \begin{equation}
    \label{e:chidepsl}
    \begin{split}
      \chi (D_{\eps,L})
      &=\useiconst{c:phi}\int_{[0,2\pi]}\int_{[0,\pi]^{d-3}}\left(\int_{0}^{\arccos(1-\eps^2/128L^2)}\!\!\!\!\cos (\phi_1) \sin^{d-2}(\phi_1 )  \d\phi_1 \right)
      \\
      &\phantom{*************}\times \sin^{d-3}(\phi_2)\dots \sin(\phi_{d-2}) \d \phi_2\dots \d \phi_{d-2} \d \psi
      \\
      &=c\sin^{d-1}\left(    \arccos(1-\eps^2/128L^2)          \right)
            =
      c\frac{\eps^{d-1}}{L^{d-1}}\Big(1-\frac{\eps^2}{256L^2}\Big)^{\frac{d-1}{2}}
      ,
    \end{split}
  \end{equation}
  which finishes the proof of the lemma.
\end{proof}

For the proof Proposition~\ref{p:sprinkle2} we will need a lemma quantifying the influence that each line of~$\eta_{S_1}$ has on~$\Gamma_2(\Gamma_1 (\eta))$.
Consider a line~$l \in \eta_{S_1}$ with parameters $(p_1(l),d_1(l))$. We first apply~$\Gamma_1$ to it in order to obtain a line with parameterization~$(p_1(l),d_1'(l))$ belonging to~$\Gamma(\eta_{S_1})$. Note that this stochastic operation \emph{changes the intersection point of}~$l$ \emph{with}~$\Pi_2$. We then apply~$\Gamma_2$ to the resulting line. We denote this stochastic operation by~$\Gamma_2\circ\Gamma_1$. Informally, the next lemma shows that this operation greatly dilutes the information carried by conditioning on~$l$. \newiconst{c:intersecdens}

\begin{lemma}
  \label{l:theta}
  Consider~$l\in \eta_{S_1}$. Denote by~$\Gamma_2\circ\Gamma_1(l)$ the line in~$\Gamma_2(\Gamma_1(\eta))$ corresponding to~$l$ in~$\eta_{S_1}$, and by $\bar\chi_{\eps,L}$ the distribution $\chi$ conditioned on sampling from $D_{\eps,L}$. There exists a constant~$\useiconst{c:intersecdens}>0$ such that for every~$p\in \Pi_1$,~$d\in D_{\eps,L}$ and sufficiently large~$L$,
  \begin{equation}
    \label{e:intersecdens}
    \begin{split}
      \lefteqn{\IP \left(  p_2(\Gamma_2\circ \Gamma_1(l)) \in A   , d_2(\Gamma_2\circ \Gamma_1(l)) \in B \middle \vert p_1(l)=p,d_1(l)=d \right)}\phantom{****************}
      \\
      &\leq \useiconst{c:intersecdens}L^{-(1+\alpha)(d-1)}  \int {\bf 1}_A \d v_{d-1}\cdot  \bar\chi_{\eps,L}(B),
    \end{split}
  \end{equation}
  for every Borelian subsets~$A\subseteq \Pi_2$, $B\subseteq D_{\eps,L}$.
\end{lemma}

\begin{proof}
Consider the projection
\begin{equation}
  \label{e:pid}
  \pi_{p,\mathbb{D}}: \Pi_2 \to \mathbb{D} \text{ taking } x \text{ and mapping to } \frac{x-p}{|x-p|},
\end{equation}
and notice that it is actually a bijection.
\begin{figure}[ht]
	\centering
	\includegraphics[scale = 1]{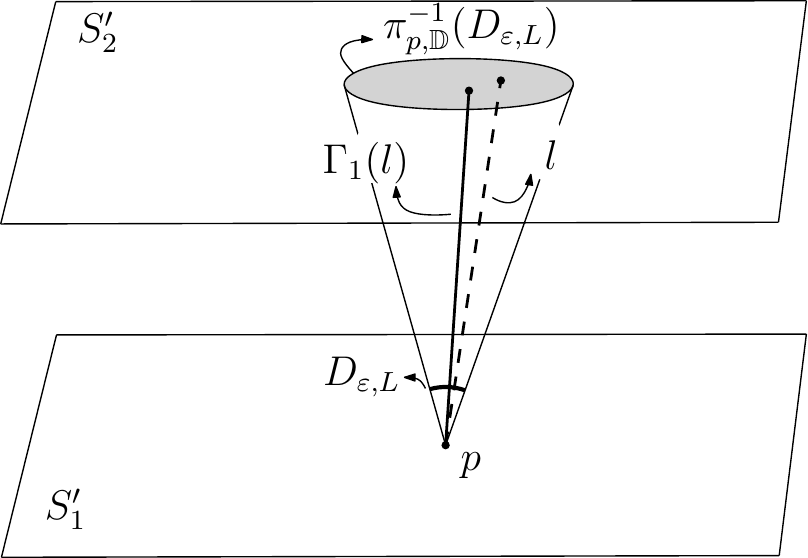}
	\vspace{0.1cm}
	\caption{The figure shows schematically how~$\Gamma_1$ acts on a line~$l$.}
	\label{f:piepsd}
\end{figure}
Writing
\[
x=(x_1,\dots,x_{d-1},L+L^{2+\alpha}\eps^{-1}),\quad\quad p=(p_1,\dots,p_{d-1},L),
\]
we can compute the partial derivative of~$\pi_{p,\mathbb{D}}$ in the~$j$-th direction: for~$j=1,\dots,d-1$,
\begin{equation}
\label{e:dpid}
\begin{split}
\partial_j \pi_{p,\mathbb{D}}(x)
&=
\frac{{\bf e}_i}{|x-p|}-\frac{(x-p)(x_{j}-p_j)}{|x-p|^3}.
\end{split}
\end{equation}
In particular, since~$|x-p|\geq L^{2+\alpha}\eps^{-1}$, each coordinate of~$\partial_j \pi_{p,\mathbb{D}}$ is bounded from above in absolute value by~$2L^{-(2+\alpha)}\eps$ for large enough $L$, which yields
\begin{equation}
\label{e:dpid2}
|\det \d \pi_{p,\mathbb{D}}| \leq c L^{-(d-1)(2+\alpha)}\eps^{d-1}.
\end{equation}
Conditioned on the fact that~$p_1(l)=p\in  S_1$, $p_2(\Gamma_1(l))$ takes value on $\pi_{p,\mathbb{D}}^{-1}(D_{\eps,L})$, and, by construction, its direction~$d_1(\Gamma_1(l))=d_2(\Gamma_1(l))$ is independent from~$d_1(l)$. Furthermore, $\pi_{p,\mathbb{D}}(p_2(\Gamma_1(l)))$ is distributed according to~$\bar\chi_{\eps,L}$. We then have, by the change of variables formula, Equations~\eqref{e:defspheremes}, \eqref{e:chidepsl}, \eqref{e:dpid2}, and the definition of~$\bar\chi_{\eps,L}$, for some Borelian set~$A\subset \pi_{p,\mathbb{D}}^{-1}(D_{\eps,L})$,
\begin{equation}
\label{e:pid2}
\begin{split}
\lefteqn{\IP\left(p_2( \Gamma_1(l))\in A \middle \vert p_1(l)=p,d_1'(l)=d \right)}\phantom{**************}
\\
&=
\frac{\useiconst{c:phi}}{\chi(D_{\eps,L})}\int_{\pi_{p,\mathbb{D}}(A)} \langle w, {\bf e}_d\rangle \sigma(\d w)
\\
&=
\frac{\useiconst{c:phi}}{\chi(D_{\eps,L})}\int_{A} \langle \pi_{p,\mathbb{D}}(x), {\bf e}_d\rangle
|\det \d \pi_{p,\mathbb{D}}|
\d v_{d-1} (x)
\\
&\leq c L^{-(d-1)(1+\alpha)} \int{\bf 1}_A
\d v_{d-1}.
\end{split}
\end{equation}
We note that, by definition,
\[
p_2(\Gamma_2 \circ \Gamma_1(l)) = p_2( \Gamma_1(l)).
\]
Also, by the definition of the sets~$S_1$ and~$S_2'$, as well as elementary trigonometry, we must have~$p_2( \Gamma_1(l))\in S_2'$ for sufficiently large~$L$.
This implies, by the construction of~$\Gamma_2$, that~$d_2(\Gamma_2\circ \Gamma_1(l)) $ is independent from these random elements and distributed according to~$\bar\chi_{\eps,L}$.
\end{proof}

We can now prove Proposition~\ref{p:sprinkle2}.
\begin{proof}[Proof of Proposition~\ref{p:sprinkle2}]

Let~$(l_k)_{k=1}^{\mathcal{N}(\eta_{S_1})}$ denote the collection of lines of~$\eta_{S_1}$. We have
\begin{equation}
\begin{split}
\label{e:lkN12}
{\IE \big[f_2\big(\mathcal{M}_{B_2}^{u,\rho+\eps}\big(\Gamma_2(\Gamma_1(\eta))+\eta_2^0\big)\big)  \big\vert \eta_{S_1} \big]}
=
\IE \left[f_2\big(\mathcal{M}_{B_2}^{u,\rho+\eps}\big(\Gamma_2(\Gamma_1(\eta))+\eta_2^0\big)\big)  \big\vert(l_k)_{k=1}^{\mathcal{N}(\eta_{S_1})}\right].
\end{split}
\end{equation}
Note that~$\Gamma_2\circ\Gamma_1$ acts independently on each line of~$(l_k)_{k=1}^{\mathcal{N}(\eta_{S_1})}$ by construction, and that by Lemma~\ref{l:s1s2intens}, the variable~$\mathcal{N}(\eta_{S_1})$ denoting the number of lines in~$\eta_{S_1}$ has Poisson distribution with parameter bounded from above by~$\uc{c:s1s2intens}u\eps^{d-1}$. Let~$G_N$ denote the event where~$\mathcal{N}(\eta_{S_1})\leq N\in\N$. By Lemma~\ref{l:s1s2intens}, Equations~\eqref{e:lkN12} and the fact that
\[
\mathcal{M}_{B_2}^{u,\rho+\eps}(\Gamma_2(\Gamma_1(\eta))+\eta_2^0)\eqd\mathcal{M}_{B_2}^{u,\rho+\eps}(\omega),
\]
and that~$\|f_2\|_\infty\leq 1$, we obtain
\begin{equation}
\begin{split}
\label{e:lkN12part2}
\lefteqn{\IE \left[f_2\big(\mathcal{M}_{B_2}^{u,\rho+\eps}\big(\Gamma_2(\Gamma_1(\eta))+\eta_2^0\big)\big)  \big\vert(l_k)_{k=1}^{\mathcal{N}(\eta_{S_1})}\right]}\phantom{****}
\\
&\leq \IP_u^{\rho+\eps}\left(     G_N^C      \right)+
\IE \left[{\bf 1}_{G_N} \IE \left[
f_2\big(\mathcal{M}_{B_2}^{u,\rho+\eps}\big(\Gamma_2(\Gamma_1(\eta))+\eta_2^0  \big)\big)  \big\vert(l_k)_{k=1}^{\mathcal{N}(\eta_{S_1})}\right]\right]
\\
&\leq c\exp\left\{  -\useiconst{c:s1s2intens}\eps^{d-1}N      \right\}
+
\IE \left[{\bf 1}_{G_N} \IE \left[
f_2\big(\mathcal{M}_{B_2}^{u,\rho+\eps}\big(\Gamma_2(\Gamma_1(\eta))+\eta_2^0\big)\big)  \big\vert(l_k)_{k=1}^{\mathcal{N}(\eta_{S_1})}\right]\right].
\end{split}
\end{equation}

We aim to show that, on~$G_N$, with a suitably chosen~$N$, the subset of lines of~$\Gamma_2\circ\Gamma_1((l_k)_{k=1}^{\mathcal{N}(\eta_{S_1})})$ that actually intersect~$ S_2$ can be dominated by a Poisson point process of lines with distribution~$\PLP(\delta \mu)$. The idea is based on a ``poissonization'' argument: we use a Poisson process to stochastically dominate the binomial process of lines that~$\Gamma_2(\Gamma_1(\eta_{S_1}))$ conditioned on~$\eta_{S_1}$ generates on~$\Gamma_2(\eta)$. We do this in order to simplify the computations and to make the later comparison to the process with the distribution~$\PLP(\delta \mu)$ straightforward.

As in Lemma~\ref{l:theta}, we write~$\Gamma_2\circ\Gamma_1(l_k)$ to denote the resulting line after $\Gamma_2\circ\Gamma_1$ acts on~$l_k\in \eta_{S_1}$. Lemma~\ref{l:theta} implies
\begin{equation}
\label{e:mu2stochdom1}
\begin{split}
\IP\left(   p_2(   \Gamma_2\circ\Gamma_1(l_k)   )    \in   S_2   \middle \vert l_k\right) \leq
4^{d-1}\useiconst{c:intersecdens} L^{-\alpha(d-1)},
\end{split}
\end{equation}
and therefore, for sufficiently large~$L$,
\begin{equation}
\label{e:mu2stochdom2}
\begin{split}
\IP\left(   p_2(   \Gamma_2\circ\Gamma_1(l_k)   )    \in   S_2    \middle \vert l_k\right) \leq
1-\exp\left\{-2\IP\left(   p_2(   \Gamma_2\circ\Gamma_1(l_k)   )    \in   S_2    \middle \vert l_k\right)\right\}.
\end{split}
\end{equation}
Consider the random measure~$\IP^{l_k}$ such that for~$A\subset  S_2$, $B\subset D_{\eps,L}$,
\[
\IP^{l_k}(A\times B)=\IP\left(   p_2(   \Gamma_2\circ\Gamma_1(l_k)   )    \in  A,  d_2(   \Gamma_2\circ\Gamma_1(l_k)   )    \in  B \middle \vert l_k \right) .
\]
Considering each line~$l\in\LL$ to be parametrized by their intersection point with $p_2(l)\in \Pi_2$ and their direction~$d_2(l)\in D$, we can construct a Poisson point process~$\eta_{l_k}$ in~$\LL_{ S_2}$ with intensity measure~$2  \IP^{l_k}$. Furthermore, by~\eqref{e:mu2stochdom2}, we can consider~$\eta_{l_k}$ to be coupled to~$\Gamma_2\circ\Gamma_1(l_k) $ so that whenever~$\Gamma_2\circ\Gamma_1(l_k) $ intersects~$  S_2$, $\Gamma_2\circ\Gamma_1(l_k) $ is a line in~$\eta_{l_k}$. To see this, note that a line in~$\eta_{l_k}$, if one such line exists, has the same distribution as~$\Gamma_2\circ\Gamma_1(l_k) $ conditioned on intersecting~$S_2$. One can then sample~$\Gamma_2\circ\Gamma_1(l_k)\cap \LL_{ S_2} $ by first sampling~$\eta_{l_k}$, then, on the event where~$\eta_{l_k}\neq \varnothing$, selecting a line~$l_*$ in the support of~$\eta_{l_k}$ uniformly at random and letting
\begin{equation}
\nonumber
\begin{split}
\label{e:etalkcoup}
\Gamma_2\circ\Gamma_1(l_k)\cap \LL_{ S_2}=
\left\{
\begin{array}{l}
l_* \text{ with probability }
\frac{\IP\left(   p_2(   \Gamma_2\circ\Gamma_1(l_k)   )    \in   S_2    \middle \vert l_k\right)}{1-\exp\left\{-2\IP\left(   p_2(   \Gamma_2\circ\Gamma_1(l_k)   )    \in   S_2    \middle \vert l_k\right)\right\}};
\\
\varnothing \text{ with probability }
1-\frac{\IP\left(   p_2(   \Gamma_2\circ\Gamma_1(l_k)   )    \in   S_2    \middle \vert l_k\right)}{1-\exp\left\{-2\IP\left(   p_2(   \Gamma_2\circ\Gamma_1(l_k)   )    \in   S_2    \middle \vert l_k\right)\right\}}.
\end{array}
\right.
\end{split}
\end{equation}
If~$\Gamma_2\circ\Gamma_1(l_k)\cap \LL_{ S_2}=\varnothing$, we sample~$\Gamma_2\circ\Gamma_1(l_k)$ independently from~$\eta_{l_k}$ conditioned on intersecting~$\LL_{\Pi_2\setminus S_2}$. We note that we can construct the above coupling independently for each~$l_k\in\Gamma_1(\eta_{S_1})$.

By~\eqref{e:chidepsl} and~\eqref{e:intersecdens}, we obtain that uniformly over all possible collections $(l_k)_{k=1}^{\mathcal{N}(\eta_{S_1})}$, the intensity measure of the process
\[
\sum_{k=1}^{\mathcal{N}(\eta_{S_1})}\eta_{l_k}
\]
is bounded from above in~$G_N$ by
\[
\tilde c N L^{-\alpha(d-1)}  \cdot   v_{d-1}\otimes \bar\chi_{\eps,L},
\]
for some~$\tilde c>0$, where we consider~$  v_{d-1}$ to be the Lebesgue measure on the plane~$\Pi_2$. Let~$N:=\lfloor \useiconst{c:mu} \tilde c^{-1}\delta L^{\alpha(d-1)}\rfloor$. Using Lemma~\ref{l:danielplpsampling} and elementary properties of the Poisson process, we can construct a process~$\tilde \omega_\delta $ with distribution~$\PLP(\delta \mu)$ such that,on~$G_N$,
\[
(\Gamma_2\circ \Gamma_1(l_k))_{l=1}^{\mathcal{N}(\eta_{S_1})}\subset \tilde \omega_\delta\quad \text{ and }\quad (l_k)_{l=1}^{\mathcal{N}(\eta_{S_1})} \perp \tilde \omega_\delta.
\]
Given~$B,B'\subset B_2$ and~$m,m'\in\N$, we define
\[
(B,m)\oplus(B',m'):=(B\cup B,m+m').
\]
We then obtain from~\eqref{e:lkN12part2},
\begin{equation}
\begin{split}
\label{e:lkN12part3}
\lefteqn{\IE \left[f_2\big(\mathcal{M}_{B_2}^{u,\rho+\eps}\big(\Gamma_2(\eta)+\eta_2^0\big)\big)  \big\vert(l_k)_{k=1}^{\mathcal{N}(\eta_{S_1})}\right]}\phantom{*}
\\
&\leq c\exp\left\{  -c'\eps^{d-1}\delta L^{\alpha(d-1)}      \right\}
+
\IE \big[
f_2\big(\mathcal{M}_{B_2}^{u,\rho+\eps}\big(\Gamma_2(\eta)+\eta_2^0\big)\oplus\mathcal{M}_{B_2}^{\delta,\rho+\eps}\big( \tilde \omega_\delta\big)\big)  \big]
\\
&\leq c\exp\left\{  -c'\eps^{d-1}\delta L^{\alpha(d-1)}      \right\}
+
\IE \big[
f_2\big(\mathcal{M}_{B_2}^{u+\delta,\rho+\eps}\big(  \omega\big)\big)  \big],
\end{split}
\end{equation}
finishing the proof of the result.
\end{proof}

\section{3-Box decoupling}
\label{s:3box}

Theorem~\ref{thm:2boxdec} is unfortunately not strong enough for our (and possible future) applications: it requires too large a distance between the two boxes in order to be useful in multi-scale arguments.

For illustrative purposes, imagine a standard multi-scale proof with a sequence of scales~$(L_k)_{k\geq 0}$, where the occurrence of a \emph{bad event} in a box at the $(k+1)$-th scale implies the occurrence of two analogous events at the~$k$-th scale in two boxes far away from each other. Denoting by~$p_k$ the probability of the bad event at scale~$k$, one gets the general inequality after ignoring the sprinkling terms:
\[
p_{k+1}\leq    ({\rm combinatorial }\, {\rm complexity })_{k+1} \left( p_k^2 + ({\rm decoupling }\, {\rm error })_{k+1} \right) .
\]
The problem is, in order to use Theorem~\ref{thm:2boxdec}, the scales must grow very fast: we must have~$L_{k+1} \gg L_k^{2+\alpha}$. This fast growth makes the combinatorial complexity too large, outweighing the influence of the exponent~$2$ in the term~$p_k^2$. We will therefore need a stronger decoupling result relating three boxes. More than that, we shall see that we will need three~\emph{sufficiently unaligned boxes} in order to translate the arguments in Theorem~\ref{thm:2boxdec} into this new context.
This is the subject of our next result.

\begin{theorem}
\label{thm:3boxdec} \newiconst{c:3boxdec}
For~$\alpha,\eps\in(0,1)$ and~$L\in\R_+$ sufficiently large, let~$x_1,x_2,x_3\in \R^d$ that are sufficiently ``far apart'':
\begin{equation}
\label{e:3boxdechyp0}
|x_1-x_2|,|x_1-x_3|,|x_2-x_3|\geq \eps^{-1}L^{2+\alpha},
\end{equation}
and ``unaligned'':
\begin{equation}
\label{e:3boxdechyp}
\sqrt{2} \geq \dist \left( \frac{x_1-x_2}{|x_1-x_2|}, \frac{x_1-x_3}{|x_1-x_3|}\right)\geq 30\sqrt{d} \frac{\eps}{L}.
\end{equation}
Define ~$B_i:=B_\infty(x_i,L)$, for~$i=1,2,3$.
Then there exists a constant~$\useiconst{c:3boxdec}>0$ depending only on the dimension~$d$ such that, for~$\delta>0$ and increasing functions
\[
f_i: \Omega \to [0,1], \text{ measurable with respect to $\sigma(\{\mathcal{M}_{B_i}^{u,\rho}(\omega);u,\rho\in\R_+\})$}
\]
for~$i=1,2,3$, we have
\begin{equation}
\label{e:3boxdec}
\begin{split}
  \IE \big[ f_1 & \big( \mathcal{M}_{B_1}^{u, \rho}(\omega) \big) f_2 \big( \mathcal{M}_{B_2}^{u, \rho}(\omega) \big) f_3 \big( \mathcal{M}_{B_3}^{u, \rho}(\omega) \big) \big]
\\
&\leq
\IE \big[ f_1\big(\mathcal{M}_{B_1}^{u + \delta, \rho + \eps}(\omega) \big) \big]
\IE \big[ f_2\big(\mathcal{M}_{B_2}^{u + \delta, \rho + \eps}(\omega) \big) \big]
\IE \big[ f_3\big(\mathcal{M}_{B_3}^{u + \delta, \rho + \eps}(\omega) \big) \big]
\\
&\quad
+\useiconst{c:3boxdec}^{-1}\exp\big\{     -\useiconst{c:3boxdec}   \delta\eps^{d-1}L^{\alpha(d-1)}      \big\}.
\end{split}
\end{equation}
An analogous theorem is also valid for decreasing events.
\end{theorem}
\begin{remark}
The~$\sqrt{2}$ upper bound in~\eqref{e:3boxdechyp} is not too restricting: in the triangle formed by the points~$x_1,x_2,x_3$ there exists at least one acute angle.
\end{remark}

In order to prove the above result, we will show that conditions~\eqref{e:3boxdechyp0} and~\eqref{e:3boxdechyp} imply the existence of two pairs of boxes, one covering~$B_1$ and~$B_2$, the other covering~$B_1$ and~$B_3$, such that a coupling construction analogous to the one in Proposition~\ref{p:sprinkle} works, and such that the line sets involved in this construction are disjoint, which makes the associated Poisson line processes independent, see Figure~\ref{f:3boxdec}.

From now on we will assume~$x_1,x_2,x_3\in\R^d$ fixed and satisfying~\eqref{e:3boxdechyp0} and~\eqref{e:3boxdechyp}. Define the unit vectors
\begin{equation}
\label{e:v12v13def}
v_{12}:=\frac{x_2-x_1}{|x_2-x_1|},\quad v_{13}:=\frac{x_3-x_1}{|x_3-x_1|}.
\end{equation}
These vectors will play the role which the ``north pole'' ${\bf e}_d$ played in Section~\ref{s:2box}. With that in mind, we fix two rotations~$\mathcal{R}_{12}$ and~$\mathcal{R}_{13}$ which bring~${\bf e}_d$ to~$v_{12}$ and~$v_{13}$ respectively. Letting~$\tau_x$ denote the translation by~$x\in\R^d$ in~$\R^d$, we define the rotated boxes
\begin{equation}
\label{e:B12def}
\begin{split}
B_{12}:=\tau_{x_1} \mathcal{R}_{12} \tau_{x_1}^{-1} (B_\infty(x_1,2\sqrt{d}L)), &\quad
B_{13}:=\tau_{x_1} \mathcal{R}_{13} \tau_{x_1}^{-1} (B_\infty(x_1,2\sqrt{d}L)),
\\
B_{22}:=\tau_{x_2} \mathcal{R}_{12} \tau_{x_2}^{-1} (B_\infty(x_2,2\sqrt{d}L)), &\quad
B_{33}:=\tau_{x_3} \mathcal{R}_{13} \tau_{x_3}^{-1} (B_\infty(x_3,2\sqrt{d}L)),
\end{split}
\end{equation}
and we note that
\begin{equation}
  \label{e:B12contain}
  B_{12}\supset B_1, \quad
  B_{13}\supset B_1, \quad
  B_{22}\supset B_2 \quad \text{and} \quad
  B_{33}\supset B_3.
\end{equation}
\begin{figure}[ht]
	\centering
	\includegraphics[scale = 1]{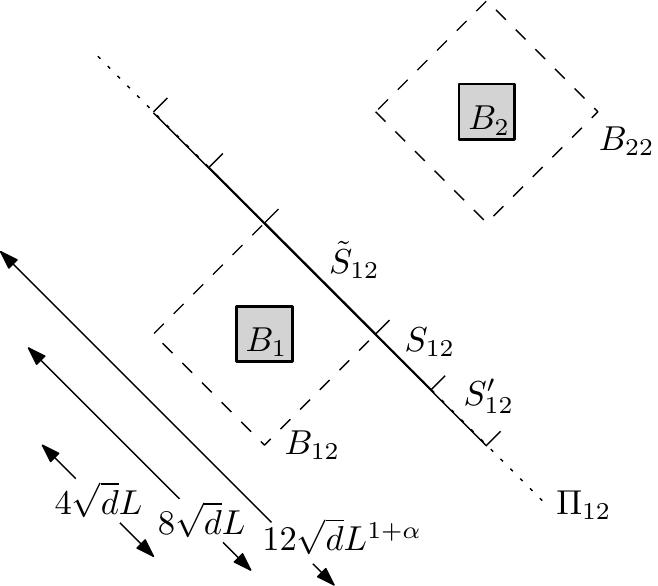}
	\vspace{0.1cm}
	\caption{An illustration (not to scale) showing various sets defined for the proof of Theorem~\ref{thm:3boxdec}.}
	\label{f:s12def}
\end{figure}

Denote by~$\Pi_{12}$ and~$\Pi_{22}$ the hyperplanes orthogonal to~$v_{12}$ containing the respective hyperfaces of~$B_{12}$ and~$B_{22}$ which are closest to each other, denote these faces by~$\tilde S_{12}$ and~$\tilde S_{22}$ respectively. Define $ S_{12}'$ and $S_{22}'$, as~$(d-1)$-dimensional $\ell_\infty$-boxes with radii~$6\sqrt{d}L^{1+\alpha}$ containing respectively  $\tilde S_{12}$ and $\tilde S_{22}$, and having also the same respective barycenters. Denote also by~$ S_{12}$ and $ S_{22}$ the~$(d-1)$-dimensional $\ell_\infty$-boxes with radii~$4\sqrt{d}L$ containing respectively~$\tilde S_{12}$ and $\tilde S_{22}$ and also with same centers of mass. Analogously define~$\Pi_{13}$ and~$\Pi_{33}$, $\tilde S_{13}$ and~$\tilde S_{33}$, $ S_{13}'$ and $S_{33}'$, and $ S_{13}$ and $ S_{33}$. For~$\lambda\in\{12,13,22,33\}$, we have
\begin{equation}
\label{e:Slambdacontain}
\begin{split}
\tilde S_\lambda \subset S_\lambda \subset  S_\lambda' \subset  \Pi_\lambda .
\end{split}
\end{equation}
We refer to Figure~\ref{f:s12def} for clarification. We also define the rotated spherical caps
\begin{equation}
\label{eq:D12def}
\begin{split}
\mathbb{D}_{22}=\mathbb{D}_{12}:=\mathcal{R}_{12}(\mathbb{D}),\quad
\mathbb{D}_{33}=\mathbb{D}_{13}:=\mathcal{R}_{13}(\mathbb{D}).
\end{split}
\end{equation}
Since the radii of the boxes considered was increased, the size of the ``north pole neighborhoods'' must be decreased so that a result analogous to Lemma~\ref{l:bend} may hold. With that in mind, we define
\begin{equation}
\label{e:depstildeldef}
\tilde{D}_{\eps,L}:=\left\{ x \in \mathbb{D}; \dist(x,{\bf e}_d)< \eps(16\sqrt{d}L)^{-1}      \right\},
\end{equation}
as well as
\begin{equation}
\label{eq:D12tildef}
\begin{split}
D_{\eps,L}^{12}=D_{\eps,L}^{12}:=\mathcal{R}_{12}(\tilde{D}_{\eps,L}),\quad
D_{\eps,L}^{13}=D_{\eps,L}^{13}:=\mathcal{R}_{13}(\tilde{D}_{\eps,L}).
\end{split}
\end{equation}

\begin{figure}[ht]
	\centering
	\includegraphics[scale = 1]{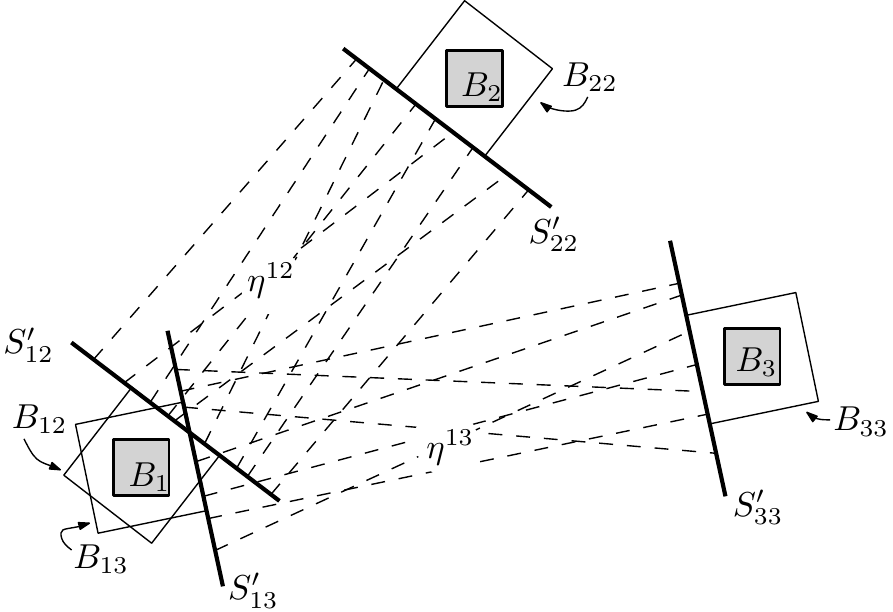}
	\vspace{0.1cm}
	\caption{A schematic showing the sets involved in the proof of Theorem~\ref{thm:3boxdec}.}
	\label{f:3boxdec}
\end{figure}

For~$\lambda=12,13,22,33$, we now characterize (except in a zero $\mu$-measure set) a line~$l\in\LL$ by~$p_\lambda(l)$, its intersection point with~$\Pi_\lambda$, and its direction~$d_\lambda (l)\in\mathbb{D}_\lambda$. Again, a result in the manner of Lemma~\ref{l:danielplpsampling} holds, where we can sample~$\PLP(u\mu)$ starting with a Poisson point process in the above planes instead of in~$\R^{d-1}\times\{0\}$.

We define the ``harmless'' Poisson line processes
\begin{equation}
\begin{split}
\eta_1^0 &=
\sum_{i \geq 0,u_i\leq u} \delta_{(l_i,u_i)}{\bf 1}\{
l_i\cap S_1\neq \varnothing;
d_{12}(l_i)\notin D_{\eps,L}^{12}  ; d_{13}(l_i)\notin D_{\eps,L}^{13}   \},
\\
\eta_2^0 &=
\sum_{i \geq 0,u_i\leq u} \delta_{(l_i,u_i)}{\bf 1}\{
l_i\cap S_2\neq \varnothing;
d_{12}(l_i)\notin D_{\eps,L}^{12}  \}
,
\\
\eta_3^0 &=
\sum_{i \geq 0,u_i\leq u} \delta_{(l_i,u_i)}{\bf 1}\{
l_i\cap S_3\neq \varnothing;
d_{13}(l_i)\notin D_{\eps,L}^{13}  \}
,
\end{split}
\end{equation}
as well as the processes which can ``carry information'' between the pairs of boxes
\begin{equation}
\begin{split}
\eta^{12}& =\sum_{i \geq 0,u_i\leq u} \delta_{(l_i,u_i)}{\bf 1} \{  d_{12}(l_i) \in D_{\eps,L}^{12} \text{ and either }      p_{12}(l)\in S_{12}' \text{ or }    p_2(l)\in  S_2   '  \},
\\
\eta^{13}& =\sum_{i \geq 0,u_i\leq u} \delta_{(l_i,u_i)}{\bf 1} \{  d_{13}(l_i) \in D_{\eps,L}^{13} \text{ and either }      p_{13}(l)\in S_{13}' \text{ or }    p_3(l)\in  S_3  '  \},
\end{split}
\end{equation}
see Figure~\ref{f:3boxdec} for an illustration depicting the two last point measures. What is crucial for the proof of Theorem~\ref{thm:3boxdec} is the already advertised fact that~$\eta^{13}$ and~$\eta^{12}$ are independent line processes, and therefore a coupling construction like the one of Proposition~\ref{p:sprinkle} can be done simultaneously for the two processes. The next lemma rigorously states this result.

\begin{lemma}
\label{l:unalignedboxes}
Using the notation above defined, we have that, for sufficiently large~$L$, $\eta^{13}$ and~$\eta^{12}$ are independent Poisson line processes.
\end{lemma}

\begin{proof}
We will show that lines intersecting both~$S_{13}'$ and~$S_{33}'$ cannot have its direction in~$D_{\eps,L}^{12}$, which will show the result by elementary properties of the Poisson process. With that in mind, consider~$y_3\in S_{33}'$ and~$y_1\in S_{13}'$. By the Pythagorean Theorem, we have that, for large enough~$L$, there exist vectors~$w_1,w_3\in\R^d$ such that
\begin{equation}
\label{eq:unnalignedboxes1}
\begin{split}
y_1=x_1+w_1,\, y_3=x_3+w_3,\text{ and }|w_1|,|w_3| \leq 7\sqrt{d} L^{1+\alpha}.
\end{split}
\end{equation}
We then have
\begin{equation}
\label{eq:unnalignedboxes2}
\begin{split}
\frac{y_3-y_1}{|y_3-y_1|}
&=
\left(\frac{x_3-x_1}{|x_3-x_1|}+\frac{w_3-w_1}{|x_3-x_1|}\right)
\left(\frac{|x_3-x_1 +w_3-w_1|}{|x_3-x_1|}\right)^{-1},
\end{split}
\end{equation}
and for large enough~$L$, by the triangle inequality,
\begin{equation}
  \label{eq:unnalignedboxes3}
  \begin{split}
    \Big| \frac{x_3-x_1}{|x_3-x_1|} & \left(\frac{|x_3-x_1 +w_3-w_1|}{|x_3-x_1|}\right)^{-1}-\frac{x_3-x_1}{|x_3-x_1|} \Big|
    =
    \Big| 1-\left(\frac{|x_3-x_1 +w_3-w_1|}{|x_3-x_1|}\right)^{-1} \Big|
    \\
    &\leq
    15\sqrt{d} \frac{\eps}{L}.
  \end{split}
\end{equation}
Using again the triangular inequality and Equation~\eqref{eq:unnalignedboxes2}, we obtain, for large enough~$L$,
\begin{equation}
\label{eq:unnalignedboxes4}
\begin{split}
\left|\frac{y_3-y_1}{|y_3-y_1|}-\frac{x_3-x_1}{|x_3-x_1|}\right|
&\leq
15\sqrt{d} \frac{\eps}{L}+14\sqrt{d} \frac{\eps}{L}
\left(\frac{|x_3-x_1 +w_3-w_1|}{|x_3-x_1|}\right)^{-1}
< 30\sqrt{d} \frac{\eps}{L}.
\end{split}
\end{equation}
Now, the definition of~$D_{\eps,L}$ in~\eqref{e:depsldef}, the definition of~$D_{\eps,L}^{12}$ in~\eqref{eq:D12def}, the hypothesis~\eqref{e:3boxdechyp}, and the triangular inequality show that
\begin{equation}
\label{eq:unnalignedboxes5}
\begin{split}
\frac{y_3-y_1}{|y_3-y_1|}\notin
D_{\eps,L}^{12},
\end{split}
\end{equation}
finishing the proof of the result.
\end{proof}

Let~$\tilde \chi_{\eps,L}$ denote the measure~$\chi_{\eps,L}$ conditioned on selecting a direction in~$\tilde D_{\eps,L}$, and let~$\bar \chi_{\eps,L}^{12}$ and~$\bar \chi_{\eps,L}^{13}$ denote respectively the pushforward of the measure~$\tilde \chi_{\eps,L}$ by the rotations~$\mathcal{R}_{12}$ and~$\mathcal{R}_{13}$. For $\lambda=12,22$, we define direction re-sampling operations in the same manner of~\eqref{eq:gammadef}:
\begin{equation}
\label{eq:gamma12def}
\begin{array}{cclc}
\Gamma_{\lambda}^{12}: &\eta^{12} & \to& \Gamma_{\lambda}^{12}(\eta^{12})\\
&(p_\lambda(l),d_\lambda(l)) & \mapsto& (p_\lambda(l), d_\lambda'(l)) ,
\end{array}
\end{equation}
where~$ d_\lambda'(l)$ is defined to be either a random vector in~$D_{\eps,L}^{12}$ sampled according to~$\bar \chi_{\eps,L}^{12}$ independently for each~$l\in\eta^{12}$ if~$p_\lambda(l)\in S_{\lambda}'$, or simply equal to~$d_\lambda(l)$ otherwise. We analogously define~$\Gamma_{13}^{13}$ and~$\Gamma_{33}^{13}$.

In the manner of~(\ref{e:etas1}) and~(\ref{e:etas1g}) we define, for~$i=2,3$,
\begin{equation}
\label{e:etas1i}
\begin{split}
\eta^{1i}_{S_{1i}}&:=\sum_{(l_j,u_j)\in\eta^{1i}} \delta_{(l_j,u_j)}{\bf 1} \{  p_{1i}(l_j)\in S_{1i} ;    d_{1i}(l_j)\in D_{\eps,L}      \};
\\
\eta^{1i}_{S_{1i}'\setminus S_{1i}}&:=\eta^{1i}-\eta^{1i}_{S_{1i}},
\end{split}
\end{equation}
as well as
\begin{equation}
\label{e:etas1ig}
\begin{split}
\Gamma^{1i}_{1i}(\eta^{1i}_{S_{1i}})&:=\sum_{(l_j,u_j)\in\Gamma^{1i}_{1i}(\eta^{1i})} \delta_{(l_j,u_j)}{\bf 1} \{  p_{1i}(l_j)\in S_{1i} ;    d_{1i}'(l_j)\in D_{\eps,L}      \};
\\
\Gamma^{1i}_{1i}(\eta^{1i}_{S_{1i}'\setminus S_{1i}})&:=\Gamma^{1i}_{1i}(\eta^{1i})-\Gamma^{1i}_{1i}(\eta^{1i}_{S_{1i}}).
\end{split}
\end{equation}
As in \eqref{eq:gammarevers}, we have, by definition, detailed balance equations for these stochasic operations:
\begin{equation}
  \label{eq:gammalambdarevers}
  \begin{split}
  \big( \eta^{12}, \Gamma_{\lambda}^{12}( \eta^{12} ) \big)
    &\eqd
      \big( \Gamma_{\lambda}^{12}( \eta^{12} ), \eta^{12} \big), \qquad
  \big( \eta^{12}, \Gamma_{\lambda}^{22}( \eta^{12} ) \big)
    \eqd
      \big( \Gamma_{\lambda}^{22}( \eta^{12} ), \eta^{12} \big),
  \\
  \big( \eta^{13}, \Gamma_{\lambda}^{13}( \eta^{13} ) \big)
    &\eqd
      \big( \Gamma_{\lambda}^{13}( \eta^{13} ), \eta^{13} \big), \qquad
  \big( \eta^{13}, \Gamma_{\lambda}^{33}( \eta^{13} ) \big)
    \eqd
      \big( \Gamma_{\lambda}^{33}( \eta^{13} ), \eta^{13} \big).
  \end{split}
\end{equation}

The following lemmas are analogous to lemmas~\ref{l:bend}, \ref{l:s1s2intens}, and~\ref{l:theta}, and they are proved in the same way.

\begin{lemma}
  \label{l:bend3box}
  With the notation above developed we have, for large~$L$,
  \begin{equation}
    \nonumber
    \label{eq:lemmabend3box}
    \begin{split}
&\mathcal{M}_{B_{12}}^{u,\rho-\eps}\big(\Gamma^{12}_{12}(\eta^{12})+\eta_1^0\big)\preceq \mathcal{M}_{B_{12}}^{u,\rho}\big(\eta^{12}+\eta_1^0\big)\preceq \mathcal{M}_{B_{12}}^{u,\rho+\eps}\big(\Gamma^{12}_{12}(\eta^{12})+\eta_1^0\big),
      \\
&\mathcal{M}_{B_{13}}^{u,\rho-\eps}\big(\Gamma^{13}_{13}(\eta^{13})+\eta_1^0\big)\preceq \mathcal{M}_{B_{13}}^{u,\rho}\big(\eta^{13}+\eta_1^0\big)\preceq \mathcal{M}_{B_{13}}^{u,\rho+\eps}\big(\Gamma^{13}_{13}(\eta^{13})+\eta_1^0\big),
      \\
&\mathcal{M}_{B_{22}}^{u,\rho-\eps}\big(\Gamma^{12}_{22}(\eta^{12})+\eta_2^0\big)\preceq \mathcal{M}_{B_{22}}^{u,\rho}\big(\eta^{12}+\eta_2^0\big) \preceq \mathcal{M}_{B_{22}}^{u,\rho+\eps}\big(\Gamma^{12}_{22}(\eta^{12})+\eta_2^0\big),
      \\
&\mathcal{M}_{B_{33}}^{u,\rho-\eps}\big(\Gamma^{13}_{33}(\eta^{13})+\eta_3^0\big)\preceq \mathcal{M}_{B_{33}}^{u,\rho}\big(\eta^{13}+\eta_3^0\big) \preceq \mathcal{M}_{B_{33}}^{u,\rho+\eps}\big(\Gamma^{13}_{33}(\eta^{13})+\eta_3^0\big).
    \end{split}
  \end{equation}
\end{lemma}
\newiconst{c:s1s2intens3box}

\begin{lemma}
  \label{l:s1s2intens3box}
  Denote by~$\tilde\mu^{12}$ and~$\tilde\mu^{13}$ the pushforward of the measure~$\tilde \mu$ by the rotations~$\mathcal{R}_{12}$ and~$\mathcal{R}_{13}$ respectively. There exists a constant~$\useiconst{c:s1s2intens3box}>0$ such that, for~$i=2,3$,
  \begin{equation}
    \label{e:s1s2intens3box}
    \tilde\mu^{1i}\left(  S_{1i} \times D_{\eps,L}^{1i}\right) = \useiconst{c:s1s2intens3box} \eps^{d-1}\Big(1-\frac{\eps^2}{256L^2}\Big)^{\frac{d-1}{2}}.
  \end{equation}
\end{lemma}
\newiconst{c:intersecdens3box}
\begin{lemma}
  \label{l:theta3box}
  For~$i=2,3$, consider~$l\in \eta^{1i}_{S_{1i}}$. Denote by~$\Gamma_{ii}^{1i}\circ\Gamma_{1i}^{1i}(l)$ the line in~$\Gamma_{ii}^{1i}(\eta^{1i})$ corresponding to~$\Gamma_{1i}^{1i}(l)$ in~$\eta_{S_{1i}}^{1i}$. There exists a constant~$\useiconst{c:intersecdens3box}>0$ such that for every~$p\in \Pi_{1i}$,~$d\in D_{\eps,L}^{1i}$ and sufficiently large~$L$,
  \begin{equation}
    \nonumber
    \label{e:intersecdens3box}
    \begin{split}
      \lefteqn{\IP \left(  p_{ii}(\Gamma_{ii}^{1i}\circ \Gamma_{1i}^{1i} (l)) \in A   , d_{ii}(\Gamma_{ii}^{1i}\circ \Gamma_{1i}^{1i}(l)) \in B \middle \vert p_{1i}(l)=p,d_{1i}'(l)=d \right)}\phantom{****************}\phantom{************}
      \\
      &\leq \useiconst{c:intersecdens3box}L^{-(1+\alpha)(d-1)}  \int {\bf 1}_A \d v_{d-1}\cdot  \bar\chi_{\eps,L}^{1i}(B),
    \end{split}
  \end{equation}
  for every Borelian subsets~$A\subseteq \Pi_{ii}$, $B\subseteq D_{\eps,L}^{1i}$.
\end{lemma}

\begin{proof}[Proof of Theorem~\ref{thm:3boxdec}]

Note that, for~$L$ large enough and~$i=2,3$, in order for a line with direction in~$D_{\eps,L}^{1i}$ to intersect~$B_1$, it has to intersect also~$ S_{1i}$. We obtain, in the manner of~\eqref{e:decorr1}, using Lemma~\ref{l:bend3box} and the monotonicity of the functions being considered,
\begin{align}
\label{e:decorr3}
\lefteqn{\IE\big[f_1\big(\mathcal{M}_{B_{1}}^{u,\rho}(\omega)\big) f_2\big(\mathcal{M}_{B_{2}}^{u,\rho}(\omega)\big)  f_3\big(\mathcal{M}_{B_{3}}^{u,\rho}(\omega)\big)    \big] }
\\
&\leq \IE \left[
\begin{array}{l}
f_1\big(\mathcal{M}_{B_{1}}^{u,\rho+\eps}\big(\Gamma_{12}^{12}(\eta^{12}_{S_{12}})+\Gamma_{13}^{13}(\eta^{13}_{S_{13}})+\eta_{1}^0\big)\big)
\\
\quad \times \IE \left[
\begin{array}{l}
f_2\big(\mathcal{M}_{B_{22}}^{u,\rho+\eps}\big(\Gamma_{22}^{12}(\eta^{12})+\eta_{2}^0\big)\big)
\\
\quad\times f_3\big(\mathcal{M}_{B_{33}}^{u,\rho+\eps}\big(\Gamma_{33}^{13}(\eta^{13})+\eta_{3}^0\big)\big)
\end{array}
\middle \vert \Gamma_{12}^{12}(\eta^{12}_{S_{12}}),\Gamma_{13}^{13}(\eta^{13}_{S_{13}}),\eta_1^0 \right]
\end{array}
\right].
 \nonumber
\end{align}
Using lemmas~\ref{l:unalignedboxes} and~\ref{l:theta3box}, we can construct two couplings, analogous to the one in Proposition~\ref{p:sprinkle}, simultaneously and independently. In this way we obtain a coupling between~$\Gamma^{12}_{22}\circ\Gamma^{12}_{12} (\eta^{12}_{S_{12}})$ conditioned on~$\eta^{12}_{S_{12}}$, $\Gamma^{13}_{33}\circ\Gamma^{13}_{13}(\eta^{13}_{S_{13}})$ conditioned on~$\eta^{13}_{S_{13}}$, and a process~$\omega_\delta\eqd \PLP(\delta \mu)$ independent from~$\eta^{12}_{S_{12}}$ and~$\eta^{13}_{S_{13}}$ such that, whenever the number of lines in~$\eta^{12}_{S_{12}}$ and~$\eta^{13}_{S_{13}}$ is not too large,
\[
\left(\Gamma^{12}_{22}\circ\Gamma^{12}_{12}(\eta^{12}_{S_{12}})\right)\cup \left(\Gamma^{13}_{33}\circ\Gamma^{13}_{13}(\eta^{13}_{S_{13}})\right)\subseteq \omega_\delta,
\]
where we identified the point measures with their supports in~$\LL$. This implies, by the same reasoning as in~\eqref{e:lkN12part2} and~\eqref{e:lkN12part3}, as well as the reversibility equations in \eqref{eq:gammalambdarevers},
\begin{equation}
\label{e:decorr4}
\begin{split}
\lefteqn{\IE\big[f_1\big(\mathcal{M}_{B_{1}}^{u,\rho}(\omega)\big) f_2\big(\mathcal{M}_{B_{2}}^{u,\rho}(\omega)\big)  f_3\big(\mathcal{M}_{B_{3}}^{u,\rho}(\omega)\big)    \big] }\phantom{******}
\\
&\leq
\IE\big[f_1\big((\mathcal{M}_{B_{1}}^{u,\rho + \eps}(\omega)\big) \big]
\IE\big[ f_2\big((\mathcal{M}_{B_{22}}^{u+\delta,\rho+\eps}(\omega)\big)  f_3\big(\mathcal{M}_{B_{33}}^{u+\delta,\rho+\eps}(\omega)\big)    \big]
\\
&\quad +c^{-1}\exp\big\{     -c  \delta\eps^{d-1}L^{\alpha(d-1)}      \big\}.
\end{split}
\end{equation}
Now applying Theorem~\ref{e:2boxdec} to the expectation of the product in the above right hand side, considering slightly larger boxes in order for them to be parallel, we obtain~\eqref{e:3boxdec} after substituting~$2\eps$ by~$\eps$ and~$2\delta$ by~$\delta$.
\end{proof}

\section{Renormalization strategy}
\label{s:renorm}

In this section we describe how one can use the decoupling inequality obtained in Theorem~\ref{thm:3boxdec} in order to prove results about the vacant set of the cylinder percolation process for small intensities of the parameter~$u$.
The idea is to use multi-scale renormalization to prove that with high probability there exists a fractal-like `carpet' where the percolation process is well behaved.
We start with the necessary definitions of scales in our renormalization scheme.

Throughout the next sections, the $0$-th scale, denoted by $L_0$, will play an important role: we will need to choose it to be sufficiently large in order for the statements that follow to hold. We will therefore consider a constant $\clzero > 0$, which will ultimately depend only on the dimension~$d$, but whose value will be updated as needed in $(\ref{eq:0boxdecay}, \ref{e:L_k_large}, \ref{e:khole2}, \ref{eq:mathsfGnonempty}, \ref{eq:0scale_path_isop}, \ref{eq:flow0_1})$, and we will take $L_0 > \clzero$. We let
\begin{equation}
\label{eq:alphafix}
\alpha \in \left(  1-\frac{\gamma}{2(d-1)}, 1 \right)
\end{equation}
and~$\beta\in(0,1-\alpha)$. We then define the sequence of growing scales
\begin{equation}
\label{eq:Lkdef}
L_k:= 17 \left( k^2 \cdot 2L_{k-1}^2 \lceil  L_{k-1}^{\alpha+\beta}   \rceil + L_{k-1}\right),
\end{equation}
for~$k\in\N$.

Despite the above definition looking involved, it is a simple choice that guarantees that the scales $L_k$ will satisfy the following properties:
\begin{enumerate}
\item $L_k \in 17 \N$ for every~$k \in \N$, as some of our arguments divide boxes $B(x_k, L_k)$ into boxes of radius $17^{-1} L_k$;
\item $L_k$ is roughly of order $L_{k-1}^{\alpha + \beta + 2}$;
\item $L_k$ is divisible by~$L_{k-1}$, but is not divisible by~$2L_{k-1}$, which we will need in order to partition the faces of boxes at scale $k$ into faces of boxes at scale $k - 1$. 
\end{enumerate}
Having defined the scales, we introduce, for each~$k\in \N$, the coarse-grained lattices
\begin{equation}
\label{eq:Mxdef}
\mathbb{M}_{k}=\mathbb{M}_k(L_0):=2L_k\Z^d \times \{k\} .
\end{equation}
If~$m=(x,k)\in\mathbb{M}_k$, we write
\begin{equation}
\label{eq:Bmdef}
B_m:=B_{\infty}\left( x, L_k \right),
\end{equation}
and we call~$B_m$ a \emph{box of the $k$-th scale}. We will sometimes abuse the notation and refer to~$m$ directly as a box. It will be crucial for us that the boxes $(B_m)_{m \in \mathbb{M}_k}$ for a gien $k \geq 0$ are not disjoint, the box $B_m$ will share faces with its neighboring boxes of the same scale.

During the renormalization argument, both the intensity of the cylinder's process, as well as the radius of our cylinders will vary from scale to scale.
This will allow us to use our decoupling result when relating probabilities of bad events in different scales.

To introduce these sequences, fix some~$\gamma \in (0, 1/5)$.
Given $L_0$ as above, we define the initial intensity $\tilde{u}$ and radius $\tilde{\rho}$ as
\begin{equation}
\label{eq:u0def}
\tilde u=\tilde u(\gamma,L_0):=\frac{1}{L_0^{d-1-\frac{\gamma}{2}}},\quad \tilde \rho:=2.
\end{equation}
The denisity $\tilde u$ is chosen such that w.h.p.\ at most $L_0^\gamma$ cylinders actually intersect the boxes at the $0$-th scale, as we will see in~\eqref{eq:0boxdecay}. For $k\in\N$, we then define
\begin{equation}
\label{eq:ukdef}
u_k=u_k(\gamma,L_0):=\tilde u\cdot \left(  1- \frac{1}{k+2}  \right),\quad \rho_k:=2\left(  1- \frac{1}{k+2}  \right).
\end{equation}

We can now define \emph{good} and \emph{bad} boxes in different scales.
For the first scale, we simply control the number of cylinders intersecting the box:

\begin{definition}
  \label{def:0badbox}
  Given~$m\in\mathbb{M}_0$, we say that the box~$B_m$ is $(u,\rho,0)$-bad (or simply bad) for~$\omega$ if the number of cylinders of radius~$\rho$ at level~$u$ of~$\omega$ intersecting $B_m$ is larger than~$ L_0^{\gamma}$.
\end{definition}

For other values of scale $k$ we will introduce the notion of bad box inductively.
Roughly speaking, we will say that a box is bad if it has at least three bad sub-boxes that are well separated and not aligned.
The requirements are inspired by the decoupling of three boxes introduced in Section~\ref{s:3box}.

\begin{definition}
  \label{def:kbadbox}
  Given~$k\in\N$ and~$m\in\mathbb{M}_k$, we say that the box~$B_m$ is $(u,\rho,k)$-bad (or simply bad) for~$\omega$ if there exist~$m_1,m_2,m_3\in \mathbb{M}_{k-1}$, represented respectively by~$(x_1,k),(x_2,k),(x_3,k)$ such that
  \begin{itemize}
  \item [(i)] $B_{m_i}\subset B_m$ for~$i=1,2,3$;
  \item [(ii)] $ \displaystyle
    |x_1-x_2|,|x_1-x_3|,|x_2-x_3|\geq k^2 \cdot   L_{k-1}^{2+\alpha}
    $;
  \item [(iii)]
    $ \displaystyle
    \sqrt{2} \geq \dist \left( \frac{x_1-x_2}{|x_1-x_2|}, \frac{x_1-x_3}{|x_1-x_3|}\right)\geq 30\sqrt{d} \frac{1}{k^2 L_{k-1}}
    $.
  \item[(iv)]$B_{m_1},B_{m_2},B_{m_3}$ are~$(u,\rho,k-1)$-bad for~$\omega$.
  \end{itemize}
\end{definition}

Given~$m\in \mathbb{M}_k$, we say that~$m$ is $(u,\rho,k)$-good (or simply good) if it is not~$(u,\rho,k)$-bad. We note that the event where~$m$ is bad for~$\omega$ is increasing.

\bigskip

In what follows we will show that for appropriate choices of parameters, the probability that a box is bad decays fast with the scale.
First consider the probabilities
\begin{equation}
\label{eq:pkdef}
p_k(u,\rho) := \sup_{m = (x, k) \in \mathbb{M}_k} \IP\left[ (x,k) \text{ is } (u, \rho, k) \text{-bad} \right] = \IP \left[ (0,k) \text{ is } (u, \rho, k) \text{-bad} \right] .
\end{equation}
We want to estimate the probabilities $p_k(u_k, \rho_k)$, starting from the initial scale. Note that the boxes at smaller scales \emph{will use $u_k$ and $\rho_k$ as parameters}. For example, $p_k(u_k, \rho_k)$ heuristically is the probability that there are three well separated and unaligned boxes at scale $k$ which are $(u_k, \rho_k, k-1)$-bad and contained in a specific box at scale $k$.

By Lemma~$(2.2)$ of~\cite{TW10b}, the number of cylinders of radius~$2$ intersecting a $0$-box is Poisson distributed with parameter bounded from above by $\uc{c:capacity} u (L_0+2)^{d-1}$. For~$\clzero$ large depending on $d$, ~$L_0 > \clzero$, and~$(x,0)\in\mathbb{M}_0$, we obtain, using the definition of~$u_0$ and of the Poisson distribution,
\nc{c:capacity}
\begin{equation}
\label{eq:0boxdecay}
p_0(u_0, \rho_0) \leq \exp\left\{  - \uc{c:capacity} \cdot L_0^\gamma \right \},
\end{equation}
where in the last equality we used the translation invariance of the cylinder's process.

We turn now to the estimate of $p_k$ for every $k$, which is obtained by induction.
First we use the decoupling of three boxes provided by Theorem~\ref{thm:3boxdec} and the stationarity of the cylinder process under translations, to obtain, for~$k\in\N$,
\begin{equation}
\label{eq:pkdec}
\begin{split}
\lefteqn{p_{k}(u_{k},\rho_{k})}\,
\\
&\leq \!\!\!\!\! \!\!\!\!\!\bigcup_{\substack{m_1,m_2,m_3\in \mathbb{M}_{k-1},\\B_{m_1},B_{m_2},B_{m_3}\text{satisfy}\\\text{(i,ii,iii) in Definition \ref{def:kbadbox}}}}\!\!\!\!\!\!\!\!\!\!\!
\IP\left(  (m_i,k-1)\text{ is }(u_{k},\rho_k,k-1)\text{-bad}  \text{ for }i=1,2,3  \right)
\\
&\leq \left(\frac{L_{k}}{L_{k-1}}\right)^{3d} \left( p_{k-1}(u_{k-1},\rho_{k-1})^3   + c\exp\big\{  -   c \tilde{u} k^{-2}\cdot k^{-2(d-1)}\cdot L_{k-1}^{\alpha(d-1)}      \big\}\right)
\\
&\leq ck^{6d} L^{(1+\alpha+\beta)3d}_{k-1} \left( p_{k-1}(u_{k-1},\rho_{k-1})^3   + c\exp\left\{  -   c \frac{1}{L_0^{d-1-\frac{\gamma}{2}}} \cdot k^{-2d}\cdot L_{k-1}^{\alpha(d-1)}      \right\}\right)
\\
&\leq ck^{6d} L^{(1+\alpha+\beta)3d}_{k-1} \left( p_{k-1}(u_{k-1},\rho_{k-1})^3   + c\exp\left\{ -    c \cdot k^{-2d}\cdot L_{k-1}^{\frac{\gamma}{2}-(d-1)(1-\alpha)}      \right\}\right)
\end{split}
\end{equation}

These equations allow us to prove our next result.
\begin{proposition}
  \label{p:pkind}
  There exists~$\delta>0$ such that, for $L_0 > \clzero$, with the notation above introduced, we have, for every~$k\geq 0$,
  \begin{equation}
    \label{eq:pkind}
    \begin{split}
      p_{k}(u_{k},\rho_{k})\leq \exp\left\{  -(\log L_k)^{1+\delta}         \right\}.
    \end{split}
  \end{equation}
\end{proposition}

\begin{proof}
We prove Equation~\eqref{eq:pkind} by induction, as it is usual in such arguments. We note that the base case~$k=0$ follows as a direct consequence of~\eqref{eq:0boxdecay}. Assume then that~\eqref{eq:pkind} is valid for~$k-1$, with~$k\in\N$. Since~$(L_k)_{k\geq 0}$ grows faster than an exponential sequence with base~$L_0$, we obtain from the definition of~$\alpha$ in~\eqref{eq:alphafix} that, after possibly increasing $\clzero$, for all~$k\in\N$,
\begin{equation}
  \label{e:L_k_large}
  \exp\left\{  -3(\log L_{k-1})^{1+\delta}         \right\} \geq c\exp\left\{   -  c \cdot k^{-2d}\cdot L_{k-1}^{\frac{\gamma}{2}-(d-1)(1-\alpha)}      \right\}
\end{equation}
Equation~\eqref{eq:pkdec} then implies
\begin{equation}
\nonumber
\begin{split}
\lefteqn{p_{k}(u_{k},\rho_{k}) \exp\left\{  (\log L_k)^{1+\delta}         \right\}}
\\
&\leq  \exp\left\{  ((2+\alpha+\beta)\log L_{k-1}+2\log k+ c')^{1+\delta}    -3(\log L_{k-1})^{1+\delta}        \right\}
 ck^{6d} L^{3d(1+\alpha+\beta)}_{k-1} ,
\end{split}
\end{equation}
which, by the definition of~$\alpha$ and~$\beta$, is smaller than~$1$ for sufficiently small~$\delta$ and every~$L_0$ sufficiently large. Note that~$\delta$ does not depend on~$L_0$, as long as~$\clzero$ is sufficiently large. This finishes the induction argument, and the proof of the result.
\end{proof}

We now show that whenever a box of the~$k$-th scale is good, it will contain a fractal-like structure of boxes of all smaller scales. This structure will have nice connectivity properties we will explore in the upcoming sections. We first introduce a new notation to encode where the possible ``defects'' inside a good box may lie, and then state and prove a related geometric lemma.

\begin{definition}
  \label{def:khole}
  Given~$m\in\mathbb{M}_k$ with~$k\in\N$ and $\ell\in\LL$, we define~$\mathcal{D}_m(\ell)$ to be the set of boxes~$m'\in\mathbb{M}_{k-1}$ such that
  \begin{itemize}
  \item [(i)] $\displaystyle B_{m'}\subset B_m$;
  \item [(ii)] $\displaystyle \dist(B_{m'},\ell) \leq 2k^2L_{k-1}^{2 +\alpha}$.
  \end{itemize}
  We call~$\mathcal{D}_m(\ell)$ the $k$-\emph{defect} associated to~$m$ and~$\ell$.
\end{definition}

\begin{lemma}
  \label{l:khole}
  If~$m\in\mathbb{M}_k$ is $(u,\rho,k)$-good, with~$k\in\N$, then for $L_0 > \clzero$ there exists~$\ell\in\LL$ such that every~$m'\in\mathbb{M}_{k-1}$ satisfying
  \begin{equation}
    B_{m'}\subset B_m;\quad\quad B_{m'}\notin \mathcal{D}_m(\ell)
  \end{equation}
  is~$(u,\rho,k-1)$-good.
\end{lemma}

\begin{proof}
  Assume~$m$ is $(u,\rho,k)$-good. We refer to Figure~\ref{f:badbox} to help the reader visualize the argument that follows. If all the boxes of the scale~$k-1$ contained in~$m$ are $(u,\rho,k-1)$-good, we can just choose~$\ell$ arbitrarily and there is nothing to prove.

  Assume there exists a $(u,\rho,k-1)$-bad box $m_1=(x_1,k-1)$ such that~$B_{m_1}\subset B_m$. If there is no $(u,\rho,k-1)$-bad box contained in~$B_m$ and intersecting the complement of an Euclidean ball with center at~$x_1$ and radius~$2 k^2L_{k-1}^{2+\alpha}$, we can choose~$\ell$ arbitrarily containing~$x_1$ and there is nothing more to prove. If, however, there exists such a $(u,\rho,k-1)$-bad box $m_2=(x_2,k-1)$, we choose~$\ell$ as the line passing through~$x_1$ and~$x_2$.

  Finally, take $m_1, m_2$ and $\ell$ as above and assume moreover that there exists a $(u,\rho,k-1)$-bad box $m_3=(x_3,k-1)$ contained in~$B_m$ such that~$B_{m_3}\notin \mathcal{D}_m(\ell)$. We already know that~$m_1$, $m_2$ and~$m_3$ satisfy the conditions (i), (ii) and (iv) of Definition~\ref{def:kbadbox}. We will show that they also satisfy condition (iii), contradicting the hypothesis of~$m$ being $(u,\rho,k)$-good.
  \begin{figure}[ht]
    \centering
    \includegraphics[scale = .5]{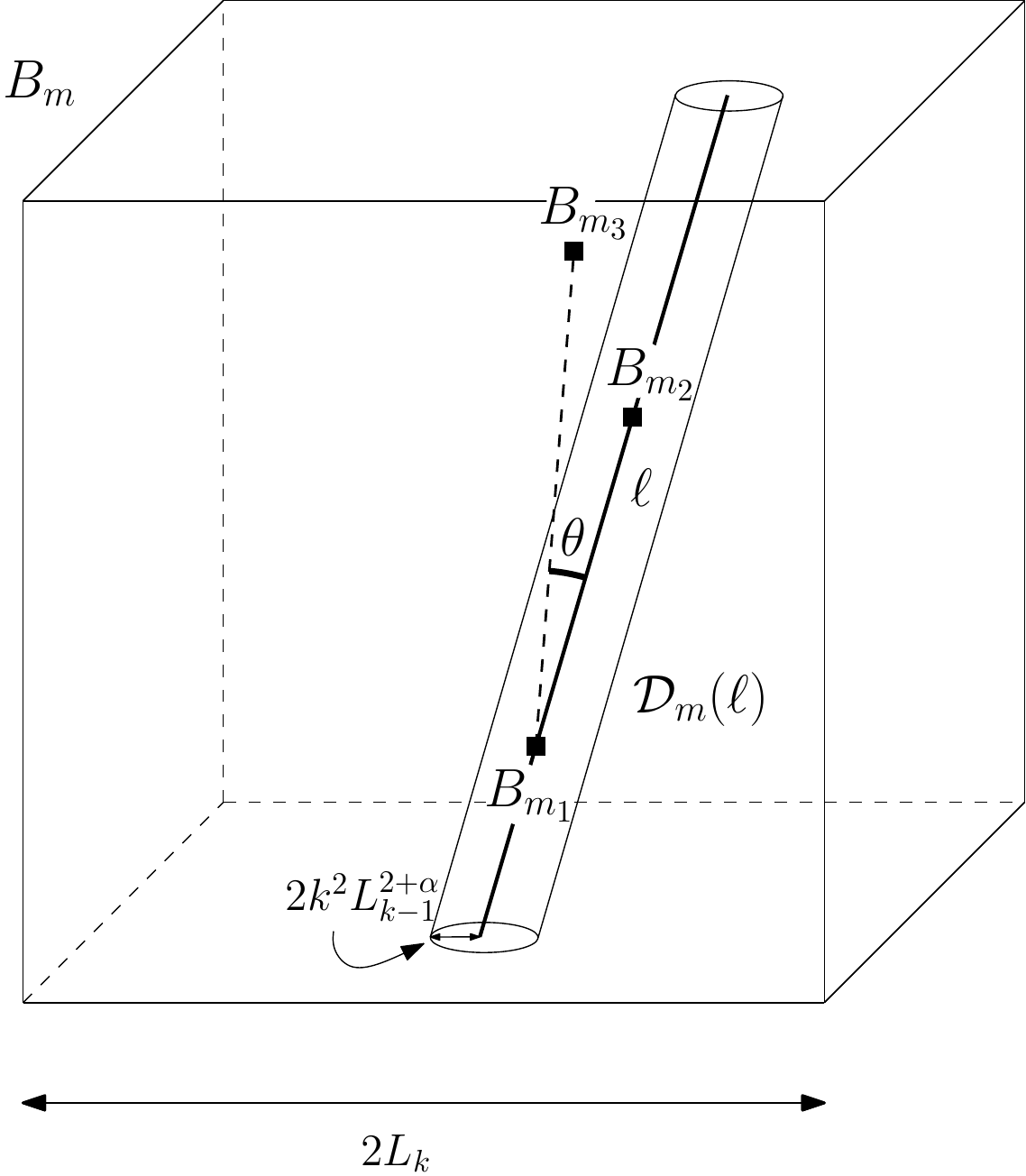}
    \vspace{0.1cm}
    \caption{A bad box~$B_m$. The existence of the unaligned boxes~$B_{m_1}$, $B_{m_2}$ and~$B_{m_3}$ makes an application of Theorem~\ref{thm:3boxdec} possible.}
    \label{f:badbox}
  \end{figure}
  Consider the triangle formed by the vertices~$x_1$, $x_2$ and~$x_3$. Either the angle corresponding to~$x_2$ or~$x_1$ must be acute. Without loss of generality, assume the latter holds, and denote this angle by~$\theta$. After a rigid motion of~$\R^d$, we may consider~$x_1$ as being the origin and the line~$\ell$ as being the axis~$\{t \cdot {\bf e}_d;t \in \R \}$. Let~$t_3$ denote the~$d$-th coordinate of~$x_3$ after this rigid motion, and~$d_3$ the distance between~$x_3$ and~$\ell$. Since~$\theta<\pi/2$, we have~$t_3>0$, and therefore, after possibly increasing $\clzero$,
  \begin{equation}
    \label{e:khole2}
    \begin{split}
      \theta=\arctan\left(  \frac{d_3}{t_3}\right)\geq \arctan\left(  \frac{c k^2 L_{k-1}^{2+\alpha}}{k^2 L_{k-1}^{\alpha+\beta} L_{k-1}^2} \right)
      \geq \arctan\left(  \frac{c}{ L_{k-1}^{\beta}} \right) \geq \frac{c}{ L_{k-1}^{\beta}} .
    \end{split}
  \end{equation}
  Now for sufficiently large~$L_0$ this implies condition (iii) of Definition~\ref{def:kbadbox}, finishing the proof of the result.
\end{proof}

\section{Efficient unoccupied paths}
\label{s:paths}

In this section we will lay the groundwork for the study of the energy of a flow in a discretized version of the vacant set~$\mathcal{V}_u^\rho$ using the renormalization results proved in Section~\ref{s:renorm}. This study will be completed in Section~\ref{s:flow}, where we will use the discrete paths constructed in the present section in order to show the existence of a discrete finite energy flow. We start with the necessary definitions.

For~$x,y\in\Z^d$, we let~${\bf line}(x,y)$ denote the closed line segment connecting~$x$ to~$y$ in~$\R^d$. We then define the set of points in~$\Z^d$ whose line segments associated to their nearest neighbors do not intersect the cylinder set:
\begin{equation}
  \mathsf{V}^u_\rho = \mathsf{V}^u_\rho(\omega) := \left\{
    x \in\Z^d;\, \right(\bigcup_{j=1}^d {\bf line}(x,x+{\bf e}_j)\cup {\bf line}(x,x-{\bf e}_j) \left) \cap \mathcal{C}^u_{\rho}(\omega) = \emptyset
  \right\}.
\end{equation}

We consider in the \emph{discrete vacant set}~$\mathsf{V}^u_\rho$ the graph structure inherited from the nearest-neighbors graph of~$\Z^d$.

\begin{remark}
  \label{r:discrete}
  The reason why we consider the discrete set $\mathsf{V}^u$ instead of its continuous counterpart is for technical simplification of the arguments, specially comparing the random walk on $\mathsf{V}^u$ instead of the Brownian Motion on $\mathcal{V}^u$.
  But we are confident that these results can be extended to analogous ones for the continuous setting.
\end{remark}

The flow we want to define using the carpet from Section~\ref{s:renorm} will be constructed from paths which will be defined in a hierarchical fashion at each scale. From a ``coarse'' path at scale~$k$, we will construct a finer path with of boxes at scale $k - 1$ and so on. We do so in order for these paths to avoid the defects present at every scale, so that they navigate through boxes where the cylinder set is well behaved.

For each good box $m$ we will now introduce the notion of the \emph{hole} $\mathsf{H}_m$ which roughly speaking will represent a region in $m$ to be avoided.
For the precise definition, we need to consider the cases $m \in \mathbb{M}_0$ in separate.

For $m\in\mathbb{M}_0$, the \emph{hole} $\mathsf{H}_m$ will correspond exactly to the closed sites in $B_m$ or more precisely $\mathsf{H}_m(\omega):= (B_m\cap\Z^d)\setminus \mathsf{V}^u_\rho$. For $k \geq 1$ and a good box~$m\in\mathbb{M}_k$ with associated $k$-defect $\mathcal{D}_m(\ell)$, we define the hole of~$m$ as
\begin{equation}
\label{eq:mhole}
\mathsf{H}_m=\mathsf{H}_m(\omega):=  \bigcup_{m' \in \mathcal{D}_m(\ell) }  B_{m'} \cap \Z^d.
\end{equation}

We define the set of unit vectors parallel to the cordinate axes
\begin{equation}
  \label{eq:unit_vectors}
  \mathsf{U} := \{ {\bf e}_1, \dots, {\bf e}_d, -{\bf e}_1, \dots, -{\bf e}_d \}
\end{equation}
Given~$m=(x,k) \in \mathbb{M}_k$, with~$k\geq 0$, and some ${\bf v} \in \mathsf{U}$, we define the \emph{face} of~$m$ associated to ${\bf v}$
\begin{equation}
  \label{eq:mface1}
  \mathsf{F}_{m,{\bf v}}
    :=
      \big\{ y \in B_m \cap \Z^d ; \, \langle y - x, {\bf v} \rangle = L_k \big\}.
\end{equation}

In order to transfer flow from one box to the adjacent one, we will first define a suitable collection of points and squares along their interfaces.
This is illustrated in Figure~\ref{f:Bmycollec} and it is rigorously defined below.
\begin{figure}[ht]
	\centering
	\includegraphics[scale = .5]{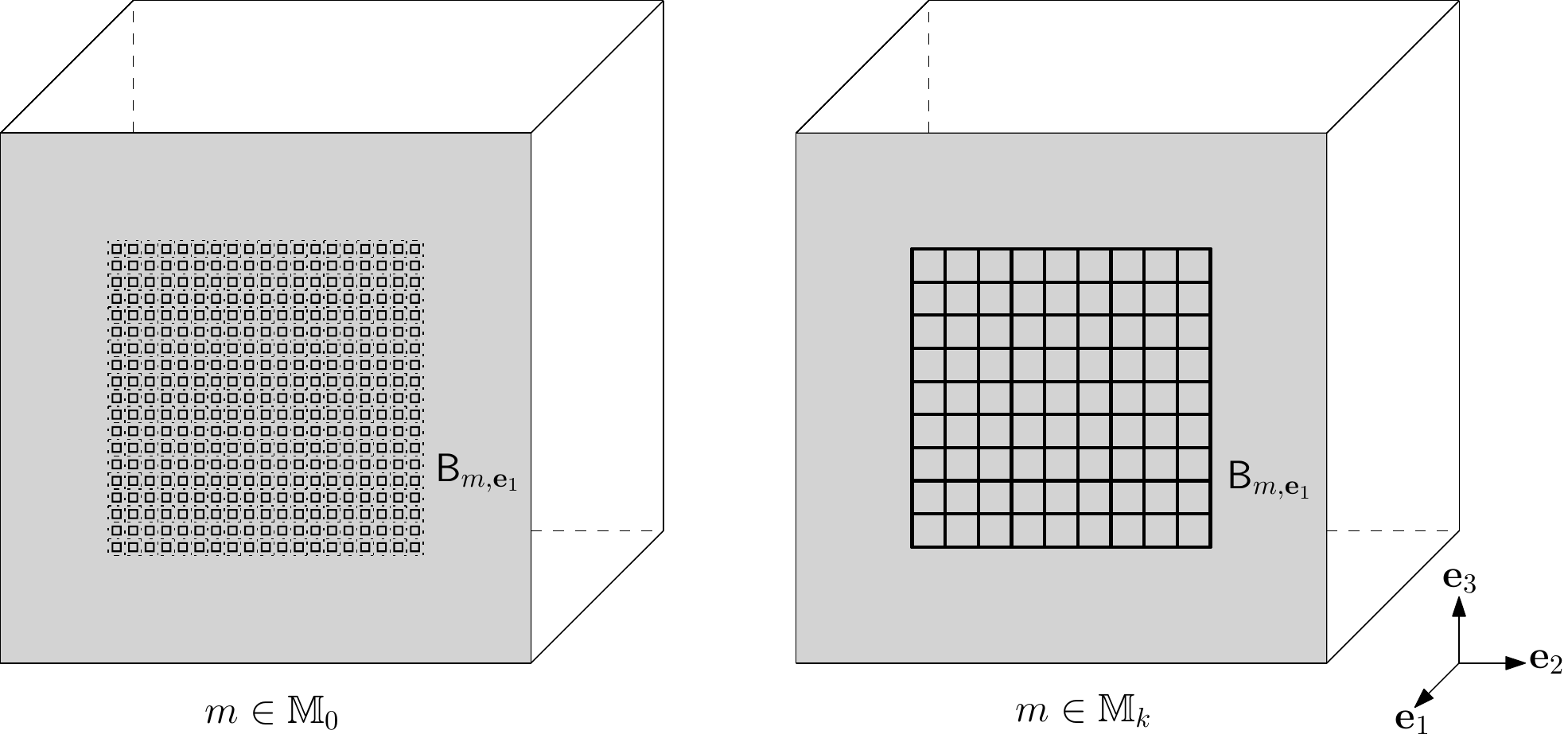}
	\vspace{0.1cm}
	\caption{The collection of boxes~$\mathsf{B}_{m,{\bf e}_1}$ in the case~$m\in\mathbb{M}_0$ and~$m\in\mathbb{M}_k$. In the first case, the discrete $(d-1)$-dimensional boxes have radius~$ 4^{-1} L_0^{\frac{7}{10}} $, in the second, radius~$17^{-1} L_k $. }
	\label{f:Bmycollec}
\end{figure}

We start at scale zero.
More precisely, for~$m = (x,0) \in \mathbb{M}_0$ and $j = 1, \dots, d$ we define the vertex collection
\begin{equation}
  \label{eq:fbbdef0}
  \dot{\mathsf{F}}_{m,{\bf e}_j} := \left\{
    \begin{split}
      \smash{
      x + L_0 {\bf e}_j + \sum_{i\neq j } a_i \Big\lfloor L_0^{\frac{7}{10}} \Big\rfloor {\bf e}_i;} \;
      & a_i = \Big(- 2^{-1} L_0^{\frac{3}{10}},2^{-1} L_0^{\frac{3}{10}} \Big)\cap \Z \\
      & i = 1, \dots, j - 1, j + 1, \dots, d
    \end{split}
  \right\},
\end{equation}
which is composed of lattice points on the face $\mathsf{F}_{m, {\bf e}_j}$, with inter-spacing $\lfloor L_0^{7/10} \rfloor$ and spanning a square with half the width of the box $B_m$, see Figure~\ref{f:Bmycollec}.

To each $y \in \dot{\mathsf{F}}_{m, {\bf e}_j}$ we associate a $(d - 1)$-dimensional ``small face''
\begin{equation}
  B_{m,{\bf e}_j}(y)
    :=
      B_\infty\Big(y, 4^{-1} L_0^{\frac{7}{10}}\Big) \cap \mathsf{F}_{m,{\bf e}_j}
\end{equation}
We also define the whole collection of such small faces
\begin{equation}
  \label{eq:bcaldef0}
  \mathsf{B}_{m,{\bf e}_j}:=\left\{ B_\infty\Big(y, 4^{-1} L_0^{\frac{7}{10}}\Big) \cap \mathsf{F}_{m,{\bf e}_j};\, y \in \dot{\mathsf{F}}_{m,{\bf e}_j} \right\},
\end{equation}
defining analogously the collections~$\dot{\mathsf{F}}_{m,-{\bf e}_j}$ and~$\mathsf{B}_{m,-{\bf e}_j}$. At scale~$0$, we will use these small faces such as~$B_{m,{\bf e}_j}(y)$ as bases of long prisms contained inside $B_m$. Good prisms will evade the hole $\mathsf{H}_m$, and we will use isoperimetric properties of~$\Z^d$ in order to connect good prisms inside~$B_m$ using paths of vacant vertices -- these will be the \emph{good paths} at scale~$0$.

We are now ready to treat the case $m=(x,k)\in\mathbb{M}_k$ with $k \geq 1$, which will have a different choice of sizes:
\begin{equation}
  \label{eq:fbbdefk}
  \dot{\mathsf{F}}_{m,{\bf e}_j} := \left\{
  \begin{split}
    \smash {
      x + L_k {\bf e}_j + \sum_{i\neq j }  17^{-1} L_k  a_i {\bf e}_i; \;
    }
    & a_i \text{ takes value in} \\
    & \{-8,-6,-4,-2,0,2,4,6,8\}
  \end{split}
  \right\},
\end{equation}
and
\begin{equation}
\label{eq:bcaldefk}
\mathsf{B}_{m,{\bf e}_j}:=\left\{  B_\infty\Big(y, 17^{-1} L_k \Big) \cap \mathsf{F}_{m,{\bf e}_j};\, y \in \dot{\mathsf{F}}_{m,{\bf e}_j} \right\},
\end{equation}
again defining analogously the collections~$\dot{\mathsf{F}}_{m,-{\bf e}_j}$ and~$\mathsf{B}_{m,-{\bf e}_j}$. In general, we will denote the element of~$\mathsf{B}_{m,{\bf e}_j}$ associated to~$y \in \dot{\mathsf{F}}_{m, {\bf e}_j}$ by~$B_{m,{\bf e}_j}(y)$. Note that the smaller faces at scale~$k \geq 1$ have size of the same order as $L_k$, which was not the case for scale~$0$.

Given $m = (x, k) \in \mathbb{M}_k$, we consider a graph structure in $B_m$ isomorphic to the finite lattice box with radius $8$, $B(0, 8) \cap \Z^d$. Recall that~$L_k\in 17\N$ and define the collection of points
\begin{equation}
  \label{eq:Lksobre17}
  \mathcal{B}_m := \left\{
    \begin{split}
      \smash{
        x + \sum_{i=1}^d a_i \frac{ L_k}{17} {\bf e}_i; \;
      }
      & a_i \text{ takes value in} \\
      & \{-16,-14,\dots,-2,0,2,\dots,14,16 \}
    \end{split}
  \right\},
\end{equation}
and notice that~$\mathcal{B}_m\subset \mathbb{M}_{k-1}$. Fixed some $m \in \mathbb{M}_{k}$ for $k \geq 1$ and any given $y \in \dot{\mathsf{F}}_{m, {\bf e}}$, there exists $y' \in \mathcal{B}_m$ such that $B_{m,{\bf e}_j}(y)\subset B_\infty(y', 17^{-1} L_k)$.
In fact, the~$(d-1)$-dimensional box~$B_{m,{\bf e}_j}(y)$ is contained in one of the faces of $B_\infty(y', 17^{-1} L_k)$.

We will use this finite lattice~$\mathcal{B}_m$ inside $B_m$ in order to construct collections of coarse-grained paths at the $k$-th scale which avoid the hole $\mathsf{H}_m$ and which behave well in our hierarchical construction. Since for $k \geq 1$ the problematic region~$\mathsf{H}_m$ is quite small, we can avoid it more easily than at scale $0$. Again, we will use prisms whose bases are faces in~$\mathsf{B}_{m,{\bf e}_j}$. Figure~\ref{f:dminus1defects} shows such a prism.

In order to be able to concatenate good paths from adjacent boxes, it will be necessary to introduce more notation related to the face shared by such boxes.
For~$k \geq 0$, if~$m=(x,k)$ and~$m'=(x + 2L_k {\bf e}_j,k)$, we have
\begin{equation}
  \mathsf{F}_{m,{\bf e}_j} = \mathsf{F}_{m',-{\bf e}_j}.
\end{equation}
If the boxes associated to $m$ and~$m'$ are both good, we say that the face~\emph{$\mathsf{F}_{m ,{\bf e}_j}$ is good}.
In this case we also define the projections of the holes~$\mathsf{H}_{m }$ and~$\mathsf{H}_{m'}$ onto~$\mathsf{F}_{m ,{\bf e}_j}$:
\begin{equation}
  \label{eq:mholedminus1}
  \mathsf{H}_{m ,{\bf e}_j}^{d-1}:= \left\{
    \begin{split}
      (y_1,\dots,y_d) \in \mathsf{F}_{m,{\bf e}_j}; \;
      & \exists (x_1,\dots,x_d)\in \mathsf{H}_{(x,k)}\cup\mathsf{H}_{(x + 2L_k {\bf e}_j,k)} \\
      & \text{ such that } x_i = y_i \text{ for }i\neq j
    \end{split}
  \right\},
\end{equation}
see Figure~\ref{f:dminus1defects}.

\begin{figure}[ht]
  \centering
  \includegraphics[scale = .6]{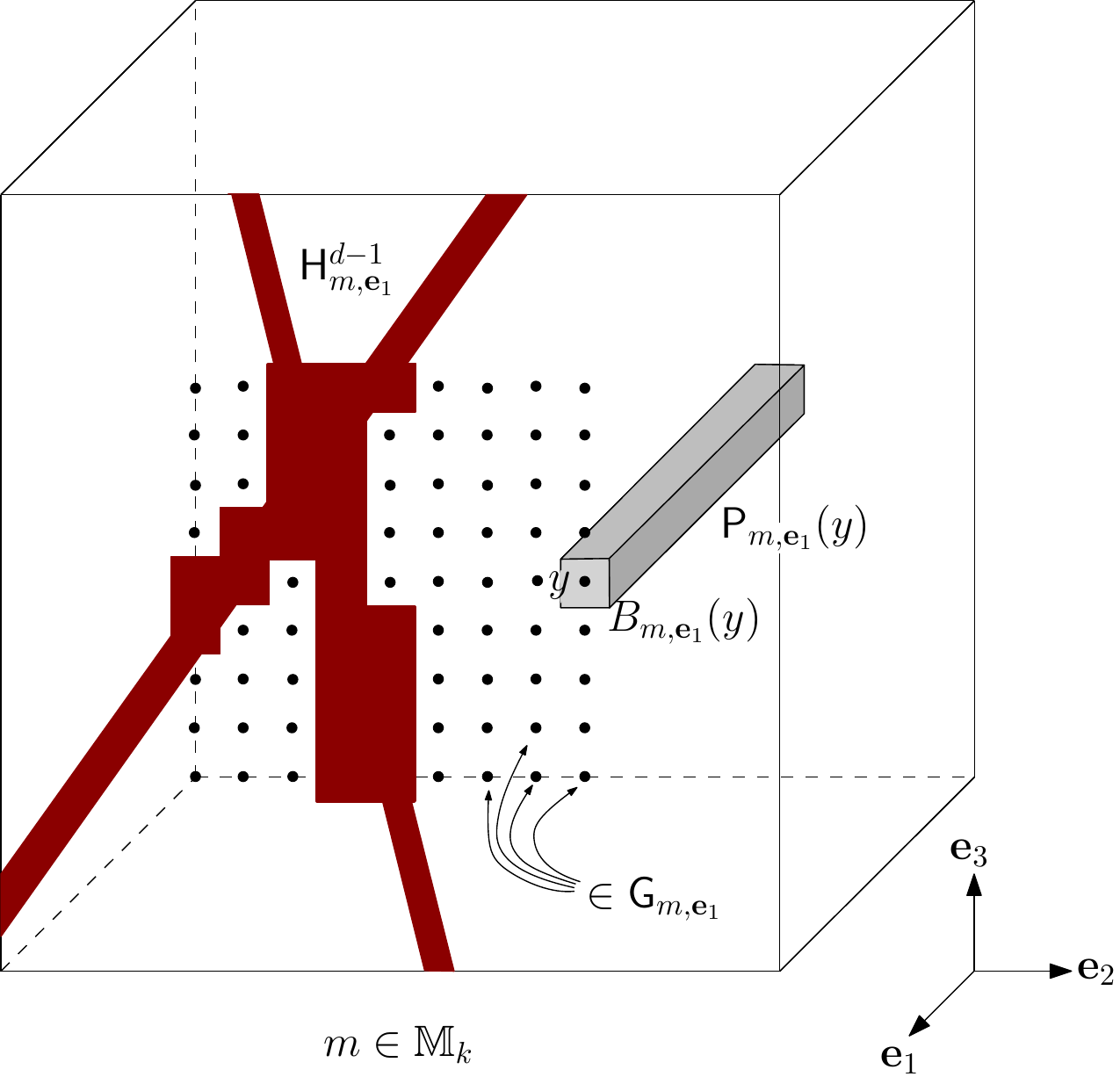}
  \vspace{0.1cm}
  \caption{Sets associated to a good face~$\mathsf{F}_{m,{\bf e}_1}$, when~$m\in\mathbb{M}_k$, and~$k\in\N$. When~$k=0$, an analogous picture holds place, this time the boxes~$B_{m,{\bf e}_1}(y)$ having mesoscopic radius~$ 4^{-1} L_0^{\frac{7}{10}}$. }
  \label{f:dminus1defects}
\end{figure}

We then define,
\begin{equation}
\label{eq:intbmejdef}
\mathsf{G}_{m ,{\bf e}_j}\equiv\mathsf{G}_{m',-{\bf e}_j}:= \left\{ y \in  \dot{\mathsf{F}}_{m ,{\bf e}_j};\,B_{m,{\bf e}_j}(y) \cap \mathsf{H}_{m ,{\bf e}_j}^{d-1} = \emptyset \right\},
\end{equation}
the sets of points in~$\dot{\mathsf{F}}_{m ,{\bf e}_j}$ which are centers of $(d-1)$-dimensional boxes in~$\mathsf{B}_{m ,{\bf e}_j}$, and whose associated boxes do not intersect the $(d-1)$-dimensional defect~$\mathsf{H}_{m ,{\bf e}_j}^{d-1}$. We define~$\mathsf{P}_{m ,{\bf e}_j}(y)$, the~\emph{prism} of~$y\in \mathsf{G}_{m ,{\bf e}_j}$, as the set of points in~$B_{m }\cap\Z^d$ whose orthogonal projection onto~$\mathsf{F}_{m ,{\bf e}_j}$ belongs to~$B_{m ,{\bf e}_j}(y)$. We have that, after possibly increasing the value of $\clzero$:
\begin{display}
  \label{eq:mathsfGnonempty}
  As long as the face~$\mathsf{F}_{m ,{\bf e}_j}$ is good, the set~$ \mathsf{G}_{m ,{\bf e}_j}$ is non-empty.
\end{display}
This can be seen using an elementary counting argument for~$L_0$ sufficiently large. We refer to Figure~\ref{f:dminus1defects}.

We now start the construction of the collections of \emph{efficient} paths: paths of unoccupied vertices that traverse long Euclidean distances without spending \emph{too much ``time''} in any one given box, and which do not intersect each other \emph{too much}. This will later be used in order to construct a low energy flow. We start by proving a lemma which starts this construction in the $0$-th scale, where we may allow some inefficiency. For $x \in \R^d$ and $r > 0$ we let $\mathrm{int}(B_\infty(x,r))$ denote the interior of the box~$B_\infty(x,r)$, that is, the box~$B_\infty(x,r)$ minus its faces.

\begin{lemma}
\label{l:0path}
Consider~$m\in\mathbb{M}_0$, $L_0 > \clzero$, and~${\bf v},{\bf w} \in \mathsf{U}$, ${\bf v} \neq {\bf w}$.
Then in the event where both~$\mathsf{F}_{m,{\bf v}}$ and~$\mathsf{F}_{m,{\bf w}}$ are good, given~$y_{\bf v}\in\mathsf{G}_{m,{\bf v}}$ and~$y_{\bf w}\in\mathsf{G}_{m,{\bf w}}$, there exists a path of neighboring vertices in~$\mathsf{V}^u_\rho\cap B_m$ connecting~$y_{\bf v}$ to~$y_{\bf w}$ of length at most~$(2L_0+1)^d$ which only intersects the faces of~$B_m$ at~$y_{\bf v}$ and~$y_{\bf w}$.
\end{lemma}

\begin{remark}
  Note the inefficiency that we allow ourselves in bounding the length of the path by the volume of the box.
  This is not problematic at scale zero, since it only contributes to the energy of flows by a multiplicative constant depending on $L_0$.
\end{remark}

\begin{proof}
If such path exists, it must have length at most~$(2L_0+1)^d$ simply because this is the cardinality of the discrete box~$B_m\cap\Z^d$. To show the existence of the path with the required properties, we note that
\begin{equation}
\mathsf{P}_{m,{\bf v}}(y_{\bf v}), \mathsf{P}_{m,{\bf w}}(y_{\bf w})\subset \mathsf{V}^u_\rho\cap B_m.
\end{equation}
Furthermore, for sufficiently large~$L_0$, the cardinality of both these prisms intersected with~$\mathrm{int}(B_m)$ is larger than
\[
8^{-(d-1)} \cdot L_0^{\frac{7(d-1)}{10}+1} ,
\]
while the cardinality of~$\mathsf{H}_m$ is smaller than~$c\rho^{d-1} L_0^{1+\gamma}$. Since~$\gamma<1/5$ and~$d\geq 3$, the fraction
\begin{equation}
\label{eq:0scale_path_isop}
\frac{|\mathsf{H}_m \cap \mathrm{int}(B_m) |}{|\mathsf{P}_{m,{\bf v}}(y_{\bf v})\cap \mathrm{int}(B_m)|^{\frac{d-1}{d}} } 
\end{equation}
can be made arbitrarily small by increasing $L_0$. Since the discrete box~$\mathrm{int}(B_m)\cap \Z^d$ inherits the isoperimetric inequality of~$\Z^d$ with a smaller constant depending on the dimension, there must exist, after possibly increasing $\clzero$ and requiring~$L_0 > \clzero$, a path from~$\mathsf{P}_{m,{\bf v}}(y_{\bf v})$ to $\mathsf{P}_{m,{\bf w}}(y_{\bf w})$ which does not intersect~$\mathsf{H}_m$, nor the faces of~$B_m$. This finishes the proof of the lemma.
\end{proof}

The next lemma is the first step in the construction of a collection of efficient paths at a scale $k\in\N$. We construct coarse paths in $\mathcal{B}_m$, which will later in lemmas \ref{l:B17kminus1path} and \ref{l:kminus1paths} serve as guides to construct paths at scale $k - 1$. We denote by~$\mathsf{V}^{u,k-1}_{\rho}$ the set of vertices $x\in 2L_{k-1} \Z^d$ whose associated boxes $(x,k-1)\in\mathbb{M}_{k-1}$ are $(u,\rho,k-1)$-good. If~$m\in\mathbb{M}_{k}$, we let~$B_m^{k-1}$ denote the set of vertices of~$2L_{k-1} \Z^d$ whose associated boxes are contained in~$B_m$. Similarly, if~$y\in\mathsf{G}_{m,{\bf e}_j}$, we denote by~$ \mathsf{P}_{m,{\bf e}_j}^{k-1}(y)$ the set of vertices of~$2L_{k-1} \Z^d$ whose associated boxes are contained in~$ \mathsf{P}_{m,{\bf e}_j}(y)$. We also define~$B_{m ,{\bf e}_j}^{k-1}(y)$ as the set of points of~$B_{m ,{\bf e}_j}(y)$ contained in $ 2L_{k-1} \Z^d + L_{k-1}{\bf e}_j$, that is, points of the $(d-1)$-dimensional box associated to~$y$ which are translations by~$L_k{\bf e}_j$ of points from the $(k-1)$-th scale.
We will also utilize analogous notation when considering~$-{\bf e}_j$ instead of~${\bf e}_j$. Given~$m\in\mathbb{M}_k$, we consider in~$\mathcal{B}_m$ the nearest-neighbor graph structure, so that we may talk about adjacent points and paths in~$\mathcal{B}_m$.

\begin{lemma}
  \label{l:17path}
  Consider~$m\in\mathbb{M}_k$, with~$k\in\N$, ${\bf v},{\bf w} \in \mathsf{U}$, ${\bf v} \neq {\bf w}$, and assume the occurrence of the event where both~$\mathsf{F}_{m,{\bf v}}$ and~$\mathsf{F}_{m,{\bf w}}$ are good.
  Then, given~$y_{\bf v}\in\mathsf{G}_{m,{\bf v}}$ and~$y_{\bf w}\in\mathsf{G}_{m,{\bf w}}$, there exists a simple path $z_1,\dots, z_n$ of neighboring vertices in~$\mathcal{B}_m$, with $n\leq 17^d$, such that~$B_m(y_{\bf v})\subset B_\infty(z_1,17^{-1}L_{k})$, $B_m(y_{\bf w})\subset B_\infty(z_n,17^{-1}L_{k})$, and every box~$(x,k-1)\in\mathbb{M}_{k-1}$ such that~$x\in B_\infty(z_i,17^{-1}L_{k})$, $i=1,\dots,n$, is good.
\end{lemma}

\begin{proof}
  Since~$17^d$ is the cardinality of~$\mathcal{B}_m$, if a suitable path exists, its length automatically satisfies the requested upper bound. Furthermore, we can focus on the case when ${\bf v} \neq -{\bf w}$, that is, when the faces considered are adjacent. Indeed, If ${\bf v} = -{\bf w}$, we can choose an ${\bf u} \in \mathsf{U}$ orthogonal to ${\bf v}$, and if we can construct simple paths connecting ${\bf u}$ to ${\bf v}$ and ${\bf u}$ to $-{\bf v}$, we can also construct a simple path between ${\bf v}$ and $-{\bf v}$.

  We notice that, since~$y_{\bf v}\in\mathsf{G}_{m,{\bf v}}$ and~$y_{\bf w}\in\mathsf{G}_{m,{\bf w}}$,~$\mathsf{P}_{m,{\bf v}}^{k-1}(y_{\bf v})$ and~$\mathsf{P}_{m,{\bf w}}^{k-1}(y_{\bf w})$ are contained in~$\mathsf{V}^{u,k-1}_{\rho}$, the set of vertices whose associated boxes are~$(k-1)$-good. Furthermore, each of these prisms is the union of~$17$ boxes with center in~$\mathcal{B}_m$ and radius~$17^{-1} L_k$, these boxes sharing faces in the prism's corresponding directions. That is, the prisms already contain a long path of boxes with centers in~$\mathcal{B}_m$ and radius~$17^{-1}L_k$ whose vertices of the~$(k-1)$-th scale are contained in~$\mathsf{V}^{u,k-1}_{\rho}$. We will show now how to join these paths while avoiding the hole~$\mathsf{H}_m\cap 2L_{k-1}\Z^d$.

  \begin{figure}
    \centering
    \begin{subfigure}[b]{0.32\textwidth}
        \centering
        \includegraphics[width=\textwidth]{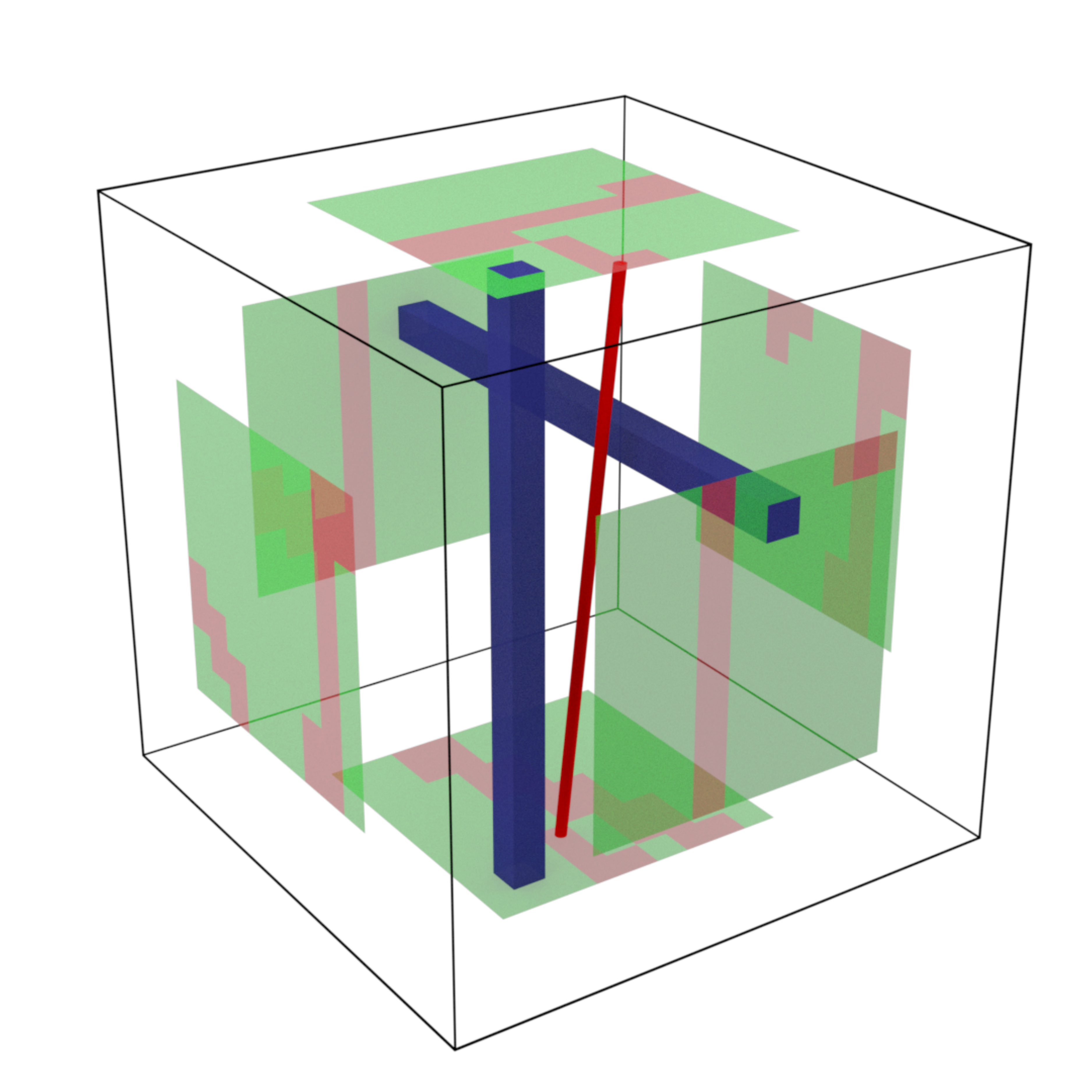}
        \label{f:B17_1}
    \end{subfigure}
    \hfill
    \begin{subfigure}[b]{0.32\textwidth}
        \centering
        \includegraphics[width=\textwidth]{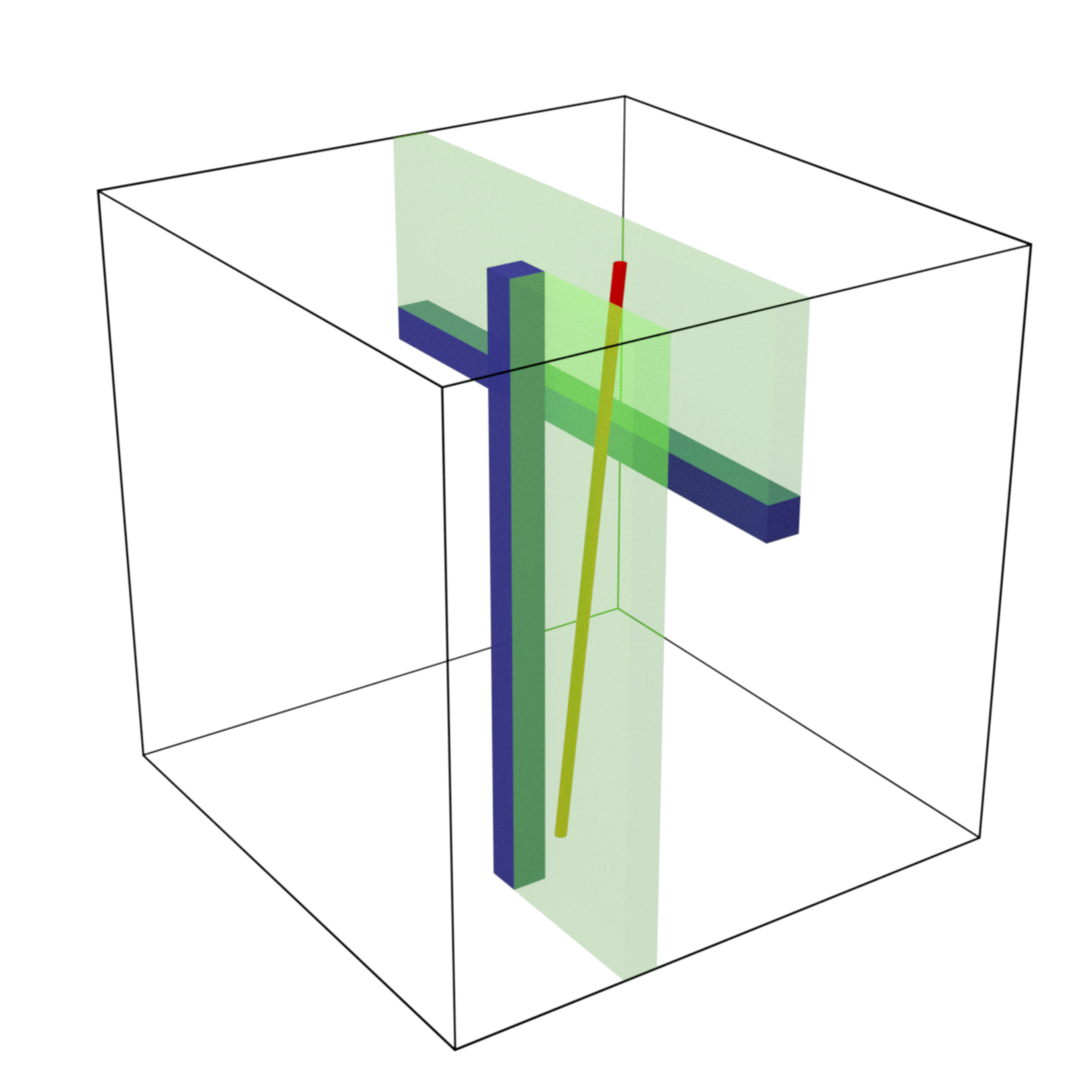}
        \label{f:B17_2}
    \end{subfigure}
    \hfill
    \begin{subfigure}[b]{0.32\textwidth}
        \centering
        \includegraphics[width=\textwidth]{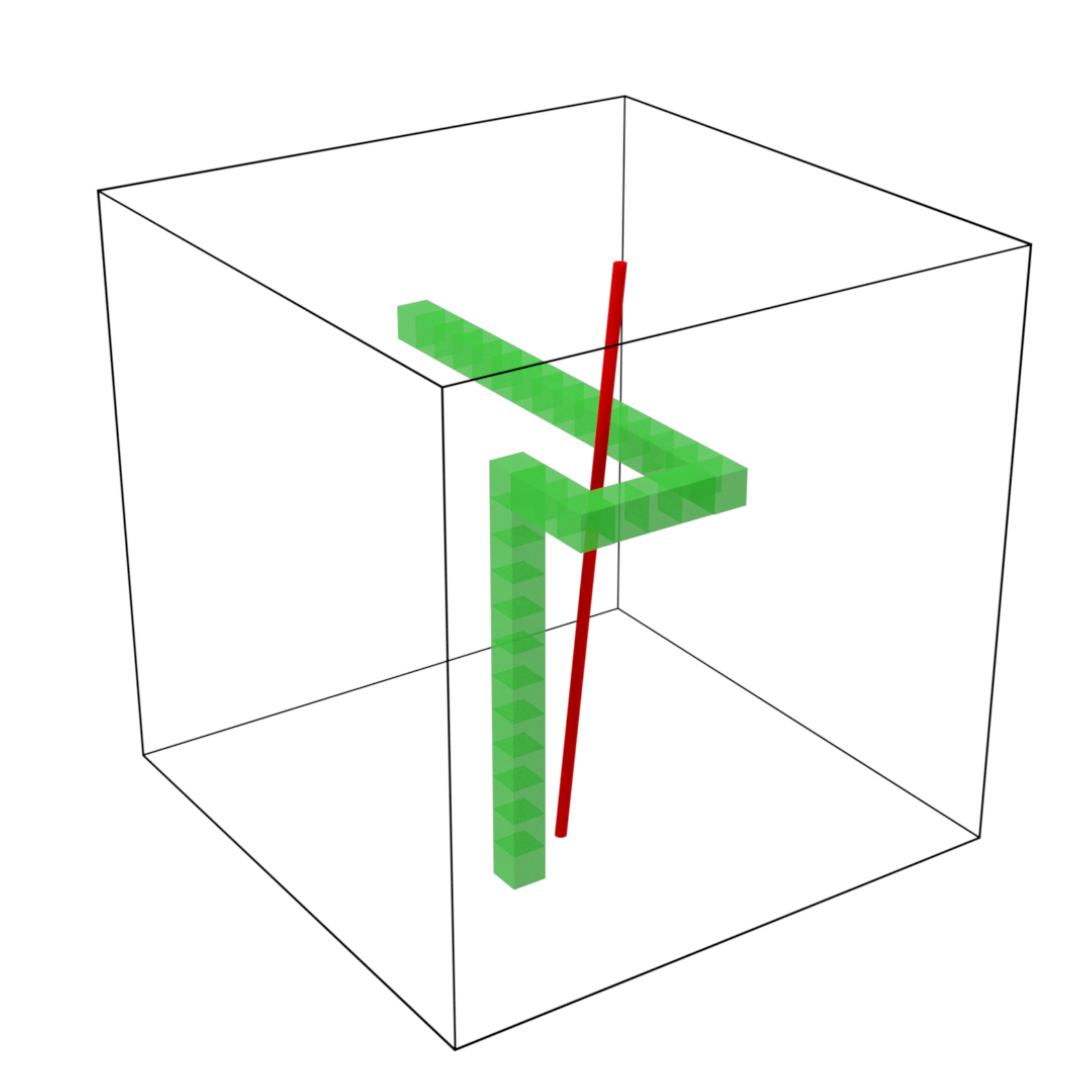}
        \label{f:B17_4}
    \end{subfigure}
    \caption{$3$-dimensional representation of the construction present in the proof of Lemma~\ref{l:17path}. We want to join the two prisms above by boxes of $\mathcal{B}_m$ that do not meet the cylinder-like defect $\mathcal{D}_m$. We construct from the prisms two parallel ``sheets'' of boxes that do not intersect the defect. There exists at least $4^{d - 1}$ rectilinear paths connecting the two sheets, and the defect cannot block them all. In this way, the desired path of boxes can be constructed.}
    \label{f:B17}
  \end{figure}

  Without loss of generality, we assume ${\bf v} = {\bf e}_1$ and ${\bf w} = {\bf e}_2$. In what follows we consider~$\mathcal{B}_{m}$ as a \emph{subgraph} of the $d$-dimensional hypercubic lattice $(2 L_k / 17) \Z^d$ -- specifically, as a box with side-length $17$. In this way, we can regard the prism
  \begin{equation*}
    \mathsf{P}_1 := \mathsf{P}_{m,{\bf e}_1}^{k-1}(y_{{\bf e}_1})
  \end{equation*}
  as a union of $17$ aligned ``box-vertices'', doing the same for
  \begin{equation*}
    \mathsf{P}_1' := \mathsf{P}_{m,{\bf e}_2}^{k-1}(y_{{\bf e}_2}).
  \end{equation*}
  The Lemma will be proved once we show that there exists a path of boxes inside $\mathcal{B}_m$ from~$\mathsf{P}_1$ to $\mathsf{P}_1'$ which avoids boxes intersecting the defect $\mathcal{D}_m$. We refer to Figure~\ref{f:B17} for an overview of the construction.

  We consider the translations of $\mathsf{P}_1$ by integer multiples of $(2 L_k / 17) {\bf e}_2$. By definiton of the defect $\mathcal{D}_m$, it can either intersect translations of $\mathsf{P}_1$ by positive integer multiples of $(2 L_k / 17) {\bf e}_2$, or by negative integer multiples, but not both. If it does not intersect the positive translations, we define
  \begin{equation*}
    \mathsf{P}_2 := B_m \cap \bigcup_{i \geq 0} \big( \mathsf{P}_1 + (2 L_k / 17) i \cdot {\bf e}_2 \big),
  \end{equation*}
  otherwise, we let
  \begin{equation*}
    \mathsf{P}_2 := B_m \cap \bigcup_{i \geq 0} \big( \mathsf{P}_1 - (2 L_k / 17) i \cdot {\bf e}_2 \big).
  \end{equation*}
  We then continue this process for each vector ${\bf e}_n$, with $n = 2, \dots, d - 1$, considering translations of $\mathsf{P}_{n - 1}$ by positive and negative integer multiples of $(2 L_k / 17) {\bf e}_n$, and defining
  \begin{equation*}
    \mathsf{P}_n := B_m \cap \bigcup_{i \geq 0} \big( \mathsf{P}_{n - 1} \pm (2 L_k / 17) i \cdot {\bf e}_n \big),
  \end{equation*}
  choosing the sign in the $\pm$ symbol above so that $\mathsf{P}_n$ does not intersect the defect associated to the box. We thus obtain a ``$(d - 1)$-dimensional'' sheet of boxes~$\mathsf{P}_{d - 1}$. We perform the same construction starting with $\mathsf{P}_1'$
  and enlarging this set by uniting it with successive translations by multiples of the vectors ${\bf e}_1, {\bf e}_3, {\bf e}_4, \dots, {\bf e}_{d - 1}$, selecting the sign appropriately so they also do not intersect the defect, finally obtaining another sheet $\mathsf{P}_{d - 1}'$.

  The sheets $\mathsf{P}_{d - 1}$ and $\mathsf{P}_{d - 1}'$ are parallel by construction: they both have thickness consisting of \emph{one} box in the direction ${\bf e}_d$. Also, by construction, the projections of these sheets onto the $(d - 1)$-dimensional sublattice $(2 L_k / 17) \Z^{d - 1} \times \{0\}$ intersect in a $(d - 1)$-dimensional box of side-length at least $4$. This implies the existence of $4^{d - 1}$ disjoint linear paths of boxes on $\mathcal{B}_m$ from $\mathsf{P}_{d - 1}$ to $\mathsf{P}_{d - 1}'$, these path being parallel to ${\bf e}_d$. By the definition of the defect $\mathcal{D}_m$, it cannot intersect all of these paths, and we obtain the desired result.
\end{proof}

We now continue with the second step of the hierarchical construction of good paths: we prove a very elementary lemma showing how to construct good paths at scale $k - 1$ inside a box of $\mathcal{B}_m$, $m \in \mathbb{M}_k$, which is completely vacant at scale~$k - 1$. The recipe will later be used in Lemma~\ref{l:kminus1paths} to concatenate paths at scale $k - 1$ inside boxes of the $k$-th scale.
\begin{figure}[ht]
	\centering
	\includegraphics[scale = 1]{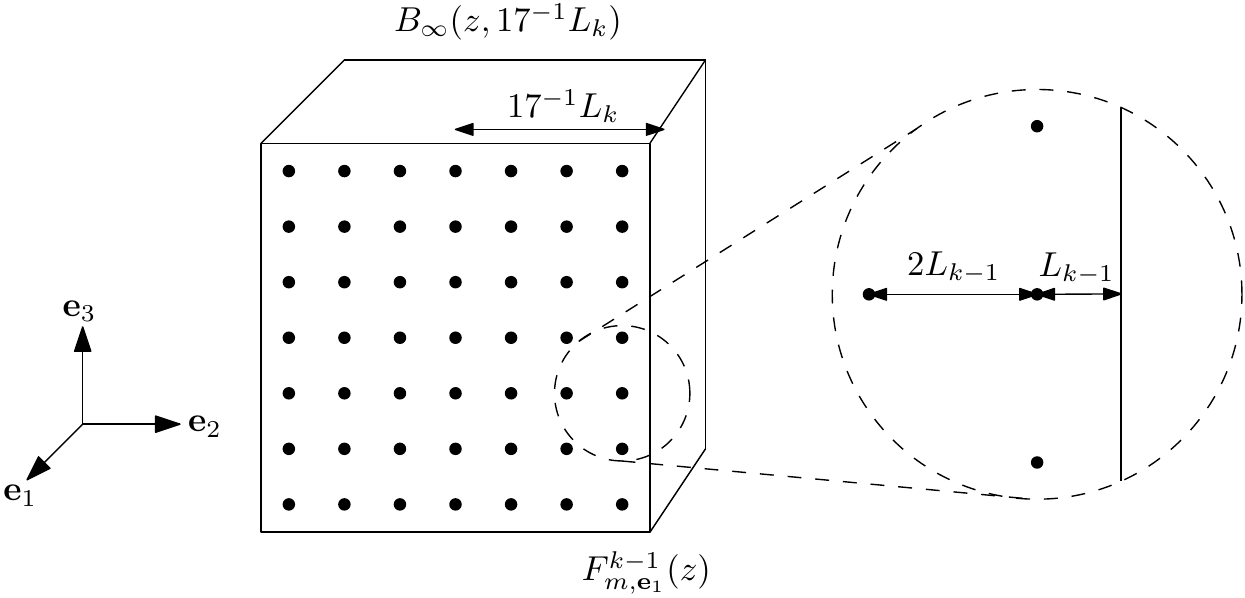}
	\vspace{0.1cm}
	\caption{The points in the image represent the set~$F^{k-1}_{m,{\bf e}_1}(z)$, a subset of the box~$B_\infty(z,17^{-1}L_k)$.}
	\label{f:Fkdef}
\end{figure}
We will consider in~$\mathsf{V}^{u,k-1}_\rho$ the nearest-neighbor graph structure and define, for~$m\in\mathbb{M}_k$, $k\in \N$, $z\in \mathcal{B}_m$, and ${\bf v} \in \mathsf{U}$, the set~$F^{k-1}_{m,{\bf v}}(z)$ as the points of~$2L_{k-1}\Z^d + L_{k-1}{\bf v}$ belonging to the face of the box~$B_\infty(z,17^{-1}L_k)$ associated to~${\bf v}$:
\begin{equation}
\label{eq:face17def}
\begin{split}
F^{k-1}_{m,{\bf v}}(z)
&:=
\left\{
\begin{array}{c}
 y \in B_\infty(z,17^{-1}L_k) \cap \Z^d ; \, \langle y-z,{\bf v} \rangle=  17^{-1}L_k     ,
\\
y \in ( 2L_{k-1}\Z^d + L_{k-1}{\bf v} )
\end{array}
\right\}
.
 \end{split}
\end{equation}
Note that, since $L_k$ is divisible by~$L_{k-1}$ and not by~$2L_{k-1}$, the points of~$2L_{k-1}\Z^d$ do not belong to faces of boxes associated to $\mathcal{B}_m$. We refer to Figure~\ref{f:Fkdef}.

\begin{lemma}
  \label{l:B17kminus1path}
  Given~$m\in\mathbb{M}_k$, $k\in\N$, and~$z\in \mathcal{B}_m$, assume that the box
  \[	B_\infty(z,17^{-1}L_k)\cap 2L_{k-1}\Z^d  \]
  is contained in~$\mathsf{V}^{u,k-1}_\rho$. Then, given the sets~$F^{k-1}_{m,{\bf v}}(z),F^{k-1}_{m,{\bf w}}(z) $ associated respectively to two distinct unit vectors~${\bf v},{\bf w} \in \mathsf{U}$, there exists a collection~$\mathcal{T}({\bf v},{\bf w},z,k-1)$ of vertex-disjoint nearest-neighbor paths of~$B_\infty(z,17^{-1}L_k)\cap 2L_{k-1}\Z^d$ such that for every~$x_0\in F^{k-1}_{m,{\bf v}}(z)$
  there exists
  \[
    (z_0,z_1,\dots,z_n)\in \mathcal{T}(y_{\bf v},y_{\bf w},k-1)
  \]
  such that~$n\leq 4\cdot 17^{-1} L_k L_{k-1}^{-1} $, $x_0=z_0 + L_{k-1}{\bf v}$, $z_n + L_{k-1}{\bf w}$ is in~$F^{k-1}_{m,{\bf w}}(z) $, and \[z_0,z_1,\dots,z_n\in\mathsf{V}^{u,k-1}_\rho. \]
\end{lemma}

\begin{proof}
  Since $B_\infty(z,17^{-1}L_k)\cap 2L_{k-1}\Z^d \subset \mathsf{V}^{u,k-1}_\rho$, we need only to construct this collection as a bundle of non-intersecting paths matching the vertices of the associated faces in orderly fashion, as shown in Figure~\ref{f:B17flow}.

  If~${\bf v}=-{\bf w}$, that is, if the faces are opposite to one another, we simply take $\mathcal{T}({\bf v},{\bf w},z,k-1)$ to be the collection of discrete straight lines in~$2L_{k-1} \Z^d$ parallel to~${\bf w}$ that bring each point~$z_0$ in~$F^{k-1}_{m,{\bf v}}(z)- L_{k-1}{\bf v}$ to $z_0+((2/17)L_k- L_{k-1}){\bf v}$ in~$F^{k-1}_{m,{\bf w}}(z)- L_{k-1}{\bf w}$.

  If not, without loss of generality we assume that~$z=0$, ${\bf v}={\bf e}_1$, and~${\bf w}={\bf e}_2$. Then, for
	\begin{equation}
	\begin{split}
	x_0 = ( 17^{-1}L_k ){\bf e}_1   +a_2 {\bf e}_2 +\dots + a_ d{\bf e}_d  \in F^{k-1}_{m,{\bf v}}(z),
	\end{split}
	\end{equation}
	\begin{figure}[ht]
		\centering
		\includegraphics[scale = .6]{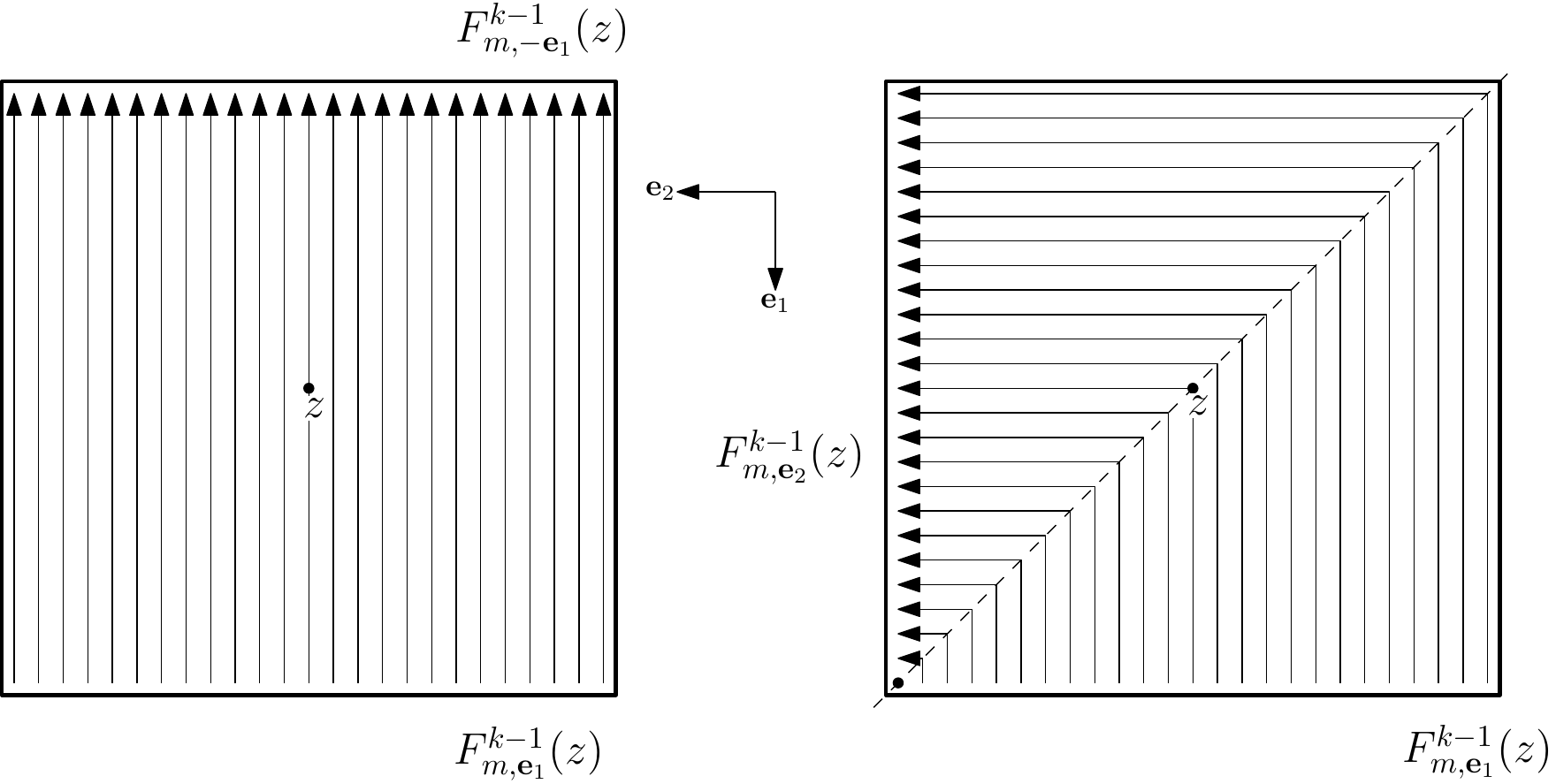}
		\vspace{0.1cm}
		\caption{A two-dimensional representation of the collection of paths constructed in Lemma~\ref{l:B17kminus1path}. }
		\label{f:B17flow}
	\end{figure}
	we consider in~$ \mathcal{T}({\bf e}_1,{\bf e}_2,z,k-1)$ the path of vertices of~$2L_{k-1}\Z^d$ which starts at $x_0 - L_{k-1} {\bf e}_1$, goes to the discrete hyperplane
	\[
	\left\{ (y_1,\dots,y_d)\in  2L_{k-1}\Z^d;\, y_1 = y_2           \right\}
	\]
	as a discrete straight line in~$2L_{k-1}\Z^d$ parallel to~${\bf e}_1$, reaching the point
	\[
	x' = a_2{\bf e}_1 +a_2 {\bf e}_2 +\dots + a_ d{\bf e}_d ,
	\]
	and then goes to
	\[
	x'' = a_2{\bf e}_1 + (17^{-1}L_k -L_{k-1}) {\bf e}_2 +\dots + a_ d{\bf e}_d  \in F^{k-1}_{m,{\bf e}_2}(z),
	\]
	as a discrete straight line in~$2L_{k-1}\Z^d$ parallel to~$-{\bf e}_2$. If~$x_0 - L_{k-1} {\bf e}_1=x'=x''$, we take the path to be simply comprised of one point~$\{ x' \}$. This collection of paths satisfies the required properties by elementary geometric considerations. We refer to Figure~\ref{f:B17flow}.
\end{proof}

\begin{figure}[ht]
  \centering
  \includegraphics[scale = .6]{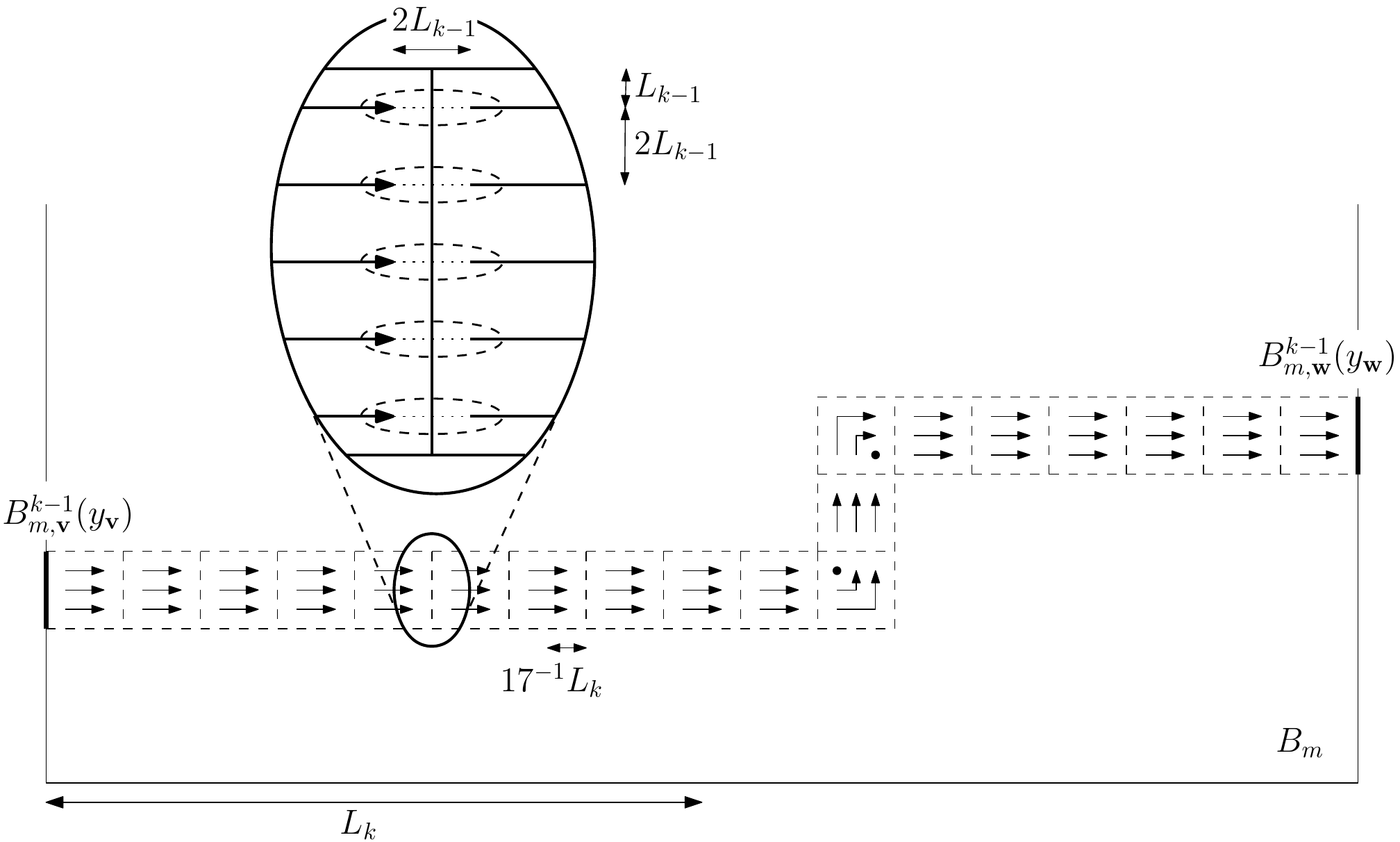}
  \vspace{0.1cm}
  \caption{A two-dimensional representation of the collection of paths constructed in Lemma~\ref{l:kminus1paths}. }
  \label{f:kminus1paths}
\end{figure}
We can now finally state and prove the main result of this section, which shows the existence of a collection of efficient paths in~$\mathsf{V}^{u,k-1}_\rho$ joining subsets of good faces of a good box of the $k$-th scale.

\begin{lemma}
  \label{l:kminus1paths}
  Consider~$m\in\mathbb{M}_{k}$, with~$k\in\N$, and ${\bf v},{\bf w} \in \mathsf{U}$, ${\bf v} \neq {\bf w}$.
  Assume the occurrence of the event where both~$\mathsf{F}_{m,{\bf v}}$ and~$\mathsf{F}_{m,{\bf w}}$ are good.
  Then, given~$y_{\bf v}\in\mathsf{G}_{m,{\bf v}}$ and~$y_{\bf w}\in\mathsf{G}_{m,{\bf w}}$, there exists a collection~$\mathcal{T}(y_{\bf v},y_{\bf w},k-1)$ of vertex-disjoint nearest-neighbor paths of~$B_m^{k-1}$ such that for every~$x_0\in B_{m,{\bf v}}^{k-1}(y_{\bf v})$ there exists $(z_0,z_1,\dots,z_n)\in \mathcal{T}(y_{\bf v},y_{\bf w},k-1)$ such that~$n\leq 4\cdot17^{d-1}L_k \cdot L_{k-1}^{-1}$, $x_0=z_0 + L_{k-1}{\bf v}$, $z_n + L_{k-1}{\bf w}$ is in~$B_{m,{\bf w}}^{k-1}(y_{\bf w})$, and $z_0,z_1,\dots,z_n\in\mathsf{V}^{u,k-1}_\rho$.
\end{lemma}

\begin{proof}
  Lemma~\ref{l:17path} implies the existence of a nearest neighbor path~$(z_1,\dots,z_l)$ of points in~$\mathcal{B}_m$ whose associated boxes of radius~$17^{-1}L_k$ share faces, such that the length~$l$ is smaller than~$17^d$ and such that $B_{m,{\bf v}}^{k-1}(y_{\bf v})$ and~$B_{m,{\bf w}}^{k-1}(y_{\bf w})$ are faces of
  \[
    B_\infty(z_1,17^{-1}L_k)\cap(2L_{k-1}\Z^d+L_{k-1}{\bf v})\text{  and  }B_\infty(z_l,17^{-1}L_k)\cap(2L_{k-1}\Z^d+L_{k-1}{\bf w})
  \]
  respectively. Lemma~\ref{l:B17kminus1path} provides a construction of paths between vertices of consecutive faces in this path of boxes. In order to construct~$\mathcal{T}(y_{\bf v},y_{\bf w},k-1)$ one connects the paths in these collections, connecting the end of one path in one of the boxes to the nearest starting point of a path in the next box. Here we note that adjacent boxes share a face and a set of the form $\mathsf{G}_{m', {\bf v}'}$ defined in~\ref{eq:intbmejdef}.
  Take e.g.\ $z_{i-1},z_{i},z_{i+1},z_{i+2}$ consecutive points in the nearest neighbor path of~$\mathcal{B}_m$. Given the unit vectors
  \[
    {\bf v}'=\frac{z_{i}-z_{i-1}}{|z_{i}-z_{i-1}|},\quad{\bf v}''=\frac{z_{i+1}-z_{i}}{|z_{i+1}-z_{i}|},  \quad\text{ and }\quad {\bf v}'''=\frac{z_{i+2}-z_{i+1}}{|z_{i+2}-z_{i+1}|}
  \]
  we connect the endpoint~$x$ of a path in~$\mathcal{T}({\bf v}',{\bf v}'',z_i,k-1)$ to the starting point $x+2L_{k-1}{\bf v}''$ of a path in~$\mathcal{T}({\bf  v}'',{\bf v}''',z_{i+1},k-1)$. Using the fact that $B_{m,{\bf v}}^{k-1}(y_{\bf v})=F^{k-1}_{m,{\bf v}}(z_1)$ and $B_{m,{\bf w}}^{k-1}(y_{\bf w})=F^{k-1}_{m,{\bf w}}(z_l)$, as well as the bound on the length of paths given by Lemma~\ref{l:B17kminus1path}, we finish the proof of the result.
  We refer to Figure~\ref{f:kminus1paths}.
\end{proof}

\section{Finite energy flows}
\label{s:flow}

In this section we will finally construct the discrete finite energy flows, using the groundwork and notation from the previous sections.
We start with the definition of $k$-fractals, which are hierarchical sets contained on good faces at the $k$-th scale, these sets avoid defects of all previous scales, and from them we will be able to construct finite energy flows in a hierarchical fashion.

\begin{definition}
  \label{def:0fractal}
  Given a $(u,\rho,0)$-good box~$m\in\mathbb{M}_0$ and a unit vector~${\bf v} \in \mathsf{U}$ assume the occurrence of the event where~$\mathsf{F}_{m,{\bf v}}$ is good. Then, given~$y\in\mathsf{G}_{m,{\bf v}}$, we say that~$\mathcal{F}(m,{\bf v},y)\equiv B_{m,{\bf v}}(y) \subset \Z^d$ is a~$0$-\emph{fractal}.
\end{definition}

\begin{definition}
	\label{def:kfractal}
	Given a $(u,\rho,k)$-good box~$m\in\mathbb{M}_k$, with~$k\in\N$, and a unit vector~${\bf v} \in \mathsf{U}$, assume the occurrence of the event where~$\mathsf{F}_{m,{\bf v}}$ is good. Then, given~$y\in\mathsf{G}_{m,{\bf v}}$, we say that~$\mathcal{F}(m,{\bf v},y)\subset\Z^d$ is a~$k$-\emph{fractal} if
	\begin{itemize}
		\item $\mathcal{F}(m,{\bf v},y)$ is contained in $B_{m,{\bf v}}(y)$;
		\item for every~$m'\in\mathbb{M}_{k-1}$ such that~$B_{m'}\subset B_m$ and~$\mathsf{F}_{m',{\bf v}}\subset B_{m,{\bf v}}(y)$, there exists~$y_{m'}\in \mathsf{G}_{m',{\bf v}}$ so that~$\mathcal{F}(m,{\bf v},y)$ is the union of the~$(k-1)$-fractals associated to such points~$y_{m'}$ and the unit vector~${\bf v}$, that is,
		\[
		\mathcal{F}(m,{\bf v},y) = \bigcup_{ \substack{m' \in \mathbb{M}_{k-1} ,\, B_{m'}\subset B_m\\ \mathsf{F}_{m',{\bf v}}\subset B_{m,{\bf v}}(y) }} \mathcal{F}(m',{\bf v},y_{m'}).
		\]
	\end{itemize}
	We note that for every~$m'\in\mathbb{M}_{k-1}$ such that~$\mathsf{F}_{m',{\bf v}}\subset B_{m,{\bf v}}(y)$ we have that~$\mathsf{F}_{m',{\bf v}}$ is good, by definition of the set~$\mathsf{G}_{m,{\bf v}}$.
\end{definition}

The next result shows that it is possible to construct a flow between sources and sinks supported on the~$k$-fractals of distinct good faces of a good box of the~$k$-th scale such that the flow's energy decays as a polynomial of~$L_k$.

\begin{proposition}
\label{p:kflowenergy}
Let~$L_0 > \clzero$, and
consider~$k\in\N$ and two $k$-fractals~$\mathcal{F}(m,{\bf v},y_{{\bf v}})$ and $\mathcal{F}(m,{\bf w},y_{{\bf w}})$ associated to the faces of a good box~$m$ of the $k$-th scale, these faces being in turn associated to two vectors ${\bf v}, {\bf w} \in \mathsf{U}$.
There exists a discrete flow~$\theta_{{\bf v},{\bf w}}^{m}$ on the edges of~$\mathsf{V}^u_\rho$ such that, for any~$J\in(0,1)$,
\begin{itemize}
	\item[(i)] $\displaystyle \mathrm{div}( \theta_{{\bf v},{\bf w}}^{m}  ) = \frac{1}{\left| \mathcal{F}(m,{\bf v},y_{{\bf v}}) \right|}{\bf 1}_{ \mathcal{F}(m,{\bf v},y_{{\bf v}}) }
	 -
	\frac{1}{\left| \mathcal{F}(m,{\bf w},y_{{\bf w}}) \right|}{\bf 1}_{ \mathcal{F}(m,{\bf w},y_{{\bf w}}) }$;
	\item[(ii)] $\displaystyle \theta_{{\bf v},{\bf w}}^{m}(e) \neq 0$ only when at least one of the endpoints of the edge~$e$ belongs to~$\mathrm{int}(B_m)$;
	\item[(ii)] $\displaystyle \mathrm{Energy}(\theta_{{\bf v},{\bf w}}^{m} )= \sum_{ \substack{ e \text{ an edge} \\ \text{from }\mathsf{V}^u_\rho }}\theta_{{\bf v},{\bf w}}^{m}(e)^2 \leq \cdot L_0^{3d}\cdot  L_k^{-J}$ .
\end{itemize}
\end{proposition}

\begin{proof}
	We prove the result by induction in~$k$. Assume first that~$k=0$, and consider a bijection~$\varphi^0$ between the vertices of~$B_{m,{\bf v}}(y_{\bf v})$ and~$B_{m,{\bf w}}(y_{\bf w})$. From Lemma~\ref{l:0path}, we know the existence of a directed path between~$y_{\bf v}$ and~$y_{\bf w}$ contained in~$\mathsf{V}^u_\rho\cap B_m$ which only intersects the faces of~$B_m$ at~$y_{\bf v}$ and~$y_{\bf w}$. Since
	\[
	\mathsf{P}_{m,{\bf v}}(y_{\bf v})\cup \mathsf{P}_{m,{\bf w}}(y_{\bf w}) \subset \mathsf{V}^u_\rho\cap B_m,
	\]
	there also exists such a discrete vacant path between~$y_{\bf v} - {\bf v}$ and~$y_{\bf w} - {\bf w}$, and therefore one can find a directed path~$\textbf{path}^0(x,\varphi_0(x))$ starting at~$x\in B_{m,{\bf v}}(y_{\bf v})$ and ending at~$\varphi_0(x)\in B_{m,{\bf w}}(y_{\bf w})$ which only intersects the faces of~$B_m$ at~$x$ and~$\varphi_0(x)$. For each such $x$ we construct the flow~$\theta_x$ which associates to each directed edge~$e$ in the nearest-neighbor graph of~$\mathsf{V}^u_\rho$ the value~$1$ if~$e$ is traversed by~$\textbf{path}^0(x,\varphi_0(x))$, $-1$ if~$-e$ is the edge being traversed, or~$0$ otherwise. We then define
	\begin{equation}
	\label{eq:flow0}
	\theta_{{\bf v},{\bf w}}^{m}:=\sum_{ x\in B_{m,{\bf v}}(y_{\bf v}) } \theta_x,
	\end{equation}
	and it is immediate that this flow satisfies item~$(i)$ of the Proposition.
	Since~$J < 1$, $| B_{m,{\bf v}}(y_{\bf v}) | \leq  cL_0^{7(d - 1)/10} $, and the maximal length of a path in~$B_m$ is smaller than~$L_0^d$, we obtain, after possibly increasing $\clzero$ and requiring $L_0 > \clzero$,
	\begin{equation}
	\label{eq:flow0_1}
	\mathrm{Energy}(\theta_{{\bf v},{\bf w}}^{m} ) \leq c \cdot L_0^{7(d - 1)/5}\cdot L_0^d \leq L_0^{3d}\cdot L_0^{-J} ,
	\end{equation}
	and the base case of induction is proved.

	Assume now that we already proved the result for~$k-1\in\N$, and let us prove it for~$k$. We know by Lemma~\ref{l:kminus1paths} that there exists a collection~$\mathcal{T}(y_{\bf v},y_{\bf w},k-1)$ of vertex-disjoint nearest-neighbor paths of~$B_m^{k-1}\cap\mathsf{V}^{u,k-1}_\rho$ of length at most~$4\cdot17^{d-1}L_k \cdot L_{k-1}^{-1}$, such that for each~$x\in B_{m,{\bf v}}^{k-1}(y_{\bf v})$, there exists a point~$\varphi_{k}(x)\in B_{m,{\bf w}}^{k-1}(y_{\bf w})$ and a path
	\[
	(z_1,\dots,z_{n_x})=\textbf{path}^{k}(x,\varphi_{k}(x))\in \mathcal{T}(y_{\bf v},y_{\bf w},k-1)
	\]
	starting at~$x-L_{k-1}{\bf v}=z_1$ and ending at~$\varphi_{k}(x)-L_{k-1}{\bf w}=z_{n_x}$. Since the associated boxes are all~$(u,\rho,k-1)$-good, the faces between the boxes associated to two adjacent vertices in this path must be good. For~$i=1,\dots,{n_x}$, we let~$m_i=(z_i,k-1)$. We know by the definition of the $k$-fractal that there must exist two~$(k-1)$-fractals
	\begin{equation}
	\begin{split}
	\mathcal{F}(z_1,{\bf v},y_{z_1})
	\subset
	F_{ m_1 ,  {\bf v}   } \qquad \text{and} \qquad
        \mathcal{F}(z_{n_x},{\bf w},y_{z_{n_x}})
	\subset
	F_{ m_{n_x} ,  {\bf w}   }
	\end{split}
	\end{equation}
	such that
	\begin{equation}
	\begin{split}
	\mathcal{F}(z_1,{\bf v},y_{z_1})
	\subset
	\mathcal{F}(m,{\bf v},y_{{\bf v}})
	\quad\text{ and }\quad
	\mathcal{F}(z_{n_x},{\bf w},y_{z_{n_x}})
	\subset
	\mathcal{F}(m,{\bf w},y_{{\bf w}}),
	\end{split}
	\end{equation}
	and by the goodness of the boxes~$m_1,\dots,m_{n_x}$, we know the existence of $(k-1)$-fractals
	\begin{equation}
	\begin{split}
	\mathcal{F}(z_i,{\bf v}_i,y_i)
	&\subset
	F_{ m_i ,  {\bf v}_i   }
	\end{split}
	\end{equation}
	contained in each face given by the intersection of two consecutive boxes~$B_{m_i}$ and~$B_{m_{i+1}}$. Using the induction hypothesis, we obtain~${n_x}-1$ flows~$\theta_1^x,\dots,\theta_{{n_x}-1}^x$ between $\mathcal{F}(z_i,{\bf v}_i,y_{i})$ and $\mathcal{F}(z_{i+1},{\bf v}_{i+1},y_{i+1})$, as well as flows~$\theta_0^x$ and~$\theta_{n_x}^x$, the first between~$\mathcal{F}(z_1,{\bf v},y_{z_1})$ and~$\mathcal{F}(z_1,{\bf v}_1,y_1)$, and the latter between~$\mathcal{F}(z_{n_x},{\bf v}_n,y_{{n_x}})$ and~$\mathcal{F}(z_{n_x},{\bf w},y_{z_{n_x}})$, each one these flows satisfying properties (i), (ii) and (iii). Letting then
	\begin{equation}
	\begin{split}
	\theta_x := \frac{\left|\mathcal{F}(z_1,{\bf v},y_{z_1})\right|}{\left| \mathcal{F}(m,{\bf v},y_{{\bf v}}) \right|}\sum_{i=0}^{n_x} \theta_i^x,
	\end{split}
	\end{equation}
	we can define
	\begin{equation}
	\label{eq:thetakdef}
	\begin{split}
	\theta_{{\bf v},{\bf w}}^{m} := \sum_{x\in B_{m,{\bf v}}^{k-1}(y_{\bf v})} \theta_x.
	\end{split}
	\end{equation}
	The flow~$\theta_{{\bf v},{\bf w}}^{m}$ automatically satisfies properties (i) an (ii). To verify property (iii), we first notice that the set of edges which each of the flows in the set $\{\theta^x;\, x\in B_{m,{\bf v}}^{k-1}(y_{\bf v})\}$ traverses are disjoint. Moreover, for a given~$x\in B_{m,{\bf v}}^{k-1}(y_{\bf v})$, the set of edges through which each of the flows in~$\{\theta_i^x;\, i=1,\dots,n\}$ passes is also disjoint. We also note that, by the definition of a $k$-fractal, $k$-fractals have always the same cardinality, and the ratio between the cardinalities of a $(k-1)$-fractal and a $k$-fractal is smaller than $(L_{k-1}/L_k)^{d-1}$. This implies, together with the induction hypothesis and the bound on the size of~$\textbf{path}^{k}(x,\varphi_{k}(x))$,
\begin{equation}
  \label{eq:energykbound}
  \begin{split}
    \mathrm{Energy}(\theta_{{\bf v},{\bf w}}^{m} )
    &\leq
    \sum_{x\in B_{m,{\bf v}}^{k-1}(y_{\bf v})} \mathrm{Energy}(\theta_x )
    \leq
    \left(\frac{L_{k-1}}{L_k}\right)^{2(d-1)} \!\!\!\!\!\!\!\! \sum_{x\in B_{m,{\bf v}}^{k-1}(y_{\bf v})}\sum_{i=0}^{n_x} \mathrm{Energy}(\theta^x_i )
    \\
    &\leq
    c
    \left(\frac{L_{k-1}}{L_k}\right)^{2(d-1)} 	\left(\frac{L_{k}}{L_{k-1}}\right)^{d-1} 		\left(\frac{L_{k}}{L_{k-1}}\right)  L_0^{3d} L_{k-1}^{-J}
    \leq
    c L_0^{3d} L_{k-1}^{-J}
    \left(\frac{L_{k-1}}{L_k}\right)^{d-2} 	.
  \end{split}
\end{equation}
This in turn implies, again after possible increasing $\clzero$,
\begin{equation}
  \label{eq:energykbound2}
  \begin{split}
    \mathrm{Energy}(\theta_{{\bf v},{\bf w}}^{m} ) L_0^{-3d} L_{k}^{J}
    \leq
    c \left(\frac{L_{k-1}}{L_k}\right)^{d-2-J} 	< 1,
  \end{split}
\end{equation}
finishing the proof of the induction, and, consequently, of the result.
\end{proof}

Finally, we use the above result in order to show the existence, with high probability for sufficiently large~$L_0$, of a flow of finite energy in~$\mathsf{V}^u_\rho$ from the origin to infinity. We define~$A_0^{u,\rho}$ to be the event where the discretized box of the $0$-th scale~$B_{(0,0)}\cap\Z^d$ containing the origin is contained in~$\mathsf{V}^u_\rho$. We recall the definition of~$u_k$ and~$\rho_k$ in~\eqref{eq:ukdef}. For~$k\geq 1$, we define~$A_k$ as the event where every box of the $(k-1)$-th scale contained in~$B_{(0,k)}$ is $(u_{k-1},\rho_{k-1},k-1)$-good. We also write
\[
\bar A:= A_0^{u,\rho}\cap \bigcap_{k=1}^\infty A_k.
\]
On the event~$\bar A$ we will construct the aforementioned flow in~$\mathsf{V}^{u}_{1}$.

The next lemma shows that this event has probability close to~$1$ for sufficiently large~$L_0$
and small~$u$.

\begin{lemma}
\label{l:barAprob}
With the notation introduced above, we have, for $L_0 > \clzero$ and $u \leq \tilde{u}$ defined in \eqref{eq:u0def},
\begin{equation}
\label{eq::barAprob}
\begin{split}
\IP\left(     \bar A^C \right) \leq \left(1-\exp\left\{  -cuL_0^{d-1}        \right\}\right) + c \exp\left\{  -(\log L_0)^{\delta/2}         \right\},
\end{split}
\end{equation}
and note that, by choosing $L_0$ sufficiently large, and then choosing~$u$ sufficiently small, we can make the above right hand side as small as we want.
\end{lemma}

\begin{proof}
  By the definition of the event~$\bar A$, Proposition~\ref{p:pkind}, monotonicity in $u$, the stationarity of the cylinder process under translations, and the union bound, we obtain
  \begin{equation}
    \label{eq::barAproof}
    \begin{split}
      \IP\left(     \bar A \right)
      &\leq
      \IP\left( (A_0^{u,\rho})^C   \right)+
      \sum_{i=1}^\infty
      \left( \frac{L_k}{L_{k-1}} \right)^d \exp\left\{  -(\log L_{k-1})^{1+\delta}         \right\}
    \end{split}
  \end{equation}
  Recalling that, by Lemma~$(2.2)$ of~\cite{TW10b}, the number of cylinders of radius~$1$ intersecting~$B_\infty(0,2 L_0)$ is Poisson distributed with parameter bounded from above by $c u L_0^{d-1}$, we obtain
  \begin{equation}
    \label{eq::barAproof2}
    \begin{split}
      \IP\left(     \bar A \right)
      &\leq
      \left(1-\exp\left\{  -cu L_0^{d-1}     \right\}\right)
       +
      \exp\left\{  -(\log L_0)^{\delta/2}         \right\}
      \sum_{i=1}^\infty
      \left( \frac{L_k}{L_{k-1}} \right)^d \exp\left\{  -(\log L_k)^{1+\delta/2}         \right\}
      \\
      &\leq
      \left(1-\exp\left\{  -cu  L_0^{d-1}      \right\}\right) + c \exp\left\{  -(\log L_0)^{\delta/2}         \right\},
    \end{split}
  \end{equation}
  finishing the proof of the result.
\end{proof}

Recall that~${\bf line}(x,y)$ denotes the closed line segment connecting points~$x,y\in\R^d$ to each other. In the following definitions, we assume the occurrence of the event~$\bar A$. In~$\bar A$, the boxes~$B_{(0,k)}$ and~$B_{(2L_k{\bf e}_1,k)}$ are all simultaneously $(u_k,\rho_k,k)$-good for every~$k\in\N$. We can therefore choose, for each~$k\in\N$, points $y_{k}^0\in\mathsf{G}_{(0,k),{\bf e}_1}$. Define then, for~$k\in\N$, the cone set of the~$k$-th scale
\begin{equation}
\label{eq:conedef}
\begin{split}
{\bf Cone}_k
&:=
\left\{
\begin{array}{c}
x \in 2L_k\Z^d\cap B_{(0,k+1)};\, \text{there exists a point }  z\in B_{(0,k+1),{\bf e}_1}(y_{k+1}^0)
\\
\text{such that } {\bf line}(2L_k{\bf e}_1,z)\cap B_{(x,k)} \neq \emptyset
\end{array}
\right\}.
\end{split}
\end{equation}
In~$\bar A$, we have that~${\bf Cone}_k\subset\mathsf{V}^{u_k,k}_{\rho_k}$, and we can consider in this discrete set a graph structure inherited from~$2 L_k \Z^d$. Moreover, we can consider the dual graph~${\bf Cone}_k^*$, whose vertex- and edge-set are respectively defined by
\begin{equation}
\label{eq:conedualdef}
\begin{split}
V({\bf Cone}_k^*)
&:=
\left\{
\begin{array}{c}
x + L_k{\bf v};\, x \in {\bf Cone}_k ,\, {\bf v}
\in \mathsf{U}
\end{array}
\right\},
\\
E({\bf Cone}_k^*)
&:=
\left\{
\begin{array}{c}
(x + L_k{\bf v},x + L_k{\bf w});\, x \in {\bf Cone}_k ,\, {\bf v}\neq{\bf w},\,{\bf v},{\bf w} \in \mathsf{U}
\end{array}
\right\}.
\end{split}
\end{equation}
The vertices of~${\bf Cone}_k^*$ can be identified with the faces of the boxes of the $k$-th scale with center in~${\bf Cone}_k$, and two faces are neighbors when they are the faces of the same box.
\begin{figure}[ht]
	\centering
	\includegraphics[scale = .6]{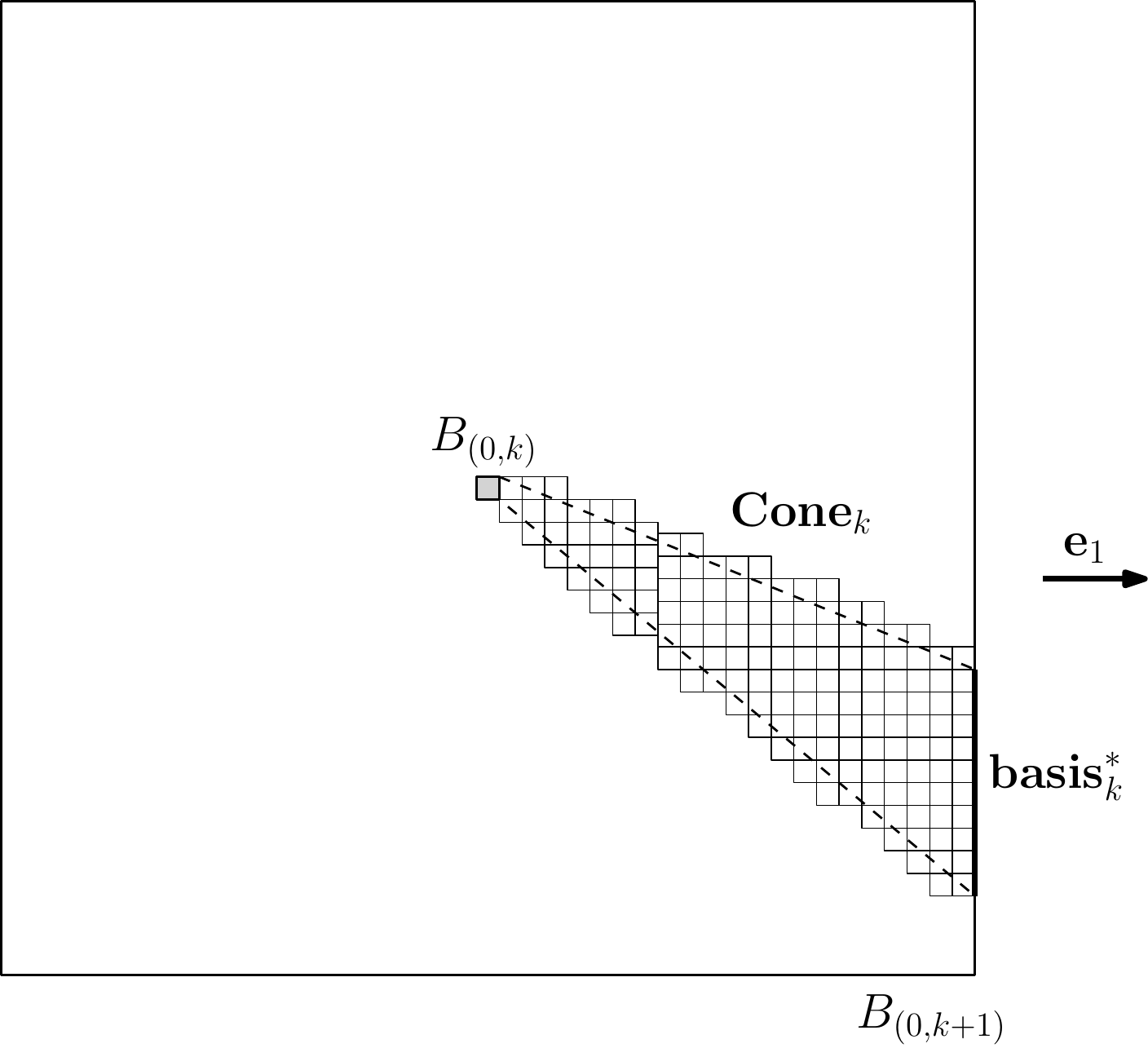}
	\vspace{0.1cm}
	\caption{Some of the sets involved in the construction of the flow~$\theta_{{\bf Cone}_k^*}$.}
	\label{f:conek}
\end{figure}

In order to simplify the notation, we denote the set $B_{(0, k + 1),{\bf e}_1}^k(y_{k+1}^0)  \subset{\bf Cone}_k^*$ by~${\bf basis}_k^*$. We construct a flow~$\theta_{{\bf Cone}_k^*}$ in~${\bf Cone}_k^*$ in the following manner:
\begin{enumerate}
	\item[(i)] Select a uniformly chosen random point~$Z\in {\bf basis}_k^* $;
	\item[(ii)] Consider the line segment~$ {\bf line}(L_k{\bf e}_1,Z)$, and choose in some predetermined arbitrary way a directed path~${\bf path}^*({\bf Cone}_k^*)$ in~${\bf Cone}_k^*$ starting at~$L_k{\bf e}_1$, ending at~$Z$, and minimizing~$\sup_{x\in {\bf path}^*({\bf Cone}_k^*)} \dist(x,  {\bf line}(L_k{\bf e}_1,Z))$;
	\item[(iii)] Let~$\theta_Z^{k,*}$ be the flow assigning~$1$ to a directed edge~$e$ if ${\bf path}^*({\bf Cone}_k^*)$ traverses~$e$, $-1$ if this path traverses~$-e$, and~$0$ otherwise;
	\item[(iv)] Define $\theta_{{\bf Cone}_k^*}(e)$ as $\IE[  \theta_Z^{k,*}(e)  ]$ for every edge~$e\in E({\bf Cone}_k^*)$, where the expectation is taken with respect to the random point~$Z$.
\end{enumerate}
The flow~$\theta_{{\bf Cone}_k^*}(e)$ will be part of the multi-scale construction of the finite-energy flow in~$\mathsf{V}^u_1$. For this construction, we will need the properties proved in the next lemma.

\begin{lemma}
\label{l:thetaCk}
In the event~$\bar A$, the flow~$\theta_{{\bf Cone}_k^*}$ above constructed has the following properties:
\begin{enumerate}
	\item[(i)] $\displaystyle \mathrm{div}( \theta_{{\bf Cone}_k^*}  )
	=
	{\bf 1}_{ L_k{\bf e}_1 }  -
	\frac{1}{\left| {\bf basis}_k^* \right|}{\bf 1}_{ {\bf basis}_k^* }$;
	\item[(ii)] There exists~$c>0$ such that, given~$x\in {\bf Cone}_k$ and an edge~$e_x\in E({\bf Cone}_k^*)$ between faces of~$x$,  $\displaystyle |\theta_{{\bf Cone}_k^*}(x)| \leq  \min\{ c L_k^{d-1}\langle x, {\bf e}_1   \rangle^{-(d-1)}, 1\}$.
\end{enumerate}
\end{lemma}

\begin{proof}
	To prove~$(i)$, we note that, conditioned on the random point~$Z\in{\bf basis}_k^*$, the flow  $\theta_Z^{k,*}$ is such that
	\[
	\mathrm{div}(\theta_Z^{k,*} )
	=
	{\bf 1}_{ L_k{\bf e}_1 }  -
	{\bf 1}_{Z}.
	\]
	By the linearity of the divergent, averaging the above equation over the possible values of~$Z$ yields property (i). Now, in order for~$\theta_Z^{k,*}(e_x)$ to be different from~$0$, we must have~${\bf line}(L_k{\bf e}_1,Z)\cap B(x,4L_k)\neq \emptyset$. Let~$\mathcal{Z}_x$ denote the set of vertices $z\in{\bf basis}_k^*$ such that~${\bf line}(L_k{\bf e}_1,z)\cap B(x,4L_k)\neq \emptyset$. Then, elementary trigonometry implies
	\begin{equation}
	\left(\frac{\langle x, {\bf e}_1   \rangle}{L_{k+1}} \right)^{d-1}\leq c \frac{L_{k}^{d-1}}{|\mathcal{Z}_x|},
	\end{equation}
	and therefore,
	\begin{equation}
	\IP(Z\in \mathcal{Z}_x) \leq c \min\left\{ \left(\frac{L_k}{\langle x, {\bf e}_1   \rangle}\right)^{d-1}, 1  \right\}
	\end{equation}
	We then obtain
	\begin{equation}
	\theta_{{\bf Cone}_k^*}(e_x) =\IE\left[ \theta_Z^{k,*}(e_x) \right] \leq  	\IP(Z\in \mathcal{Z}_x) \leq c \min\left\{ \left(\frac{L_k}{\langle x, {\bf e}_1   \rangle}\right)^{d-1}, 1  \right\},
	\end{equation}
	finishing the proof of the result.
\end{proof}

We can finally construct the promised finite energy flow. In the event~$\bar A$, for each~$k\in\N$, each~$x \in {\bf Cone}_k$, and each~${\bf v} \in \mathsf{U}$, we recall that the box~$(x,k)$ is $(u_k,\rho_k,k)$-good, that, by definition, the faces associated to points of~${\bf basis}_k^*$ are good, and therefore there exists a $k$-fractal~$\mathcal{F}((x,k),{\bf v},y_{(x,k),{\bf v}})$ contained in~$F_{(x,k),{\bf v}}$. Furthermore, since~$(u_k)_{k\geq 0}$ and $(\rho_k)_{k\geq 0}$ are increasing, and~$\rho_k\geq 1$, for sufficiently small~$u$ these fractals all exist simultaneously in~$\mathsf{V}^u_1$. We choose a sub-collection of these fractals requiring
\[
\mathcal{F}((x,k),{\bf v},y_{(x,k),{\bf v}}) = \mathcal{F}((x+2L_k{\bf v},k),-{\bf v},y_{(x+2L_k{\bf v},k),-{\bf v}})
\]
whenever the above equation is well defined. In other words, we ask that the fractals in a face shared by neighboring boxes agree.
We obtain the following result, which implies \ref{t:transient} by  the classical argument by Thompson,

\begin{theorem}
  \label{t:finiteenergyflow}
  There exists an event~$\bar A$ such that, for every~$\varepsilon>0$ there exists~$u>0$ such that~$\IP_u(\bar A)>1-\varepsilon$, and, in~$\bar A$, there exists a flow~$\theta$ of finite energy in~$\mathsf{V}^u_1$ from the origin to infinity, that is, such that~$\mathrm{div}(\theta)={\bf 1}_0$.
\end{theorem}

\begin{proof}
  Assume the occurrence of the event~$\bar A$ from Lemma~\ref{l:barAprob}. Consider the flows of the dual lattice~$(\theta_{{\bf Cone}_k^*})_{k\geq 0}$. For each~$k\in\N$, we will construct in~$\mathsf{V}^u_1$, with~$u\leq \inf_k u_k$, a flow~$\theta_k$ such that
  \begin{equation}
    \label{eq:fflow1}
    \begin{split}
      \lefteqn{\mathrm{div}( \theta_k  )
	=
	\frac{1}{\left| \mathcal{F}((0,k),{\bf e}_1,y_{(0,k),{\bf e}_1}) \right|}{\bf 1}_{ \mathcal{F}((0,k),{\bf e}_1,y_{(0,k),{\bf e}_1}) }}
      \phantom{****}
      \\ &\phantom{****}-
      \frac{1}{\left| \mathcal{F}((0,k+1),{\bf e}_1,y_{(0,k+1),{\bf e}_1}) \right|}{\bf 1}_{ \mathcal{F}((0,k+1),{\bf e}_1,y_{(0,k+1),{\bf e}_1}) },
    \end{split}
  \end{equation}
  that is, this flow has a source on a $k$-fractal on a face of~$B(0,2L_k)\cap \Z^d$ and a sink on a~$(k+1)$-fractal on a face of~$B(0,2L_{k+1})\cap \Z^d$. For every~$x\in{\bf Cone}_k$ and~${\bf v},{\bf w} \in \mathsf{U}$, we consider the flow~$\theta_{{\bf Cone}_k^*}(x)$ evaluated on the directed edge between~$x+L_k{\bf v}$ and~$x+L_k{\bf w}$, that is
  \[
    \theta_{{\bf Cone}_k^*}(x+L_k{\bf v},x+L_k{\bf w}).
  \]
  We also consider the flow~$\theta_{{\bf v},{\bf w}}^{(x,k)}$ on~$\mathsf{V}^{u}_{1}$ constructed in Proposition~\ref{p:kflowenergy}. We define then the flow in~$B(x, L_k)\cap \Z^d$:
  \begin{equation}
    \label{eq:fflow2}
    \begin{split}
      \theta_x^k
      :=
      \sum_{{\bf v},{\bf w}}\theta_{{\bf Cone}_k^*}(x+L_k{\bf v},x+L_k{\bf w})\theta_{{\bf v},{\bf w}}^{(x,k)}.
    \end{split}
  \end{equation}
  We can then define
  \begin{equation}
    \label{eq:fflow3}
    \begin{split}
      \theta_k
      :=
      \sum_{x\in {\bf Cone}_k}\theta_x^k,
    \end{split}
  \end{equation}
  and by Lemma~\ref{l:thetaCk} and Proposition~\ref{p:kflowenergy}, \eqref{eq:fflow1} holds. The same results also imply
  \begin{equation}
    \label{eq:fflow4}
    \begin{split}
      \mathrm{Energy}(\theta_k )
      &=
      \sum_{x\in {\bf Cone}_k}\mathrm{Energy}(\theta_x^k )
      \leq c L_0^{3d}\sum_{x\in {\bf Cone}_k}    \min\left\{ \left(\frac{L_k}{\langle x, {\bf e}_1   \rangle}\right)^{d-1}, 1  \right\}^2 L_k^{-2J}
      \\
      &\leq c L_0^{3d}L_k^{-2J}\sum_{j=1}^\infty  j^{d-1} j^{-2(d-1)}
      \leq c L_0^{3d}L_k^{-2J}.
    \end{split}
  \end{equation}
  Letting~$\theta_{\mathrm{origin}}$ denote a flow with finite support, with source at the origin, and sink uniformly distributed over~$ \mathcal{F}((0,0),{\bf e}_1,y_{(0,0),{\bf e}_1})$, we can define
  \begin{equation}
    \label{eq:fflow5}
    \theta:=\theta_{\mathrm{origin}}+\sum_{k\geq 0} \theta_k,
  \end{equation}
  which yields a finite energy flow with the required properties. The result follows after using Lemma~\eqref{l:barAprob}.
\end{proof}

\bibliographystyle{plain}
\bibliography{bib/all}

\end{document}